\documentclass[11pt,reqno]{amsart}
\usepackage{amssymb}
\usepackage{amsmath}
\usepackage{amsthm}
\usepackage[dvipsnames]{xcolor}
\usepackage{braket,comment,soul}
\usepackage{hyperref}
\usepackage{pdfpages,faktor}
\usepackage{graphicx,mathabx,bbm}
\usepackage{caption}
\usepackage{subcaption}
\usepackage{mathrsfs} 
\usepackage{tikz-cd} 
\usepackage{bm}
\usepackage{enumitem}
\usepackage{mathtools}
\usepackage{pgfplots}
\usepackage{import}
\usepackage{xifthen}
\usepackage{transparent}
\usepackage{xargs}                      
\usepackage[colorinlistoftodos,prependcaption,textsize=tiny]{todonotes}
\newcommandx{\unsure}[2][1=]{\todo[linecolor=red,backgroundcolor=red!25,bordercolor=red,#1]{#2}}
\newcommandx{\change}[2][1=]{\todo[linecolor=blue,backgroundcolor=blue!25,bordercolor=blue,#1]{#2}}
\newcommandx{\info}[2][1=]{\todo[linecolor=OliveGreen,backgroundcolor=OliveGreen!25,bordercolor=OliveGreen,#1]{#2}}
\newcommandx{\improvement}[2][1=]{\todo[linecolor=Plum,backgroundcolor=Plum!25,bordercolor=Plum,#1]{#2}}

\numberwithin{equation}{section}
\theoremstyle{plain}
\newtheorem{theorem}{Theorem}[section]

\theoremstyle{theorem}
\newtheorem{prop}[theorem]{Proposition}
\newtheorem{lem}[theorem]{Lemma}
\newtheorem{cor}[theorem]{Corollary}

\newtheorem{question}[theorem]{Question}
\newtheorem*{question*}{Question}

\theoremstyle{definition}
\newtheorem{defn}[theorem]{Definition}
\newtheorem{remark}[theorem]{Remark}

\newcommand{\R}{\mathbb{R}}
\newcommand{\T}{\mathbb{T}}
\newcommand{\C}{\mathbb{C}}

\newcommand{\Q}{\mathbb{Q}}
\newcommand{\Z}{\mathbb{Z}}

\newcommand{\N}{\mathbb{N}}
\newcommand{\D}{\mathbb{D}}
\newcommand{\E}{\mathscr{E}}

\newcommand{\cF}{\mathcal{F}}
\newcommand{\cN}{\mathcal{N}}
\newcommand{\cR}{\mathcal{R}}

\newcommand{\cK}{\mathcal{K}}
\newcommand{\cL}{\mathcal{L}}
\newcommand{\cD}{\mathcal{D}}
\newcommand{\cT}{\mathcal{T}}

\DeclareMathOperator{\Aut}{Aut}
\DeclareMathOperator{\Int}{Int}

\DeclareMathOperator{\cusp}{cusp}

\numberwithin{figure}{section}

\makeatletter
\DeclareFontFamily{U}{tipa}{}
\DeclareFontShape{U}{tipa}{m}{n}{<->tipa10}{}
\newcommand{\arc@char}{{\usefont{U}{tipa}{m}{n}\symbol{62}}}

\newcommand{\arc}[1]{\mathpalette\arc@arc{#1}}

\newcommand{\arc@arc}[2]{%
  \sbox0{$\m@th#1#2$}%
  \vbox{
    \hbox{\resizebox{\wd0}{\height}{\arc@char}}
    \nointerlineskip
    \box0
  }%
}
\makeatother

\date{\today}

\begin{document}

\title[Slices of Schwarz reflections, and B-involutions]{A general dynamical theory\\ of Schwarz reflections, B-involutions,\\ and algebraic correspondences}


\begin{author}[Y.~Luo]{Yusheng Luo}
\address{Department of Mathematics, Cornell University, 212 Garden Ave, Ithaca, NY 14853, USA}
\email{yusheng.s.luo@gmail.com}
\end{author}
\thanks{Y.L. was partially supported by NSF Grant DMS-2349929.}

\begin{author}[M.~Lyubich]{Mikhail Lyubich}
\address{Institute for Mathematical Sciences, Stony Brook University, 100 Nicolls Rd, Stony Brook, NY 11794-3660, USA}
\email{mlyubich@math.stonybrook.edu}
\end{author}
\thanks{M.L. was partially supported by NSF grants DMS-2247613 and 1901357, the Clay Fellowship, and the Centre for Nonlinear Analysis and Modeling (CNAM)}

\begin{author}[S.~Mukherjee]{Sabyasachi Mukherjee}
\address{School of Mathematics, Tata Institute of Fundamental Research, 1 Homi Bhabha Road, Mumbai 400005, India}
\email{sabya@math.tifr.res.in}
\thanks{S.M. was supported by the Department of Atomic Energy, Government of India, under project no.12-R\&D-TFR-5.01-0500, an endowment of the Infosys Foundation, and SERB research project grants MTR/2022/000248.}
\end{author}

\begin{abstract}
In this paper, we study matings of (anti-)polynomials and Fuchsian/reflection groups as Schwarz reflections/B-involutions or as (anti-)holomorphic correspondences, as well as their parameter spaces.
We prove the existence of matings of generic (anti-)polynomials, such as periodically repelling, or geometrically finite (anti-)polynomials, with circle maps arising from the corresponding groups.
These matings emerge naturally as degenerate (anti-)polynomial-like
maps, and we show that the corresponding parameter space slices for such matings bear strong resemblance with parameter spaces of polynomial maps.
Furthermore, we provide algebraic descriptions for these matings, and construct algebraic correspondences that combine generic (anti-)polynomials and genus zero orbifolds in a common dynamical plane, providing a new concrete evidence to Fatou’s vision of a unified theory of groups and~maps.
\end{abstract}

\maketitle

\setcounter{tocdepth}{1}
\tableofcontents

\section{Introduction}\label{intro_sec}

\bigskip

A dynamical theory of Schwarz reflection maps has emerged in the past decade
that provided a general framework for matings between anti-polynomial  maps and Fuchsian reflection groups as (anti-)algebraic correspondences, leading to relations between the corresponding parameter spaces (which also gave a new insight into the holomorphic case). Many special families have been thoroughly explored \cite{LLMM1,LLMM2,LLMM3,LMM2,LMMN,LMM23,MM3}. In this paper we will unify these various pieces into
a general theory.

A prototype of the theory we develop is the classical theory of polynomial-like maps.
Polynomial-like maps were first defined by Douady and Hubbard as a tool to study similarities between parameter spaces of complex polynomials and those of more general families of holomorphic maps \cite{DH2}. This led to great deal of work on the understanding of parameter spaces of complex dynamical systems. In the process, intimate connections between polynomial-like maps and renormalization theory emerged (see \cite{Sul1,McM1,Lyu99}). 

Inou and Kiwi developed a more general theory of polynomial-like maps and the associated straightening maps (that relate parameter spaces of holomorphic maps to appropriate standard models) which address the additional complexity furnished by higher degree renormalizations and critical orbit interactions \cite{IK}. 

However, the above theory has some limitations: over the years, various situations related to \emph{parabolic external dynamics} started to emerge which necessitated a more general set-up of pinched polynomial-like maps \cite{Mak93,BF05,Lom15,LLMM3,PR21,LMM23}. 

The presence of parabolic elements is also an important trait of all the groups that appear in our mating theory, which brings us to the general theory of \emph{pinched polynomial-like maps} (that generalizes a theory that has been recently developed in \cite{BLLM}).

The principal results of this paper can be summarized as follows.
\begin{enumerate}
\item\label{mateable_generic} We prove the existence of conformal matings between `generic' polynomials (with connected Julia set) and large classes of Fuchsian/reflection groups. Such matings are realized both as maps in one complex variable as well as algebraic correspondences (Theorems~\ref{conf_mating_polygonal_thm},~\ref{conf_mating_antiFarey_thm},~\ref{conf_mating_b_inv_thm},~\ref{all_corr_thm}).

\item We introduce the class of \emph{degenerate polynomial-like maps} that unifies the above mating constructions. This framework allows us
\begin{enumerate}
    \item to interpret the mating structure of algebraic correspondences in terms of the classical description of polynomial-like maps as matings of `hybrid classes' and `external maps' (Sections~\ref{dpl_maps_subsec},~\ref{poly_schwarz_sec},~\ref{schwarz_anti_farey_sec},~\ref{holo_para_space_reln_sec}), and

    \item to prove general bijection/homeomorphism statements between connectedness loci of polynomials and various parameter spaces of matings (Theorems~\ref{thm:CL}, \ref{thm:A},~\ref{thm:C}, \ref{farey_antipoly_thm}, \ref{b_inv_fbs_thm}).
\end{enumerate}
\end{enumerate}
In item~\eqref{mateable_generic} above, the term `generic' refers to the collection of at most finitely renormalizable, periodically repelling polynomials and geometrically finite polynomials.

Before elaborating on the main theorems, let us briefly mention some of the technical novelties involved in carrying out the tasks mentioned above. 

\begin{enumerate}
\item We use the structure of degenerate polynomial-like maps to adapt classical `puzzle piece' machinery for our setting. This plays a pivotal role in establishing various combinatorial continuity and rigidity results, which lie at the heart of our conformal mateability theorems (Sections~\ref{comb_class_schwarz_sec} and~\ref{fin_renorm_rigid_sec}).

\item The construction of algebraic correspondences (global objects) from one variable matings between polynomials and groups (i.e., certain degenerate polynomial-like maps that are semi-global objects) requires algebraic descriptions of such matings. This is straightforward for Schwarz reflection maps, which are known to be `rationally semi-conjugate' to reflection in the unit circle. To deal with the holomorphic situation, we introduce a class of meromorphic maps called \emph{B-involutions}, and give an explicit algebraic description to facilitate the task of globalization (see Section~\ref{b_inv_sec}, cf. \cite{BLLM}).

\item In order to manufacture algebraic correspondences that are matings between generic polynomials and groups with multiple parabolics, we need to define such correspondences on \emph{nodal spheres} or \emph{cacti of spheres}. This extends previously
known mating frameworks for correspondences; and in particular, gives first examples of correspondences associated with piecewise Schwarz reflection maps (Sections~\ref{polygonal_corr_sec},~\ref{corr_b_sigma_subsec}). Even for the simplest examples of piecewise Schwarz reflection maps (namely, the \emph{Circle-and-Cardioid} Schwarz reflections studied in \cite{LLMM1,LLMM2}), such a correspondence description is new.
\end{enumerate}

In Appendix~\ref{appendix}, we investigate canonical product structures in certain families of degenerate polynomial-like maps. This product structure is based on the classical idea of decomposing a polynomial-like map into a \emph{hybrid class} and an \emph{external map} \cite{DH2,Lyu99}. Loosely speaking, these families bind connectedness loci of polynomials (parameter space of hybrid classes) and Teichm{\"u}ller spaces of certain surfaces (parameter space of external maps) together in a dynamically natural fashion. A more detailed analysis of the boundary of such families of degenerate polynomial-like maps and the associated degeneration phenomena will be presented in a subsequent work.

\subsection{Antiholomorphic matings: Schwarz reflections}
A domain $\Omega$ in the Riemann sphere is called a \emph{quadrature domain} if it admits an anti-meromorphic map $S:\Omega\to\widehat{\C}$ that continuously extends to the identity map on the boundary $\partial\Omega$. Such a map  $S:\overline{\Omega}\to\widehat{\C}$ is called the \emph{Schwarz reflection map} of $\Omega$. A simply connected domain $\Omega\subsetneq\widehat{\C}$ is a quadrature domain if and only if its Riemann uniformization $\phi:\D\to\Omega$ extends as a rational map of the Riemann sphere. 

A \emph{quadrature multi-domain} $\cD$ is the union of finitely many disjoint simply connected domains $\{\Omega_j\}_{j=1}^k$ admitting a continuous map $S:\overline{\cD}\to\widehat{\C}$ which is anti-meromorphic on $\cD$ and which fixes $\partial\cD$ pointwise. Each component of a quadrature multi-domain is a quadrature domain. The map $S$, which we call a \emph{Schwarz reflection of a quadrature multi-domain}, acts as the Schwarz reflection of $\Omega_j$ on each $\overline{\Omega_j}$, for $j\in\{1,\cdots,k\}$.

The compact set $T(S):=\widehat{\C}\setminus\cD$ is called a \emph{droplet}. Its boundary consists of real-analytic curves possibly with finitely many cusp and double point singularities. Removing these finitely many singular points from the droplet yields the \emph{desingularized droplet} or the \emph{fundamental tile} of $S$. We denote the fundamental tile of $S$ by $T^0(S)$.

The dynamical plane of $S$ admits an invariant partition into the \emph{escaping/tiling set} $T^\infty(S)$ and the \emph{non-escaping set} $K(S)$. The former set  comprises all points that eventually escape to the fundamental tile $T^0(S)$, and the latter set is its complement (see Section~\ref{schwarz_sec} for precise definitions).
\begin{figure}[ht]
\captionsetup{width=0.96\linewidth}
\begin{center}
\includegraphics[width=0.5\linewidth]{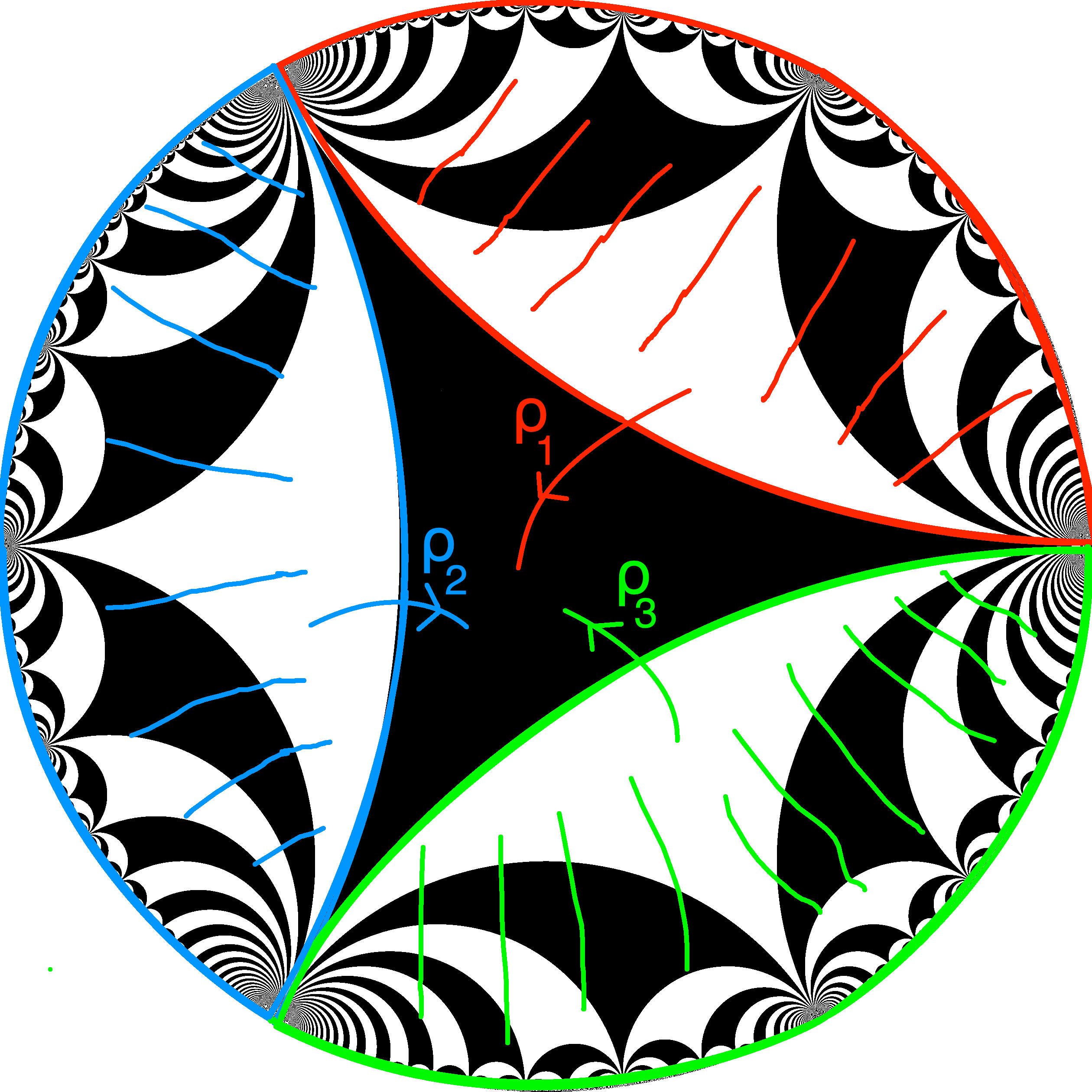}
\end{center}
\caption{The ideal triangle reflection group is generated by the reflections $\rho_1,\rho_2,\rho_3$ in the sides of an ideal hyperbolic triangle. These generators define the corresponding Nielsen map as a piecewise anti-M{\"o}bius map on the shaded regions. The action of the Nielsen map on $\mathbb{S}^1$ is topologically conjugate to $\overline{z}^2$.}
\label{itg_nielsen_fig}
\end{figure}

The exploration of Schwarz reflection dynamical systems was initiated in \cite{LLMM1,LLMM2}, and it was observed that in many cases, the Schwarz reflection map $S$ of a quadrature multi-domain is a \emph{combination/mating} of an anti-holomorphic polynomial (anti-polynomial for short) and an ideal polygon reflection group. Roughly, a Schwarz reflection $S$ is called a \emph{mating} of an anti-polynomial and  an ideal polygon reflection group if the dynamics of $S$ on its non-escaping set is topologically conjugate (conformally conjugate on the interior) to that of an anti-polynomial on its filled Julia set, while the action of $S$ on the tiling set is conjugate to a piecewise anti-M{\"o}bius map, called the \emph{Nielsen map}, that is orbit equivalent to the ideal polygon reflection group (see Definition~\ref{conf_mat_def_1} for the precise notion of conformal mating and Figure~\ref{itg_nielsen_fig} for an illustration of the Nielsen map of the ideal triangle reflection group). The Nielsen map of an ideal $(d+1)-$gon reflection group is also compatible with degree $d$ anti-polynomials as the action of the Nielsen map on the unit circle is topologically conjugate to the map $\overline{z}^d$.
In \cite{LLMM1,LLMM2}, the so-called \emph{Circle-and-Cardioid} (or C\&C) family of Schwarz reflections was studied in details, and it was shown that many maps in the C\&C family arise as matings of quadratic anti-polynomials and the ideal triangle reflection group.

In the following, we will discuss the general theory of two natural classes of Schwarz reflections:
\begin{enumerate}
\item polygonal Schwarz reflections; and
\item anti-Farey Schwarz reflections.
\end{enumerate}

\subsubsection{Polygonal Schwarz reflections}

We say that $\cD$ is a \emph{tree-like} quadrature multi-domain if $\overline{\cD}$ is connected and simply connected. In this case, each $\Omega_j$ is a Jordan domain, and their touching structure can be recorded by a plane tree (see Figure~\ref{tree_like_qmd_fig}).
A Schwarz reflection $S:\overline{\cD}\to\widehat{\C}$ associated with a tree-like quadrature multi-domain is \emph{polygonal} if the fundamental tile $T^0(S)$ is conformally equivalent to an ideal $(d+1)-$gon in the hyperbolic plane (where $d$ is the degree of the Schwarz reflection), and $S$ has no critical value in $T^0(S)$. The terminology `polygonal' is motivated by the fact that a polygonal Schwarz reflection is conformally conjugate to the Nielsen map of an ideal polygon reflection group near its droplet (see Section~\ref{poly_schwarz_sec} for more details).
We remark that Schwarz reflections in the C\&C family are prototypical examples of polygonal Schwarz reflections. 
\begin{figure}[ht]
\captionsetup{width=0.96\linewidth}
\begin{tikzpicture}
\node[anchor=south west,inner sep=0] at (0,0) {\includegraphics[width=0.8\textwidth]{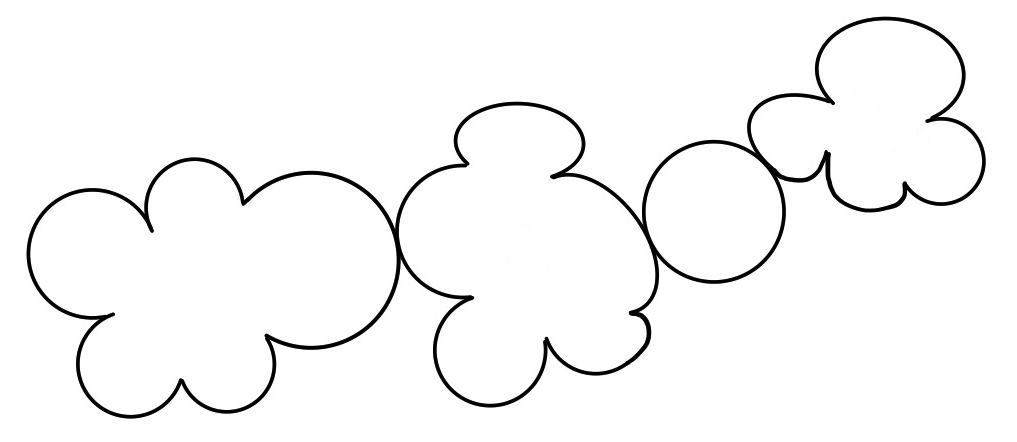}}; 
\node[anchor=south west,inner sep=0,rotate=120] at (6.6,3.888) {\includegraphics[width=0.12\textwidth]{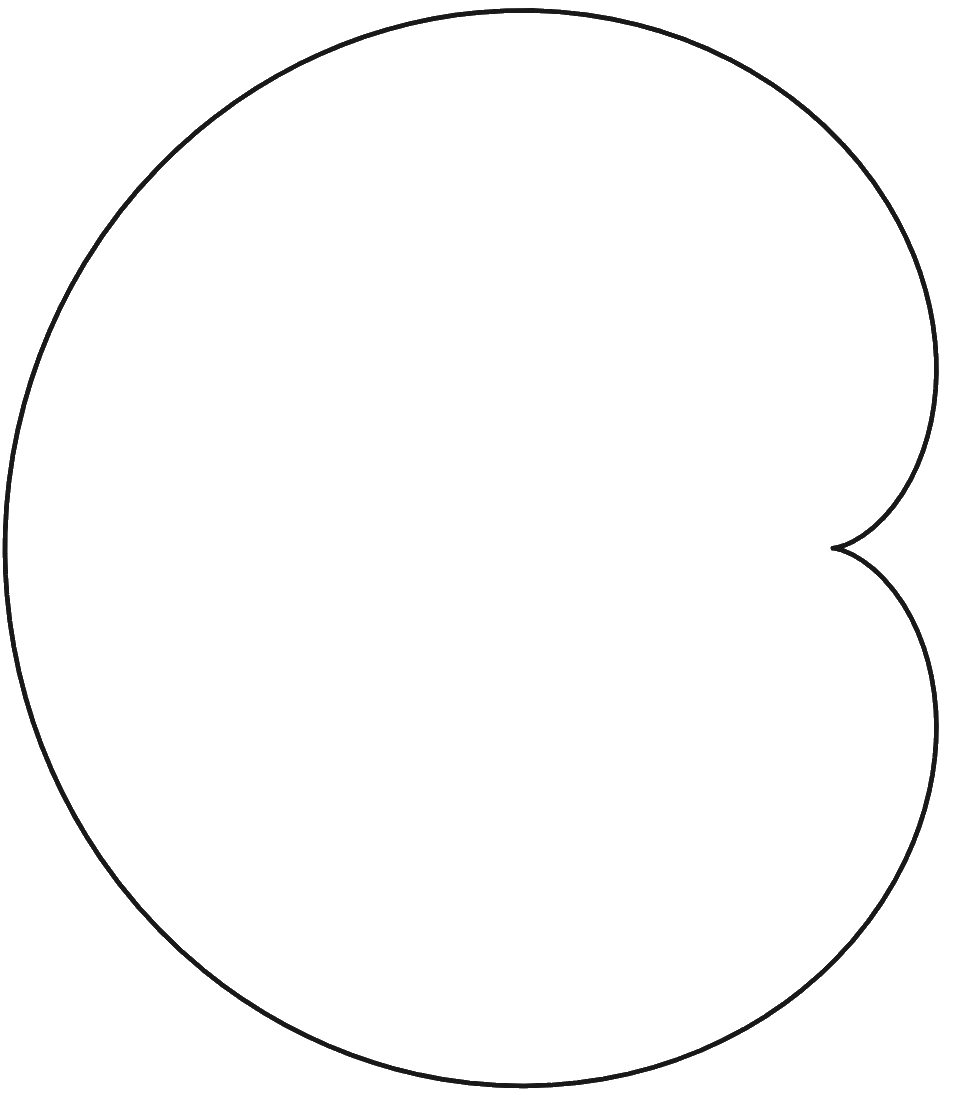}};
\node[anchor=south west,inner sep=0,rotate=120] at (6.6,3.89) {\includegraphics[width=0.12\textwidth]{postpinch_2.png}};
\node[anchor=south west,inner sep=0,rotate=120] at (6.604,3.898) {\includegraphics[width=0.12\textwidth]{postpinch_2.png}};
\node[anchor=south west,inner sep=0,rotate=-40] at (6.3,0.5) {\includegraphics[width=0.16\textwidth]{postpinch_2.png}};
\node[anchor=south west,inner sep=0,rotate=-40] at (6.29,0.5) {\includegraphics[width=0.16\textwidth]{postpinch_2.png}};
\node[anchor=south west,inner sep=0,rotate=-40] at (6.28,0.5) {\includegraphics[width=0.16\textwidth]{postpinch_2.png}};
\node[anchor=south west,inner sep=0,rotate=-40] at (6.304,0.5) {\includegraphics[width=0.16\textwidth]{postpinch_2.png}};
\node[anchor=south west,inner sep=0,rotate=-40] at (6.304,0.5) {\includegraphics[width=0.16\textwidth]{postpinch_2.png}};
\node at (7.75,0.6) {$\Omega_5$};
\node at (5.6,4.2) {$\Omega_6$};
\node at (5,2) {$\Omega_2$};
\node at (1.8,1.6) {$\Omega_1$};
\node at (7.2,2.4) {$\Omega_3$};
\node at (8.8,3.2) {$\Omega_4$};
\end{tikzpicture}
\caption{A tree-like quadrature multi-domain.}
\label{tree_like_qmd_fig}
\end{figure}

As in the degree two case, a degree $d$ polygonal Schwarz reflection map $S$ has a connected non-escaping set if and only if no critical point of $S$ escapes to the fundamental tile. This property is also equivalent to the fact that the tiling set dynamics of $S$ is conformally conjugate to the Nielsen map of an ideal $(d+1)-$gon reflection group.

We denote the connectedness locus of (normalized) degree $d$ polygonal Schwarz reflections by $\mathscr{S}_{\pmb{\cN}_d}$. Similarly, we denote by $\mathcal{C}^-_d$ the connectedness locus of monic and centered anti-polynomials of degree $d$. We say that two maps $f, g\in\mathcal{C}^-_d$ (respectively, $f, g\in \mathscr{S}_{\pmb{\cN}_d}$) are combinatorially equivalent, denoted by $f\sim g$, if they have the same rational laminations (in analogy with polynomial dynamics, the \emph{rational lamination} of a polygonal Schwarz reflection captures the co-landing patterns of its pre-periodic dynamical rays, see Definition~\ref{def_preper_lami}).

Let $\mathcal{C}^-_{d, r} \subset \mathcal{C}^-_d$ be the subset of periodically repelling anti-polynomials. For polygonal Schwarz reflections, the set $\mathfrak{S}$ of singular (fixed) points on the droplet boundary is special as these points always have parabolic dynamics. Hence, we define the analogous subspace $\mathscr{S}_{\pmb{\cN}_d, r}$ for polygonal Schwarz reflections as the collection of those maps $S$ in $\mathscr{S}_{\pmb{\cN}_d}$ for which no point of $\mathfrak{S}$ has an attracting direction in $K(S)$ and each periodic point of $S$ not in $\mathfrak{S}$ is repelling. Finally, let $\widehat{\mathcal{C}^-_{d, r}} = \mathcal{C}^-_{d, r}/\sim$ (respectively, $\widehat{\mathscr{S}_{\pmb{\cN}_d, r}} = \mathscr{S}_{\pmb{\cN}_d, r}/\sim$) be the space of corresponding combinatorial classes.

The following theorems generalize and sharpen the main results of \cite{LLMM2}. We remark that in \cite{LLMM2}, only the space of geometrically finite degree two polygonal Schwarz reflections were studied.

\begin{theorem}\label{thm:CL}
There is a dynamically natural homeomorphism 
$$
\Phi: \widehat{\mathcal{C}^-_{d, r}} \longrightarrow \widehat{\mathscr{S}_{\pmb{\cN}_d, r}}.
$$
Here `dynamically natural' means that the circle homeomorphism conjugating the Nielsen map to the power map $\overline{z}^d$ carries the rational lamination of a Schwarz reflection to the rational lamination of an anti-polynomial.
\end{theorem}

\noindent (See Section~\ref{comb_homeo_sec} for the proof of Theorem~\ref{thm:CL}.)

Let $\mathscr{S}_{\pmb{\cN}_d, fr} \subset \mathscr{S}_{\pmb{\cN}_d, r}$ be the subset of at most finitely renormalizable periodically repelling Schwarz reflections, and let $\mathcal{C}^-_{d, fr} \subset \mathcal{C}^-_{d,r}$ be its counterpart for anti-polynomials.
Adapting similar methods as in the polynomial setting, we show that each map in $\mathscr{S}_{\pmb{\cN}_d, fr}$ (or $\mathcal{C}^-_{d, fr}$) is \emph{combinatorially rigid}; i.e., its combinatorial class consists of a single map.
As an immediate corollary of this rigidity, we have the following result, which is proved in Section~\ref{fin_renorm_rigid_sec}.
\begin{theorem}\label{thm:A}
There is a dynamically natural homeomorphism 
$$
\Phi: \mathcal{C}^-_{d, fr} \longrightarrow \mathscr{S}_{\pmb{\cN}_d, fr}.
$$
Here dynamically natural means that for any $f\in \mathcal{C}^-_{d, fr}$, the Schwarz reflection $\Phi(f)$ is the unique conformal mating of $f$ and the Nielsen map of the ideal $(d+1)-$gon reflection group.
\end{theorem}

It is natural to ask if the above natural homeomorphism can be extended to a homeomorphism between $\mathcal{C}^-_d$ and $\mathscr{S}_{\pmb{\cN}_d}$.
The next result shows that it is not the case.
\begin{theorem}\label{thm:C}
There is a dynamically natural bijection 
$$
\Phi: \mathcal{C}^-_{d, gf} \longrightarrow \mathscr{S}_{\pmb{\cN}_d, gf}
$$ 
between geometrically finite anti-polynomials and polygonal Schwarz reflections with connected Julia/limit set, such that for any $f\in \mathcal{C}^-_{d, gf}$, the Schwarz reflection $\Phi(f)$ is the unique conformal mating of $f$ and the Nielsen map of the ideal $(d+1)-$gon reflection group.
However, both $\Phi$ and $\Phi^{-1}$ are discontinuous.
\end{theorem}

The existence of the bijection $\Phi$ appearing in Theorem~\ref{thm:C} is established in Section~\ref{geom_fin_polygonal_schwarz_sec}, and the discontinuity of $\Phi, \Phi^{-1}$ is proved in Section~\ref{discont_sec}.

The above theorem is analogous to a famous theorem of Kerckhoff-Thurston which states that the natural identification between two different Bers embeddings fails to extend to a continuous map on their boundaries.

It is worth emphasizing that the map $\Phi$ that relates the connectedness loci $\mathcal{C}^-_d$ and $\mathscr{S}_{\pmb{\cN}_d}$ (appearing in the previous two theorems) replaces the external dynamics of an anti-polynomial $f$ (i.e., the dynamics of $f$ on its basin of infinity) with the action of the Nielsen map of an ideal $(d+1)-$gon reflection group, or equivalently, mates the filled Julia set dynamics of $f$ with the dynamics of the Nielsen map. The following implication of Theorems~\ref{thm:CL},~\ref{thm:A} and~\ref{thm:C} is important in its own right (see Definition~\ref{comb_mating_def} for the definition of combinatorial mating).

\begin{theorem}\label{conf_mating_polygonal_thm}
Let $f$ be a degree $d$ anti-polynomial with connected Julia set.
Suppose that $f$ is either 
\begin{itemize}
\item geometrically finite; or 
\item periodically repelling, finitely renormalizable.
\end{itemize}
Then, $f$ can be conformally mated with the Nielsen map of an ideal $(d+1)-$gon reflection~group, and the resulting Schwarz reflection map is unique up to M{\"o}bius conjugacy.

Further, each periodically repelling map in $\mathcal{C}_d^-$ can be combinatorially mated with the Nielsen map of an ideal $(d+1)-$gon reflection~group.
\end{theorem}

We remark that notions of mating without assuming local connectivity of Julia sets also appeared \cite{Dud11}.

\subsubsection{Anti-Farey Schwarz reflections}
In the antiholomorphic world, there is another important expansive piecewise real-analytic circle endomorphism, called the \emph{anti-Farey map}, that is closely related to reflection groups. The anti-Farey map, which was first formally introduced in \cite{LMM23} (and had its precursors in \cite{LLMM3}), is a factor of the Nielsen map of a \emph{regular} ideal polygon reflection group (see Section~\ref{nielsen_farey_sec}). Schwarz reflection maps that admit the anti-Farey map as the conformal model of their escaping set dynamics will be referred to as \emph{anti-Farey Schwarz reflections}.

The proofs of Theorems~\ref{thm:CL}--\ref{thm:C} can be adapted for the family of anti-Farey Schwarz reflections, demonstrating close parameter space relations between the space of anti-Farey Schwarz reflections and the connectedness loci of anti-polynomials. We refer the reader to Theorem~\ref{farey_antipoly_thm} for the precise statements. As a consequence of these results, we have the following generalization of \cite[Theorem~C]{LMM23}.
\begin{theorem}\label{conf_mating_antiFarey_thm}
Let $f$ be a degree $d$ anti-polynomial with connected Julia set.
Suppose that $f$ is either 
\begin{itemize}
\item geometrically finite; or 
\item periodically repelling, finitely renormalizable.
\end{itemize}
Then, $f$ can be conformally mated with the degree $d$ anti-Farey map, and the resulting Schwarz reflection map is unique up to M{\"o}bius conjugacy.

Further, each periodically repelling map in $\mathcal{C}_d^-$ can be combinatorially mated with the degree $d$ anti-Farey map.
\end{theorem}

\subsection{Holomorphic matings: B-involutions}
We now pass to the holomorphic analog of Schwarz reflection maps that we will study in this paper. 
In order to develop a mating framework for complex polynomials and Fuchsian groups, we define a suitable class of holomorphic maps which we call \emph{$B-$involutions}. 
Special cases of B-involutions were introduced in \cite{BLLM} in the context of mating parabolic rational maps with Hecke groups (cf. Section~\ref{hyp_vs_para_subsubsec}).

An \emph{inversive multi-domain} $\cD$ is the union of finitely many disjoint simply connected domains $\{\Omega_j\}_{j=1}^k$ admitting a continuous map $S:\overline{\cD}\to\widehat{\C}$ which is meromorphic on $\cD$ and which restricts to an orientation-reversing order two self-homeomorphism of the boundary $\partial\cD$. The map $S$ is called a \emph{B-involution of an inversive multi-domain} (see Section~\ref{b_inv_sec} for a precise definition). 
We use a conformal welding argument to give an algebraic characterization of inversive multi-domains and B-involutions in terms of uniformizing rational maps. This parallels the characterization of quadrature domains as univalent images of disks under global rational maps (see Theorem~\ref{b_inv_thm}).


To study the dynamics of a B-involution of an inversive multi-domain, we partition its dynamical plane into \emph{escaping} and \emph{non-escaping} sets (see Section~\ref{b_inv_sec}). The action of a B-involution of an inversive multi-domain on its escaping set will be referred to as its \emph{external dynamics}.

Natural examples of such B-involutions are provided by matings of complex polynomials with connected Julia set and certain orientation-preserving external maps, called \emph{factor Bowen-Series maps} (see Theorem~\ref{fbs_poly_conf_mat_thm} and Proposition~\ref{mating_always_b_inv_prop}). Factor Bowen-Series maps, which first appeared in \cite{MM2}, can be thought of as holomorphic cousins of Nielsen or anti-Farey maps. These maps are semi-conjugates of classical Bowen-Series maps associated with certain genus zero orbifolds (more precisely, with the Fuchsian models of such orbifolds); in particular, they are piecewise analytic, expansive, circle coverings of degree greater than one. A formal description of factor Bowen-Series maps is given in Section~\ref{fbs_sec}.

Our analysis shows that spaces of B-involutions having factor Bowen-Series maps as the conformal model of their external dynamics are closely related to the connectedness loci of complex polynomials (see Theorem~\ref{b_inv_fbs_thm}). As a consequence of  Theorem~\ref{b_inv_fbs_thm}, we have an analog of Theorem~\ref{conf_mating_polygonal_thm} that confirms the existence of conformal matings of generic complex polynomials (with connected Julia set) and appropriate factor Bowen-Series maps.

\begin{theorem}\label{conf_mating_b_inv_thm}
Let $f$ be a degree $d$ polynomial with connected Julia set.
Suppose that $f$ is either 
\begin{itemize}
\item geometrically finite; or 
\item periodically repelling, finitely renormalizable.
\end{itemize}
Then, $f$ can be conformally mated with any degree $d$ factor Bowen-Series map, and the resulting B-involution is unique up to M{\"o}bius conjugacy. 

Further, each periodically repelling degree $d$ polynomial with connected Julia set can be combinatorially mated with any degree $d$ factor Bowen-Series~map.
\end{theorem}

\subsection{Matings as degenerate (anti-)polynomial-like maps}\label{dpl_maps_subsec}

To provide a common language for all the mating results stated above, we now put forward an appropriate generalization of polynomial-like maps, called \emph{degenerate (anti-)polynomial-like maps}. Pinched versions of polynomial-like maps were introduced in \cite{Mak93,BF05}, and more general notions (still not sufficient for the purpose of the current paper) appeared in the papers mentioned above.

A \textit{polygon} is a closed Jordan disk in $\widehat{\C}$ with a piecewise smooth boundary. The points where two smooth local boundary arcs meet are called the \emph{corners} of the polygon (note that cusps are also considered to be corners).

A \textit{pinched polygon} is a set in $\widehat{\C}$ which is homeomorphic to a closed disk quotiented by a finite geodesic lamination, and which has a piecewise smooth boundary. The cut-points of a pinched polygon will be called its \emph{pinched points}, and the non-cut points where two smooth local boundary arcs meet are called the \emph{corners} of the pinched polygon. 

\begin{defn}\label{degenerate_poly_def}
Let $P_1, P_2\subset \widehat{\C}$ be pinched polygons such that 
\begin{enumerate}
\item $P_1\subset P_2$,
\item the pinched points (respectively, corners) of $P_2$ are also pinched points (respectively, corners) of $P_1$, and
\item $\partial P_1\cap \partial P_2$ consists of the corners and pinched points of $P_2$.
\end{enumerate}
Suppose that there is a degree $d$ continuous map $g:P_1 \to P_2$ which is a proper holomorphic/anti-holomorphic map from each component of $\Int{P_1}$ onto some component of $\Int{P_2}$, and such that the singular points of $\partial P_1$ are the $g$-preimages of the singular points of $\partial P_2$.

We then call the triple $(g, P_1, P_2)$ a \emph{degenerate (anti-)polynomial-like map} of degree $d$.
\end{defn}
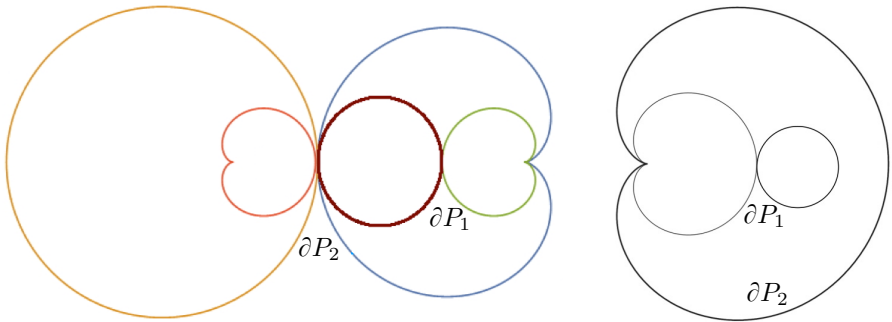
\begin{figure}[ht]
\captionsetup{width=0.96\linewidth}
\begin{tikzpicture}
\node[anchor=south west,inner sep=0] at (0,0) {\includegraphics[width=0.6\textwidth]{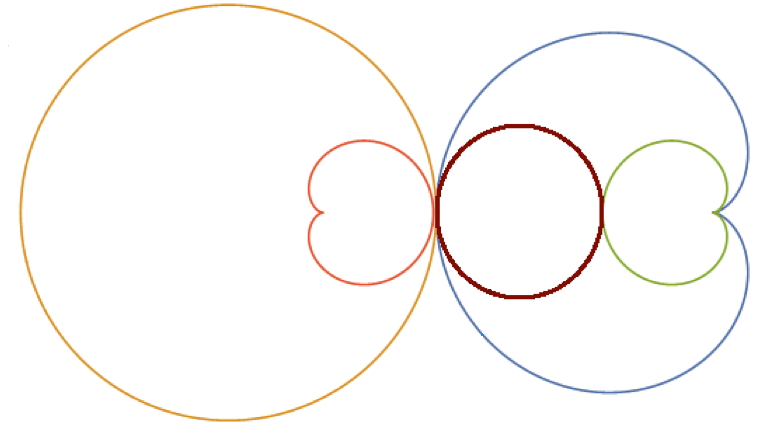}}; 
\node at (6.1,1.4) {\begin{small}$\partial P_1$\end{small}};
\node at (4.36,1) {\begin{small}$\partial P_2$\end{small}};
\node[anchor=south west,inner sep=0] at (8,0) {\includegraphics[width=0.32\textwidth]{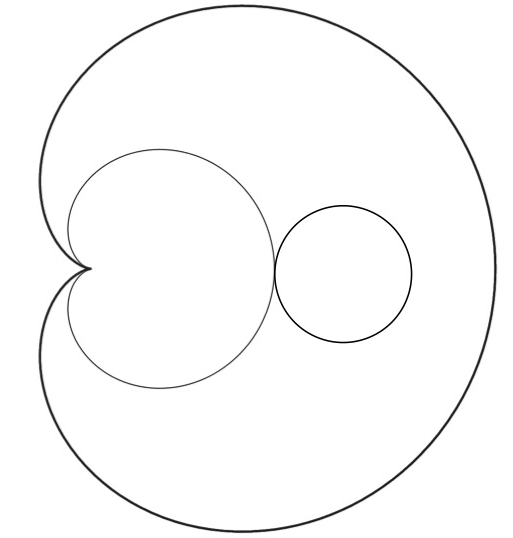}}; 
\node at (10.28,1.44) {\begin{small}$\partial P_1$\end{small}};
\node at (10.32,0.4) {\begin{small}$\partial P_2$\end{small}};
\end{tikzpicture}
\caption{Left: Pictured is the domain and codomain of a degenerate anti-polynomial-like map induced by a polygonal Schwarz reflection map. Here, $P_1$ (respectively, $P_2$) is a pinched polygon with two pinched points and two corners (respectively, one pinched point and one corner). Right: The domain and codomain of a degenerate (anti-)polynomial-like map induced by a Schwarz reflection with the anti-Farey external map are shown.}
\label{degenerate_poly_like_fig}
\end{figure}

As with polynomial-like maps, we define the \textit{filled Julia set} or \textit{non-escaping set} of a degenerate (anti-)polynomial-like map to be 
$$
K(g)=\bigcap_{n\geq 0} g^{-n}( P_1).
$$ 
Analogous to classical polynomial-like maps, the filled Julia set $K(g)$ of a degenerate (anti-)polynomial-like map is connected if and only if it contains all of the critical values of $g$.

For a degree $d$ degenerate (anti-)polynomial-like map $(g, P_1, P_2)$ with connected filled Julia set, one can define the notion of an \emph{external map} following the classical construction of external maps of polynomial-like maps (cf. \cite[Chapter I.1]{DH2}). Note that the external map of such a $g$ is a piecewise real-analytic, degree $d$, covering map of $\mathbb{S}^1$. Since the fundamental domain $P_2\setminus\Int{P_1}$ of $g$ is a pinched annulus, the external map of $g$ is in general not expanding on $\mathbb{S}^1$. This causes significant difficulty in the study of degenerate polynomial-like maps.
In fact, the external maps of all degenerate polynomial-like maps that we will encounter in this paper have parabolic cycles and hence are not expanding. However, they enjoy the weaker property of being expansive on $\mathbb{S}^1$.

Polygonal or anti-Farey Schwarz reflections and B-involutions are some natural examples of degenerate (anti-)polynomial-like maps.
For concreteness, consider $S:\overline{\cD}\to\widehat{\C}$ as a polygonal Schwarz reflection map.
The restriction $S:\overline{S^{-1}(\cD)}\to\overline{\cD}$ is a finite degree proper anti-holomorphic map from a pinched polygon to a larger pinched polygon (see Figure~\ref{degenerate_poly_like_fig}). We note that $\overline{S^{-1}(\cD)}$ is not compactly contained in $\overline{\cD}$. Thus, the map $S:\overline{S^{-1}(\cD)}\to\overline{\cD}$ is a degenerate anti-polynomial-like maps as in Definition \ref{degenerate_poly_def}.

The description of $S$ as a mating of an anti-polynomial and an ideal polygon reflection group can be recast in the classical language of polynomial-like maps: the \emph{internal class} of $S:\overline{S^{-1}(\cD)}\to\overline{\cD}$ is that of a degree $d$ anti-polynomial with connected Julia set and its \emph{external class} is represented by the Nielsen map of an ideal $(d+1)-$gon reflection group.
The parameter space homeomorphisms appearing in the theorems stated in the previous subsections can be regarded as \emph{straightening theorems} between slices of degenerate (anti-)polynomial-like maps and connectedness loci of (anti-)polynomials.

\subsection{Matings as (anti-)holomorphic correspondences}

The first obstruction to mating a polynomial map with a reflection/Fuchsian group is the fundamental mismatch between invertible and non-invertible dynamics. 
The passage from a group to its Nielsen map, anti-Farey map or factor Bowen-Series map (i.e., reduction from group dynamics to semi-group dynamics) allowed us to address this discord and combine the two dynamical systems into a single plane. However, as the conformal mating is a non-invertible map, its dynamics does not showcase the whole structure of the underlying group.

It was originally proposed by Fatou in the 1920s that the category of dynamical systems which truly encompasses rational dynamics as well as actions of Kleinian group is given by \emph{algebraic correspondences} \cite{Fatou29}. The first precise formulation of this proposal was given by Bullett and Penrose \cite{BP}, who studied certain algebraic correspondences of bi-degree $2$:$2$ on the Riemann sphere that combine the actions of quadratic polynomials with the modular group in a precise sense. 

\begin{defn}
Let $f$ be a degree $d$ (anti-)polynomial with connected Julia set and let $G$ be a discrete subgroup of (anti-)conformal automorphism group of the unit disk $\D$.
Let $\mathfrak{C}$ be an (anti-)holomorphic correspondence on a compact, simply connected (possibly noded) Riemann surface $\mathfrak{W}$.

We say $\mathfrak{C}$ a {\em mating} of $f$ and $G$ if there is a $\mathfrak{C}-$invariant partition $\mathfrak{W}=\cT\sqcup\cK$ such that the following hold.
\begin{enumerate}
	\item On $\cT$, the dynamics of $\mathfrak{C}$ is equivalent to the action of a group of (anti-)conformal automorphisms acting properly discontinuously. Further, $\cT/\mathfrak{C}$ is biholomorphic to $\D/G$.
	
	\item $\cK$ can be written as the union of two copies $\cK_1, \cK_2$ of $\cK(f)$ (where $\cK(f)$ is the filled Julia set of $f$), such that $\cK_1$ and $\cK_2$ intersect in finitely many points. Furthermore, $\mathfrak{C}$ has a forward (respectively, backward) branch carrying $\cK_1$ (respectively, $\cK_2$) onto itself with degree $d$, and this branch is conformally (respectively, anti-conformally) conjugate to $f\vert_{\mathcal{K}(f)}$. 
\end{enumerate}
\end{defn}

It turns out that in all three cases we considered in this paper, we can define (anti-)holomorphic correspondences that combine actions of (anti-)polynomials with the corresponding groups.
Specifically, such a correspondence is manufactured by pulling back the dynamics of the mating via the rational maps that uniformize the associated simply connected quadrature/inversive domains.
The first instance of this construction appeared in \cite[Appendix~B]{LLMM3}, where matings of $\overline{z}^d$ with regular ideal polygon reflection groups were realized as antiholomorphic correspondences. 
Our next theorem provides a systematic way of producing such examples.

\begin{theorem}\label{all_corr_thm}
Let $G$ be an ideal $(d+1)-$gon reflection group or the anti-Hecke group.

Let $f$ be a degree $d$ anti-polynomial with connected Julia set which is either
\begin{itemize}
\item geometrically finite; or 
\item periodically repelling, finitely renormalizable.
\end{itemize}

Then there exists an antiholomorphic correspondence $\mathfrak{C}$ which is a mating of $f$ and $G$.

Similarly, let $\Sigma$ be a hyperbolic orbifold of genus zero with arbitrarily many (at least one) punctures, at most one order two orbifold point, and at most one order $\nu\geq 3$ orbifold point, with the corresponding Fuchsian group $G$.
Let 
$$
d(\Sigma)=\begin{cases}
1-2\nu\cdot\chi_{\mathrm{orb}}(\Sigma) \text{ if $\Sigma$ has an order $\nu \geq 3$ orbifold point,}\\
1-2\chi_{\mathrm{orb}}(\Sigma) \text{ otherwise.}
\end{cases}
$$ 
Let $f$ be a degree $d(\Sigma)$ polynomial with connected Julia set which is either
\begin{itemize}
\item geometrically finite; or 
\item periodically repelling, finitely renormalizable.
\end{itemize}
Then there exists a holomorphic correspondence $\mathfrak{C}$ which is a mating of $f$ and $G$.
\end{theorem}

The antiholomorphic parts of Theorem~\ref{all_corr_thm} are proved in Sections~\ref{polygonal_corr_sec},~\ref{schwarz_anti_farey_sec}, and the holomorphic part of it is proved in Section~\ref{holo_para_space_reln_sec}.

We remark that the correspondences produced by Theorem~\ref{all_corr_thm} are \emph{reversible}; i.e., their forward branches are (anti-)conformally conjugate to their backward branches. This property facilitates the dynamical study of such correspondences.

\subsection{Connections with known combination theorems for correspondences}\label{connections_subsec}

\subsubsection{Special cases of Theorem~\ref{all_corr_thm}}\label{special_cases_subsubsec}
We now list several known results in the literature that are generalized by Theorem~\ref{all_corr_thm}.
\begin{enumerate}
\item The algebraic correspondences that arise as matings of quadratic polynomials and the modular group are precisely the bi-degree $2$:$2$ correspondences introduced by Bullett and Penrose \cite{BP} (cf. \cite{BL20,BL22}). Bullett and Ha{\"i}ssinsky showed in \cite{BH07} that this space of correspondences contains matings of Collet-Eckmann quadratic polynomials (with connected Julia set) and the modular group.

\item The antiholomorphic correspondences that arise from Theorem~\ref{all_corr_thm} as matings of quadratic anti-polynomials and the anti-Hecke group isomorphic to $\Z/2\Z\ast\Z/3\Z$ were investigated in \cite{LLMM3}.

\item Some examples of correspondences realizing matings of degree $d$ complex polynomials with the Hecke group (isomorphic to $\Z/2\Z\ast\Z/(d+1)\Z$) were furnished in \cite{BF03,BF05}. In \cite{BF03}, it was conjectured that all complex polynomials with connected Julia set can be mated with the corresponding Hecke group. Theorem~\ref{all_corr_thm} proves this conjecture for all  geometrically finite polynomials and finitely renormalizable, periodically repelling polynomials.

\item In \cite{LMM23}, correspondences were constructed as matings of degree $d$ semi-hyperbolic anti-polynomials (i.e., anti-polynomials with no parabolic cycles and no recurrent Julia critical points) with connected Julia set and the anti-Hecke group (isomorphic to $\Z/2\Z\ast\Z/(d+1)\Z$).

\item In \cite{MM2}, correspondences were constructed as matings of semi-hyperbolic polynomials with connected Julia set and the corresponding Fuchsian group under an additionally imposed real-symmetry condition. These examples also become special cases of Theorem~\ref{all_corr_thm}.
\end{enumerate}

\subsubsection{(Anti-)polynomials vs parabolic (anti-)rational maps}\label{hyp_vs_para_subsubsec}

We now compare Theorem~\ref{all_corr_thm} with related (but technically different) constructions of algebraic correspondences as matings of \emph{parabolic (anti-)rational maps} and (anti-)Hecke groups.

The external maps arising from reflection/Fuchsian groups (namely, Nielsen, anti-Farey, and factor Bowen-Series maps) have at least one parabolic fixed point on the circle, and hence are not quasisymmetrically conjugate to the power map $z^d$ or $\overline{z}^d$. This discrepancy prohibits one from applying standard quasiconformal surgery tools to mate Nielsen/anti-Farey/factor Bowen-Series maps with (anti-)polynomials. In other words, the conformal matings described in Theorems~\ref{conf_mating_polygonal_thm},~\ref{conf_mating_antiFarey_thm}, and~\ref{conf_mating_b_inv_thm} cannot be constructed by quasiconformal surgery.

However, it turns out that the anti-Farey maps and the factor Bowen-Series maps associated with Hecke groups are somewhat special as they have a unique parabolic fixed point on the circle. Equivalently, the (anti-)Hecke groups have a unique conjugacy class of parabolic elements. This implies that the above maps are quasisymmetrically conjugate to parabolic (anti-)Blaschke products since the latter class of maps also have a unique parabolic fixed point on the circle. This quasisymmetric compatibility property was utilized
\begin{enumerate}
\item to prove that every quadratic parabolic rational map (with connected Julia set) can be mated with the modular group as a holomorphic correspondence \cite{BL20},

\item to construct anti-holomorphic correspondences as matings of arbitrary degree $d$ parabolic anti-rational maps (with connected Julia set) with anti-Hecke groups \cite{LLMM3,LMM23}, and 

\item to construct holomorphic correspondences as matings of arbitrary degree $d$ parabolic rational maps (with connected Julia set) with Hecke groups \cite{BLLM}.
\end{enumerate}

To conclude, let us note that for semi-hyperbolic (anti-)polynomials with connected Julia set, the conformal matings furnished by Theorems~\ref{conf_mating_polygonal_thm},~\ref{conf_mating_antiFarey_thm}, and~\ref{conf_mating_b_inv_thm} have \emph{David regularity}. Specifically, for a semi-hyperbolic (anti-)polynomial $f$ with connected Julia set, the topological conjugacy (conformal on the interior) from the filled Julia set dynamics of $f$ to the non-escaping set dynamics of the conformal mating between $f$ and the Nielsen/anti-Farey/factor Bowen-Series map is the restriction of a global David homeomorphism (cf. \cite[Theorem~10.21]{LMMN}, \cite[Theorem~C]{LMM23}). In general, we ask the following question:

\begin{question}
What is the regularity of the mating between a general periodically repelling, finitely renormalizable (anti-)polynomial  and the Nielsen/anti-Farey/factor Bowen-Series map?
\end{question}

\subsection{Proof techniques}
Here we present an overview of the techniques employed in this paper:
\begin{itemize}
    \item Utilizing the David surgery method, we construct matings of post-critically finite (anti-)polynomials / geometrically finite (anti-)polynomials with appropriate ideal polygon reflection groups / anti-Hecke groups / Fuchsian genus zero orbifold groups.
    It is important to note that while David surgery is a powerful tool, there are analytic and geometric constraints, as the current techniques only apply to John domains.
    \item To extend our constructions to more general setting, we take limits of post-critically finite matings.
    These matings emerge naturally as degenerate (anti-)polynomial-like maps.
    As a preliminary step, we prove a boundedness result for the connectivity locus of such degenerate (anti-)polynomial-like maps.
    \item The structure of degenerate (anti-)polynomial-like maps enables us to adapt puzzle techniques, which allow us to establish combinatorial continuity of the association between periodically repelling (anti-)polynomials and their corresponding matings.
    \item By dividing the plane into pieces using puzzles, we apply generalized renormalization methods to establish combinatorial rigidity for periodically repelling, finitely renormalizable matings.
\end{itemize}
These results permit us to conclude the existence of conformal matings and combinatorial matings, as well as continuity and combinatorial continuity of `straightening maps' between the parameter spaces.
Notably, the fixed ray puzzles give a natural stratification of the space of Schwarz reflections. 
This stratification leads to discontinuity of the straightening map between geometrically finite parameters.
Furthermore, we derive algebraic descriptions for these matings, which enable us to transfer the previous mating results to the framework of correspondences.

\medskip

\noindent\textbf{Acknowledgements.} Part of the work on this project was carried out during the authors' visit to The Fields Institute, Urgench State University, and Simons Laufer Mathematical Sciences Institute (MSRI).

\medskip

\noindent\textbf{Notation:} 
\begin{itemize}
\item For $X\subset\widehat{\C}$, the sets $\partial X, X^c$ stand for the boundary and the complement of $X$ in $\widehat{\C}$.

\item $\eta^+(z)=1/z$.

\item $\eta^-(z)=1/\overline{z}$.

\item $\D^*=\widehat{\C}\setminus\overline{\D}$.

\item $B(a,r):=\{\vert z-a\vert<r\}$, where $a\in\C$ and $r>0$.

\item $\overline{B}(a,r):=\{\vert z-a\vert\leq r\}$, where $a\in\C$ and $r>0$.


\item $\textrm{Aut}(\D)$ = the group of all M{\"o}bius automorphisms of $\D$,\\ 
$\textrm{Aut}^\pm(\D)$ = the group of all M{\"o}bius and anti-M{\"o}bius automorphisms of $\D$.
\end{itemize}

\begin{LARGE}\part{Antiholomorphic world}\end{LARGE}\label{part_one}
\bigskip

\section{Reflection groups and associated external maps}\label{nielsen_farey_sec}

We denote the regular $(d+1)-$gon reflection group by $\pmb{G}_{d}$. It is generated by reflections $\rho_1,\cdots,\rho_{d+1}$ in the circles $C_1,\cdots,C_{d+1}$ such that $C_i\cap\D$ is a hyperbolic geodesic of $\D$ connecting $\omega^{i-1}$ and $\omega^i$, where $\omega:=e^{\frac{2\pi i}{d+1}}$. Let $\Pi$ be the ideal polygon in $\D$ with vertices at the $(d+1)-$st roots of unity. Then $\Pi$ is a fundamental domain for the $\pmb{G}_d-$action on $\D$. 

The group of conformal and anti-conformal automorphisms of $\D$ generated by $\pmb{G}_d$ and $\langle M_\omega\rangle$, where $M_\omega(z)=\omega z$ and $\omega:=e^{\frac{2\pi i}{d+1}}$, is called the \emph{anti-Hecke group}. We denote the anti-Hecke group by $\mathbbm{G}_d$. It is an index $d+1$ extension of $\pmb{G}_d$.

\subsection{Nielsen map} The \emph{Nielsen map} $\pmb{\cN}_d$ associated with the regular ideal polygon group $\pmb{G}_d$ is defined as
$$
\pmb{\cN}_d:\overline{\D}\setminus\Int{\Pi}\longrightarrow\overline{\D},\quad z\mapsto \rho_i(z)\quad \mathrm{if}\quad z\in \overline{H_i},
$$
where $H_i$ is the hyperbolic half-plane cut out by the bi-infinite geodesic $C_i\cap\D$ and the counter-clockwise circular arc $\arc{\omega^{i-1},\omega^i}$, for $i\in\{1,\cdots,d+1\}$ (see Figure~\ref{itg_nielsen_fig}). The Nielsen map $\pmb{\cN}_d$ fixes $\partial\Pi$ pointwise, and restricts to an orientation-reversing, $C^1$, expansive, degree $d$ covering map of $\mathbb{S}^1$. Hence, it is topologically conjugate to the map $\overline{z}^d:~\mathbb{S}^1 \longrightarrow \mathbb{S}^1$. There is a unique such conjugacy that sends $1$ to $1$. We denote this map by $\pmb{\mathcal{E}}_d:\mathbb{S}^1\to\mathbb{S}^1$, and call it the \emph{$d-$th Minkowski circle homeomorphism}.

\subsection{Anti-Farey map}

Let us now associate a circle endomorphism with the anti-Hecke group $\mathbbm{G}_d$.
Roughly speaking, we first pass to the index $d$ subgroup $\pmb{G}_d\leqslant\mathbbm{G}_d$, look at its Nielsen map $\pmb{\cN}_d$, and then pass to a semi-conjugate (or factor) of the Nielsen map using the torsion element $M_\omega$.

More precisely, observe that due to the rotational symmetry of the fundamental domain $\Pi$, the Nielsen map $\pmb{\cN}_d$ of commutes with $M_\omega$. This allows one to construct a factor dynamical system
$$
\widehat{\pmb{\cN}_d}:\faktor{\left(\overline{\D}\setminus \Int{\Pi}\right)}{\langle M_\omega\rangle}\longrightarrow \faktor{\overline{\D}}{\langle M_\omega\rangle}
$$
that is semi-conjugate to $\pmb{\cN}_d:\overline{\D}\setminus\Int{\Pi}\to\overline{\D}$ via the quotient map~$\overline{\D}~\to~\overline{\D}/\langle M_\omega\rangle$.
\begin{figure}[ht]
\captionsetup{width=0.96\linewidth}
\begin{tikzpicture}
\node[anchor=south west,inner sep=0] at (0,0) {\includegraphics[width=0.41\textwidth]{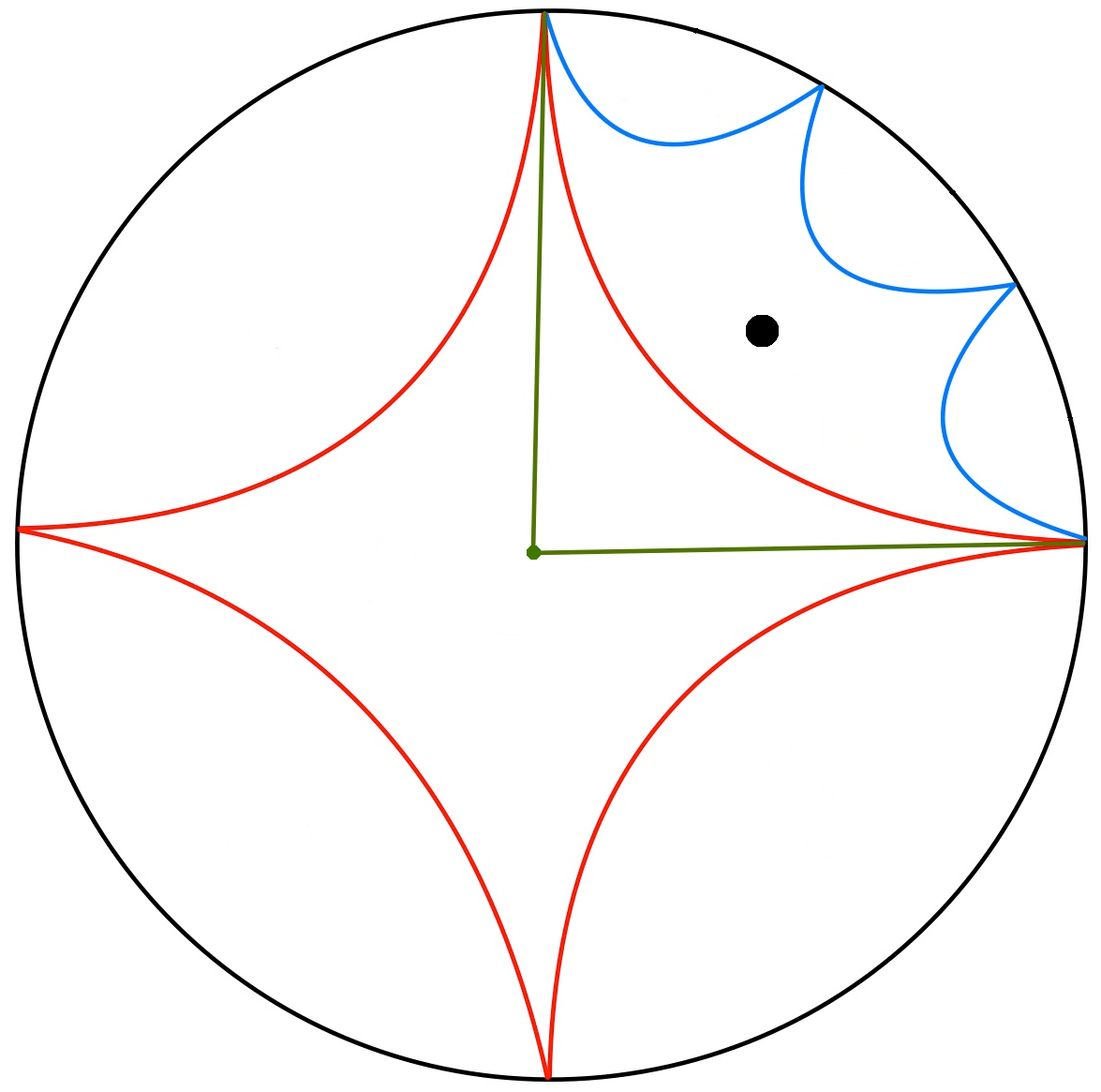}}; 
\node[anchor=south west,inner sep=0] at (6,0) {\includegraphics[width=0.4\textwidth]{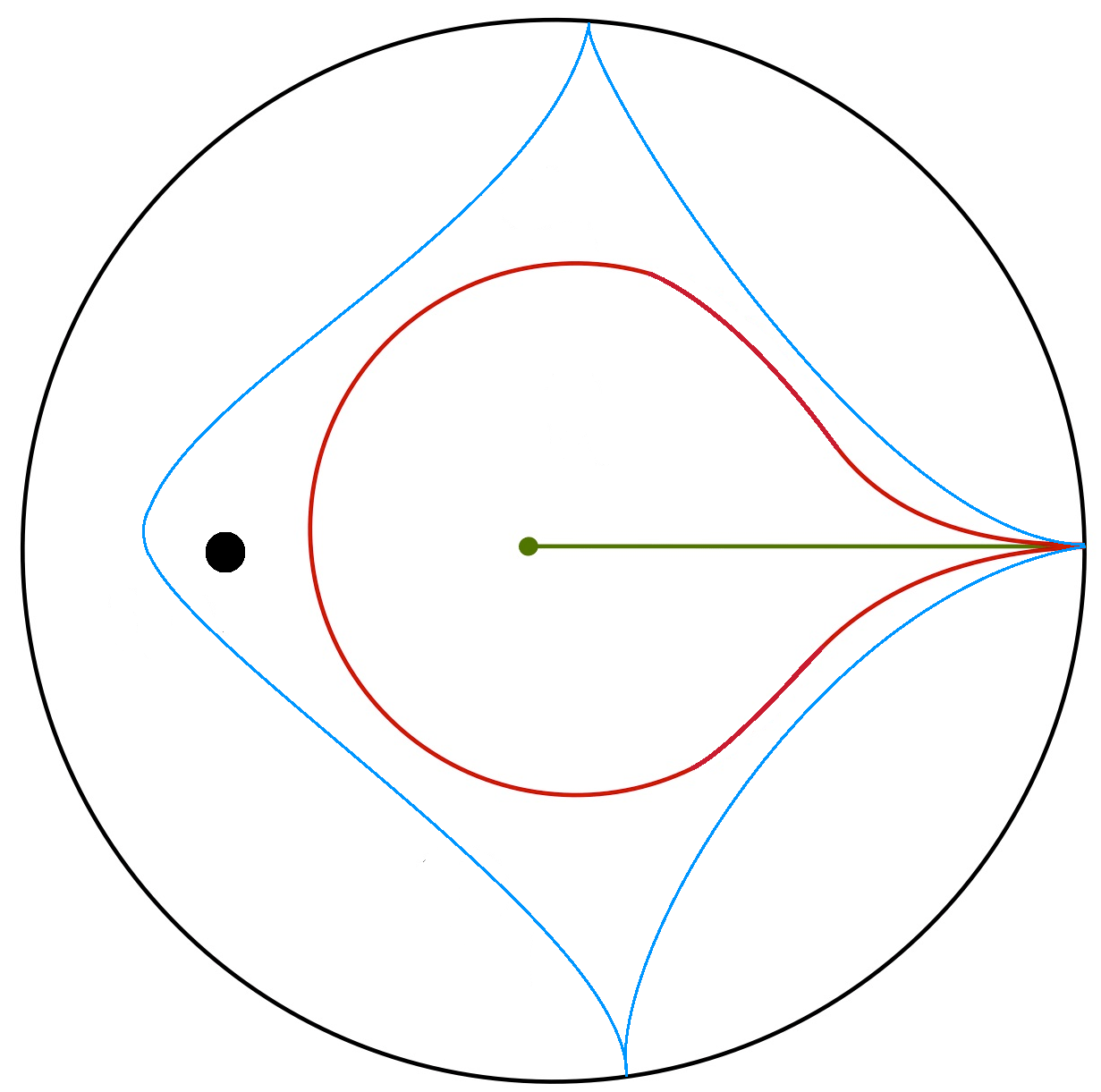}};
\node at (8.48,3.2) {$\mathcal{H}$};
\end{tikzpicture}
\caption{Left: a fundamental domain for the $\pmb{G}_3-$action on $\D$ is given by the regular quadrilateral bounded by the red geodesics. A fundamental domain for the action of $\mathbbm{G}_3$ on $\D$ is given by the triangle having the two green lines and the red geodesic connecting $1$ to $i$ as its edges. Right: When the Riemannian orbifold $\faktor{\D}{\langle M_i\rangle}$ is uniformized by the disk, the map $\pmb{\cF}_3$ has the region bounded by the black unit circle and the red monogon as its domain of definition. The map $\pmb{\cF}_3$ fixes the red monogon pointwise, and acts as an orientation-reversing, degree three circle covering with a unique parabolic fixed point at $1$. The heavy black dot represents the triple critical point of $\pmb{\cF}_3$.}
\label{anti_farey_fig}
\end{figure}

Finally, as $z\mapsto z^{d+1}$ defines a conformal isomorphism $\xi:\overline{\D}/\langle M_\omega\rangle\longrightarrow\overline{\D}$, we have the so-called \emph{degree $d$ anti-Farey map}
$$
\pmb{\cF}_d:= \xi\circ\widehat{\pmb{\cN}_d}\circ\xi^{-1}:\overline{\D}\setminus\Int{\mathcal{H}}\to\overline{\D},
$$
where $\mathcal{H}:=\xi(\Pi/\langle M_\omega\rangle)$ (see Figure~\ref{anti_farey_fig}).

The fact that $\pmb{\cN}_d$ fixes $\partial\Pi$ pointwise, implies that $\pmb{\cF}_d$ fixes $\partial\mathcal{H}$ pointwise. Moreover, the map $\pmb{\cF}_d$ restricts to an expansive, orientation-reversing, degree $d$, piecewise anti-analytic covering of $\mathbb{S}^1$. Hence, $\pmb{\cF}_d\vert_{\mathbb{S}^1}$ is topologically conjugate to $\overline{z}^d_{\mathbb{S}^1}$.

The anti-Farey map $\pmb{\cF}_d$ has a couple of key differences from the Nielsen map $\pmb{\cN}_d$. These dissimilarities come from passing to a factor dynamical system. Note that the piecewise M{\"o}bius map $\pmb{\cN}_d$ has no critical point, but the anti-Farey map has a unique critical point (of multiplicity $d$) at $\xi(\rho_1(0))$ with associated critical value at $0$. Further, while the Nielsen map has $d+1$ parabolic fixed points on the circle, the anti-Farey map has a unique parabolic fixed point on $\mathbb{S}^1$ (at the point $1$).

We refer the reader to \cite[\S 3.1]{LMM23} for a detailed discussion of the anti-Farey map and the anti-Hecke group.

\begin{remark}
The idea of constructing an external map as a factor of another will also be used in Section~\ref{fbs_sec} to define factor Bowen-Series maps.
\end{remark}

\section{Schwarz reflection maps}\label{schwarz_sec}

\subsection{Simply connected quadrature domains}

\begin{defn}\label{qd_schwarz_def}
A domain $\Omega\subsetneq\widehat{\C}$ with $\infty\notin\partial\Omega$ and $\mathrm{int}(\overline{\Omega})=\Omega$ is called a \emph{quadrature domain} if there exists a continuous function $S:\overline{\Omega}\to\widehat{\C}$ satisfying the following two properties:
\begin{enumerate}
\item $S=\mathrm{id}$ on $\partial \Omega$, and

\item $S$ is anti-meromorphic on $\Omega$.
\end{enumerate}

The map $S$ is called the \emph{Schwarz reflection map} of $\Omega$.
\end{defn}

The boundary of a quadrature domain is real-analytic with conformal cusps and double points as only possible singularities \cite{Sak91}. Thus, the map $S$ can be regarded as a semi-global anti-meromorphic extension of the local Schwarz reflection maps with respect to the non-singular points of $\partial\Omega$.

\begin{prop}\label{simp_conn_quad_prop}\cite[Theorem~1]{AS76},\cite[Lemma~4.1]{LM16}
A simply connected domain $\Omega\subsetneq\widehat{\C}$ (with $\infty\notin\partial\Omega$ and $\mathrm{int}(\overline{\Omega})=\Omega$) is a quadrature domain if and only if the Riemann map $\phi:\mathbb{D}\to\Omega$ is rational.
In this case, the Schwarz reflection map $S$ of $\Omega$ is given by $\phi\circ\eta^-\circ(\phi\vert_{\D})^{-1}$. 

Moreover, if the degree of the rational map $\phi$ is $d$, then $S:S^{-1}(\Omega)\to\Omega$ is a (branched) covering of degree $(d-1)$, and $S:S^{-1}(\Int{\Omega^c})\to\Int{\Omega}^c$ is a (branched) covering of degree $d$.
\end{prop}
\[
  \begin{tikzcd}
    \D \arrow{d}[swap]{\eta^-} \arrow{r}{\phi} & \Omega \arrow{d}{S} \\
     \widehat{\C} \arrow{r}{\phi}  & \widehat{\mathbb{C}}
  \end{tikzcd}
\] 

We will call the degree of the rational map $\phi$ the \emph{degree} of the associated quadrature domain $\Omega$.

\subsection{Schwarz reflections in quadrature multi-domains}\label{subsec:qmd}

\begin{defn}\label{mqd_schwarz_def}
Let $\{\Omega_1,\cdots,\Omega_k\}$ be a finite collection of disjoint simply connected domains with $\Omega_j\subsetneq\widehat{\C}$, $\infty\notin\partial\Omega_j$ and $\mathrm{int}(\overline{\Omega_j})=\Omega_j$, for $j=1,\cdots,k$. We call $\cD=\displaystyle\bigsqcup_{j=1}^k \Omega_j$ a \emph{quadrature multi-domain} if there exists a continuous function $S:\overline{\cD}\to\widehat{\C}$ satisfying the following two properties:
\begin{enumerate}
\item $S=\mathrm{id}$ on $\partial \cD$, and

\item $S$ is anti-meromorphic on $\cD$.
\end{enumerate}

The map $S$ is called the \emph{Schwarz reflection map} of the quadrature multi-domain $\cD$.
\end{defn}

Note that for a quadrature multi-domain $\cD:=\displaystyle\bigsqcup_{j=1}^k\Omega_j$, each component $\Omega_j$ is a quadrature domain and $S\vert_{\overline{\Omega_j}}$ agrees with the Schwarz reflection map $S_j$ of $\Omega_j$. 

Let $\phi_j:\overline{\D}\to\overline{\Omega}_j$ be the Riemann uniformizations of the simply connected quadrature domains $\Omega_j$ such that each $\phi_j$ extends as a rational map of $\widehat{\C}$ of degree $d_j$. It follows that $S_j:S_j^{-1}(\Omega_j)\to\Omega_j$ is a branched covering of degree $(d_j-1)$, and $S_j:S_j^{-1}(\Int{\Omega_j^c})\to\Int{\Omega_j^c}$ is a branched covering of degree $d_j$. Therefore, $S:S^{-1}(\cD)\to\cD$ is a branched covering of degree 
$$
d:=\sum_{j=1}^k d_j-1.
$$
On the other hand, $S:S^{-1}(\Int{\cD^c})\to\Int{\cD^c}$ is a branched covering of degree $d+1$.
We refer to $S$ as a degree $d$ Schwarz reflection map.

Set 
$$
 T(S):=\widehat{\C}\setminus\cD,\quad \mathrm{and}\quad T^0(S):=T\setminus \mathfrak{S},
$$
where $\mathfrak{S}$ is the set of singular points (cusps and double points) on $\partial T(S)$. We call the components of $T^0(S)$ the \emph{fundamental tile} and the components of $S^{-n}(T^0(S))$ the \emph{tiles of rank $n$}.
The \emph{escaping/tiling set} $T^\infty(S)$ of $S$ is defined as the union of all tiles; i.e., 
$$
T^\infty(S):=\displaystyle\bigcup_{n=0}^\infty S^{-n}(T^0(S)).
$$ 
The \emph{non-escaping set} of $S$ is defined as $K(S):=\widehat{\C}\setminus T^\infty(S)$. A connected component of $\Int{K(S)}$ is called a \emph{Fatou component} of $S$. The common boundary of $K(S)$ and $T^\infty(S)$ is called the \emph{limit set} of $S$, and is denoted by~$\Lambda(S)$.

\subsection{Conformal mating and Schwarz reflections}\label{nielsen_poly_mating_def_subsec}

Let $f$ be a monic, centered complex polynomial of degree $d$ with a connected and locally connected Julia set. We denote the \emph{B{\"o}ttcher coordinate} of $f$ by
$$
\psi_f:\widehat{\C}\setminus\overline{\D}\to\mathcal{B}_\infty(f):=\widehat{\C}\setminus\mathcal{K}(f).
$$ 
This map, which conjugates $\overline{z}^{d}$ to $f$, is normalized to be tangent to the identity map near infinity.
As $\mathcal{J}(f)$ is locally connected, $\psi_f$ extends continuously to $\mathbb{S}^1$ to yield a semi-conjugacy between $\overline{z}^d\vert_{\mathbb{S}^1}$ and $f\vert_{\mathcal{J}(f)}$. 
By Section~\ref{nielsen_farey_sec}, there exists a homeomorphism $\mathfrak{h}:\mathbb{S}^1\to\mathbb{S}^1$ that conjugates $\overline{z}^d$ to $\pmb{\cN}_d$ (respectively, $\pmb{\cF}_d$). We normalize $\mathfrak{h}$ so that it sends the fixed point $1$ of $z^d$ to the fixed point $1$ of $\pmb{\cN}_d$ (respectively, $\pmb{\cF}_d$).

We define an equivalence relation $\sim$ on $\mathcal{K}(f)\bigsqcup \overline{\D}$ generated by 
\begin{equation}
\psi_{f}(\zeta)\sim \mathfrak{h}(\overline{\zeta}),\ \textrm{for all}\ z\in\mathbb{S}^1.
\label{conf_mat_equiv_rel_1}
\end{equation} 

We now precise what it means to \emph{conformally mate} $f$ with the Nielsen map $\pmb{\cN}_d$ (the case of the anti-Farey map $\pmb{\cF}_d$ is exactly the same).

\begin{defn}\label{conf_mat_def_1}
The maps $f$ and $\pmb{\cN}_d$ are said to be \emph{conformally mateable} if there exist a continuous map $S: \mathrm{Dom}(S)\subsetneq\widehat{\C}\to\widehat{\C}$ (called a \emph{conformal mating} of $\pmb{\cN}_d$ and $f$) that is complex-analytic in the interior of $\mathrm{Dom}(S)$ and continuous maps 
	$$
	\mathfrak{X}_f:\mathcal{K}(f)\to\widehat{\C}\ \textrm{and}\ \mathfrak{X}_{\pmb{\cN}_d}: \overline{\D}\to\widehat{\C},
	$$
	conformal on $\Int{\mathcal{K}(f)}$ and $\D$ (respectively), satisfying
	\begin{enumerate}
		\item\label{topo_cond_1} $\mathfrak{X}_f\left(\mathcal{K}(f)\right)\cup \mathfrak{X}_\Sigma\left(\overline{\D}\right) = \widehat{\C}$,
		
		\item\label{dom_cond_1} $\mathrm{Dom}(S)= \mathfrak{X}_f(\mathcal{K}(f))\cup\mathfrak{X}_{\pmb{\cN}_d}(\overline{\D}\setminus\Int{\Pi})$,
		
		\item $\mathfrak{X}_f\circ f(z) = S\circ \mathfrak{X}_f(z),\quad \mathrm{for}\ z\in\mathcal{K}(f)$,
		
		\item $\mathfrak{X}_{\pmb{\cN}_d}\circ \pmb{\cN}_d (w) = F\circ \mathfrak{X}_{\pmb{\cN}_d}(w),\quad \mathrm{for}\ w\in
		\overline{\D}\setminus\Int{\Pi}$,\quad and 
		
		\item\label{identifications_1} $\mathfrak{X}_f(z)=\mathfrak{X}_{\pmb{\cN}_d}(w)$ if and only if $z\sim w$ where $\sim$ is the equivalence relation on $\mathcal{K}(f)\sqcup \overline{\D}$ defined by Relation~\eqref{conf_mat_equiv_rel_1}.
	\end{enumerate}
\end{defn}

Using David surgery techniques, it was proved in \cite{LMMN,LMM23} that many anti-polynomials with connected Julia set can be conformally mated with the Nielsen/anti-Farey map. Crucially, the property of $\pmb{\cN}_d:\overline{\D}\setminus\Int{\Pi}\longrightarrow\overline{\D}$ (respectively, of $\pmb{\cF}_d:\overline{\D}\setminus\Int{\mathcal{H}}\longrightarrow\overline{\D}$) that it fixes $\partial\Pi$ (respectively, $\partial\mathcal{H}$) pointwise, translates to the fact that the resulting conformal matings are Schwarz reflection maps.

\begin{theorem}\cite[Lemma~10.17, Theorem~10.21]{LMMN}\cite[Proposition~3.7]{LMM23}\label{nielsen_antipoly_conf_mat_thm}
Let $f$ be a degree $d$ hyperbolic anti-polynomial with connected Julia set.
Then, there exists a conformal mating $S$ of $f:\mathcal{K}(f)\to\mathcal{K}(f)$ and $\pmb{\cN}_d:\overline{\D}\setminus\Int{\Pi}\longrightarrow\overline{\D}$ (respectively, $\pmb{\cF}_d:\overline{\D}\setminus\Int{\mathcal{H}}\longrightarrow\overline{\D}$). Moreover, $S$ is a Schwarz reflection map in a quadrature multi-domain.
\end{theorem}

\section{Polygonal Schwarz reflection maps}\label{poly_schwarz_sec}

\subsection{Basic definitions and properties}\label{subsec:bdp}

We continue to use the notation of the previous section.
Suppose further that $\overline{\cD}$ is connected and simply connected. Since the boundary of a quadrature domain consists of real-analytic curves possibly with cusps and double point singularities, no more than two distinct components of $\cD$ can touch at a point. Thus, the above condition is equivalent to following ones.
\begin{enumerate}
\item Each $\Omega_j$ is a Jordan domain, and the \emph{contact graph} of the quadrature multi-domain (i.e., a graph having a vertex for each $\Omega_j$ and an edge connecting two vertices if the corresponding quadrature domains touch) is a tree.
\item $\overline{\cD}$ is a pinched polygon, and the associated finite geodesic lamination contains no polygon.
\end{enumerate}
We will refer to such a $\cD$ as a \emph{tree-like} quadrature multi-domain, and call the Schwarz reflection of the tree-like quadrature multi-domain $\cD$ a \emph{Schwarz reflection of a pinched polygon.}

For a tree-like quadrature multi-domain, the desingularized droplet $T^0(S)$ is homeomorphic to an ideal polygon in the hyperbolic plane.

\begin{defn}\label{poly_schwarz_def}
A degree $d$ Schwarz reflection map $S:\overline{\cD}\to\widehat{\C}$ of a pinched polygon $\overline{\cD}$
is said to be \emph{polygonal} if  $T^0(S)$ is homeomorphic to an ideal $(d+1)-$gon in $\D$ and $S$ has no critical values in $T^0(S)$.
 
A degree $d$ polygonal (piecewise) Schwarz reflection map is called \emph{regular polygonal} if $T^0(S)$ is conformally equivalent to the regular ideal $(d+1)-$gon $\Pi$.  
We denote by~$\mathscr{S}_d$ the space of degree $d$ \emph{regular} polygonal Schwarz reflections.
\end{defn}
We will often omit the word `regular' for brevity. We normalize a polygonal Schwarz reflection $S$ so that the conformal map between $\Pi$ and $T^0(S)$ sends the origin to $\infty$ and has asymptotics $z\mapsto 1/z+O(z)$ as $z\to 0$.

\begin{remark}
In the above definition, the condition that $T^0(S)$ is homeomorphic to an ideal $(d+1)-$gon in $\D$ is equivalent to requiring that the sum of the number of cusps and twice the number of double points on $\partial\cD$ is equal to $d+1$.
\end{remark}

The following result justifies the nomenclature `polygonal Schwarz reflection maps'. Indeed, it shows that such maps are characterized by having the Nielsen map of a polygonal reflection group as the conformal model of their dynamics near the droplet.

\begin{lem}\label{polygonal_dpl_lem}
Let $S:\overline{\cD}\to\widehat{\C}$ be a degree $d$ Schwarz reflection map of a pinched polygon. Then, $S$ is regular polygonal if and only if $S:S^{-1}(T^0(S))\to T^0(S)$ is conformally conjugate to $\pmb{\cN}_d:\pmb{\cN}_d^{-1}(\Pi)\to\Pi$.
\end{lem}
\begin{proof}
It is clear that if $S:S^{-1}(T^0(S))\to T^0(S)$ is conformally conjugate to $\pmb{\cN}_d:\pmb{\cN}_d^{-1}(\Pi)\to\Pi$, then $S$ is regular polygonal.

For the converse, let us assume that $S$ is regular polygonal. Then, there exists a conformal map $\psi_S:\Pi\to T^0(S)$ that preserves the ideal vertices. The fact that $T^0(S)$ is simply connected and $S$ has no critical point in $S^{-1}(T^0(S))$ together imply that $S^{-1}(T^0(S))$ consists of $d+1$ components, each of which maps diffeomorphically onto $T^0(S)$ under $S$. Similarly, $\pmb{\cN}_d^{-1}(\Pi)$ consists of $d+1$ components, each of which maps diffeomorphically onto $\Pi$ under $\pmb{\cN}_d$.
Moreover, $S$ and $\pmb{\cN}_d$ act as identity maps on $\partial T^0(S)$ and $\partial \Pi$, respectively. Hence, $\psi_S$ can be lifted to a conformal isomorphism from $\pmb{\cN}_d^{-1}(\Pi)$ onto $S^{-1}(T^0(S))$. Note that the trivial actions of $S$ and $\pmb{\cN}_d$ on $\partial T^0(S)$ and $\partial \Pi$ (respectively) ensure that the lifts match on the boundaries of $\Pi, T^0(S)$. By construction, the extended map conjugates $S:S^{-1}(T^0(S))\to T^0(S)$ to $\pmb{\cN}_d:\pmb{\cN}_d^{-1}(\Pi)\to\Pi$.
\end{proof}

An important consequence of Lemma~\ref{polygonal_dpl_lem} is that for a degree $d$ polygonal (piecewise) Schwarz reflection $S:\overline{\cD}\to\widehat{\C}$, the restriction $S:\overline{S^{-1}(\cD)}\to\overline{\cD}$ is a degree $d$ degenerate anti-polynomial-like map. Many of the following results can be regarded as analogs of basic facts about classical polynomial-like maps in the current setting.

\begin{prop}\label{basic_top_prop}
Let $S\in\mathscr{S}_d$.
\noindent\begin{enumerate}\upshape
\item $T^\infty(S)$ and $K(S)$ are completely invariant under $S$.

\item The escaping set $T^\infty(S)$ is an open and connected. The non-escaping set $K(S)$ is a full compact subset of $\widehat{\C}$. 

\item Each Fatou component of $S$ is simply connected.
\end{enumerate}
\end{prop}
\begin{proof}
1) This is immediate from the definitions.

2) Let us denote the union of the tiles of rank $0$ through $k$ by $E^k$. Since every tile of rank $k\geq1$ is attached to a tile of rank $(k-1)$ along a boundary curve and $E^0=T^0(S)$ is connected, it follows that $\Int{E^k}$ is connected. Moreover, $\Int{E^k}\subsetneq\Int{E^{k+1}}$ for each $k\geq 0$. 

Note that if $z\in T^\infty(S)$ belongs to the interior of a tile of rank $k$, then it lies in the interior of $E^k$. On the other hand, if $z\in T^\infty(S)$ belongs to the boundary of a tile of rank $k$, then it lies in the interior of $E^{k+1}$. Hence, $T^\infty(S)=\bigcup_{k\geq 0}\Int{E^k}.$ 
Thus, $T^\infty(S)$ is an increasing union of open, connected sets, and hence itself is such. Consequently, its complement $K(S)$ is a full compact subset of $\widehat{\C}$. 

3) This follows from the fact that $K(S)$ is a full compact.
\end{proof}

\begin{lem}\label{qcdef_lemma}
Let $(S:\overline{\cD}\to\widehat{\C})\in\mathscr{S}_d$, $\mu$ be an $S-$invariant Beltrami coefficient supported on $K(S)$, and $\Phi$ be a quasiconformal map solving the Beltrami equation with coefficient $\mu$. Then $\Phi\circ S\circ\Phi^{-1}\in\mathscr{S}_d$.
\end{lem}
\begin{proof}
The assumption that $\mu$ is $S-$invariant implies that $\Phi\circ S\circ\Phi^{-1}$ is anti-meromorphic on $\Phi(\cD)$ that continuously extends to the identity map on $\partial\Phi(\cD)$. It follows that each component of $\Phi(\cD)$ is a simply connected quadrature domain, and hence $\Phi\circ S\circ\Phi^{-1}$ is a degree $d$ Schwarz reflection map with connected non-escaping set. Since $\mu$ is trivial on $T^\infty(S)$, the quasiconformal map $\Phi$ is conformal on $T^\infty(S)$. Therefore, $T^0(\Phi\circ S\circ\Phi^{-1})=\Phi(T^0(S))$ is conformally equivalent to $\Pi$; i.e., $\Phi\circ S\circ\Phi^{-1}\in\mathscr{S}_d$.
\end{proof}

\subsection{Connectedness locus of polygonal Schwarz reflections}

\begin{prop}\label{conn_locus_prop}
Let $(S:\overline{\cD}\to\widehat{\C})\in \mathscr{S}_d$. Then the following are equivalent.
\noindent\begin{enumerate}
\item $K(S)$ is connected.
\item $S^{\circ n}(c)\notin T^\infty(S)$, $\forall\ c\in \mathrm{crit}(S)$.
\item $S:T^\infty(S)\setminus\Int{T^0(S)}\longrightarrow T^\infty(S)$ is conformally conjugate to the Nielsen map $\pmb{\cN}_d:\D\setminus \Int{\Pi}\to \D$.
\end{enumerate}
\end{prop}
\begin{proof}
(1) $\implies$ (2) If some critical point of $S$ lies in the interior of a tile, then this tile would be mapped forward with degree larger than one and its closure would have at least two complementary components. Each of these complementary components would intersect non-escaping set of $S$, providing a disconnection of $K(S)$.

Next suppose that some critical point of $S$ lies on the boundary of a tile. Then this critical point is a common boundary point of two tiles of equal rank, and the closure of the union of these two tiles has at least two complementary components. Again, each of these components would intersect $K(S)$, proving that $K(S)$ is disconnected.

(2) $\implies$ (3) Let $\psi_S:\Pi\longrightarrow T^0(S)$ be a homeomorphism that is conformal on the interior. Since $T^0(S)$ is simply connected and none of the critical points of $S$ escapes to $T^0(S)$, it follows that each tile of $S$ maps diffeomorphically onto $T^0(S)$ under some iterate of $S$. Similarly, the tiles of the tessellation of $\D$ arising from the regular $(d+1)-$gon reflection group $\pmb{G}_d$ map diffeomorphically onto $\Pi$ under iterates of $\pmb{\cN}_d$. Furthermore, $S$ and $\pmb{\cN}_d$ act as identity maps on $\partial T^0(S)$ and $\partial \Pi$, respectively. This allows us to lift $\psi_S$ to a conformal isomorphism from $\D$ (which is the union of all iterated preimages of $\Pi$ under $\pmb{\cN}_d$) onto $T^\infty(S)$ (which is the union of all iterated preimages of $T^0(S)$ under $S$). Note that the trivial actions of $S$ and $\pmb{\cN}_d$ on $\partial T^0(S)$ and $\partial \Pi$ (respectively) ensure that the iterated lifts match on the boundaries of the tiles. By construction, the extended map conjugates $\pmb{\cN}_d:\D\setminus\Int{\Pi}\to\D$ to $S:T^\infty(S)\setminus \Int{T^0(S)}\to T^\infty(S)$.

(3) $\implies$ (1) The assumption that $T^\infty(S)$ is a topological disk implies that $K(S)$ is connected.
\end{proof}

\begin{defn}
We define the  {\em connectedness locus} as
$$
\mathscr{S}_{\pmb{\cN}_d}:= \{S \in \mathscr{S}_d: K(S) \text{ is connected.}\} \subseteq \mathscr{S}_d.
$$ 
\end{defn}

Note that by Proposition \ref{conn_locus_prop}, the connectedness locus $\mathscr{S}_{\pmb{\cN}_d}$ consists precisely of maps having the Nielsen map $\pmb{\cN}_d$ as the conformal model of their tiling set dynamics.

\subsection{Fatou components and critical points}
\begin{prop}\label{fatou_comp_prop}
Let $(S:\overline{\cD}\to\widehat{\C})\in\mathscr{S}_d$.
\noindent\begin{enumerate}
\item Every Fatou component of $S$ is eventually periodic.

\item Every periodic Fatou component $U$ of $S$ is of one of the following types.
\begin{enumerate}
\item $U$ is an immediate basin of attraction of an attracting cycle.
\item $U$ is an immediate basin of attraction of a parabolic cycle.
\item $U$ is a Siegel disk.
\item $U$ is an immediate basin of attraction of some $w\in\mathfrak{S}$.
\end{enumerate}
\end{enumerate}
\end{prop}
\begin{proof}
(1) Since the space $\mathscr{S}_d$ is parametrized by finitely many rational maps the sum of whose degrees is $d+1$, it follows that $\mathscr{S}_d$ is finite-dimensional. Moreover, the family $\mathscr{S}_d$ of maps is closed under quasiconformal conjugation in the sense of Lemma~\ref{qcdef_lemma}. Hence, the classical proof of non-existence of wandering Fatou components for rational maps can be adapted for the space $\mathscr{S}_d$ (cf. \cite{Sul,Eps93}).
 
(2) Let $U$ be a periodic Fatou component of period $k$ of $S$. By Proposition~\ref{basic_top_prop}, $U$ is simply connected. According to \cite[Theorem~5.2, Lemma~5.5]{Mil06}, the Fatou component $U$ is either the (immediate) basin of attraction of an attracting cycle, or a Siegel disk, or all orbits in $U$ converge to a unique boundary point $w$. If $w$ lies in $\cD$, then the classical Snail Lemma implies that $w$ must lie on a parabolic cycle. Otherwise, $w\in\partial\cap K(S)$, and hence $w$ is a singular point of $\partial T(S)$.
\end{proof}

\begin{prop}\label{fatou_crit_prop}
Let $(S:\overline{\cD}\to\widehat{\C})\in\mathscr{S}_d$.
\noindent\begin{enumerate}
\item If $S$ has an attracting or parabolic cycle, then the forward orbit of a critical point of $S$ converges to this cycle.  

\item  Every Cremer point (i.e. an irrationally neutral, non-linearizable periodic point) and Siegel disk boundary (of $S$) is contained in the $\omega-$limit set of a recurrent the critical point of $S$. 

\item If $w\in\mathfrak{S}$ has an attracting direction in $K(S)$, then some critical orbit of $S$ converges to $w$.
\end{enumerate}
\end{prop}
\begin{proof}
(1) and (2) The proof of \cite[Proposition~5.32]{LLMM1} applies to the current setting (cf. \cite{Man93}).

(3) Let us first assume that $w\in\mathfrak{S}$ is a cusp and has an attracting direction in $K(S)$. By \cite[Appendix~A]{LMM23}, the asymptotics of $S^{\circ 2}$ near $w$ is of the form 
\begin{equation}
\zeta\mapsto \zeta+c_1(\zeta-w)^{n/2}+O((\zeta-w)^{(n+1)/2}),
\label{asymp_eqn_1}
\end{equation} 
for some $c_1\in\C^*$, and $n\geq 5$.
Standard arguments used in the theory of parabolic germs allow one to count the number of attracting/repelling directions at $w$ (cf. \cite[Proposition~A.4]{LMM23}). These directions occur in an alternating fashion. Let $U$ be a Fatou component of $S$ such that the $S^{\circ 2}-$orbit of every point in $U$ converges to $w$ asymptotic to some attracting direction. A change of coordinate of the form $\zeta\mapsto 1/(\zeta-w)^{n/2-1}$ that sends $w$ to $\infty$ can now be applied to a sector containing this attracting direction to bring the map $S^{\circ 2}$ to the form $\xi\mapsto \xi+1+O(1/\xi)$ on a right-half plane. This enables one to argue as in \cite[Theorem~10.9]{Mil06} to prove the existence of an attracting petal in $U$ on which the action of $S^{\circ 2}$ is conformally conjugate to translation by $+1$ on a right half-plane. Finally, the proof of \cite[Theorem~10.15]{Mil06}, applied to the current setting, shows that the boundary of the maximal domain of definition of this conjugacy contains a critical point of $S^{\circ 2}$. It follows that $U$ contains a critical point of $S$.

Let us now consider a double point $w\in\mathfrak{S}$ that has an attracting direction in $K(S)$. After possibly renumbering the quadrature domains in $\cD$, we may assume that $\partial\Omega_1\cap\partial\Omega_2=\{w\}$. Then, $S^{\circ 2}$ is given by the two tangent-to-identity parabolic germs $S_2\circ S_1, S_1\circ S_2$ near $w$. These two germs are inverses of each other, and hence have the same number of attracting/repelling directions. Let $U$ be a Fatou component of $S$ such that the $S^{\circ 2}-$orbit of every point in $U$ converges to $w$ asymptotic to some attracting direction. Since $S^{\circ 2}$ is given by a parabolic germ, it follows that there exists an attracting petal in $U$ on which the action of $S^{\circ 2}$ is conformally conjugate to translation by $+1$ on a right half-plane, and that the boundary of the maximal domain of definition of this conjugacy contains a critical point of $S^{\circ 2}$ (cf. \cite[\S 10]{Mil06}. Thus, $U$ contains a critical point of $S$.
\end{proof}

\subsection{Dynamical rays and lamination}\label{subsec:dl}

We now define $\pmb{G}_d-$rays in $\D$ arising from the Cayley graph of $\pmb{G}_d$ with respect to the generating set $\rho_1,\cdots, \rho_{d+1}$. 

\begin{defn}\label{Gd_rays}
Let $(i_1, i_2, \cdots)\in\{1,\cdots, d+1\}^{\N}$ with $i_j\neq i_{j+1}$ for all $j$. The corresponding infinite sequence of tiles $\{\Pi,\rho_{i_1}(\Pi),\rho_{i_1}\circ\rho_{i_2}(\Pi),\cdots\}$ shrinks to a single point of $\mathbb{S}^1\cong \R/\Z$, which we denote by $\theta(i_1, i_2, \cdots)$. We define a \emph{$\pmb{G}_d-$ray at angle $\theta(i_1, i_2, \cdots)$} to be the concatenation of hyperbolic geodesics (in $\D$) connecting the consecutive points of the sequence $\{0,\rho_{i_1}(0),\rho_{i_1}\circ\rho_{i_2}(0),\cdots\}$. 
\end{defn}

We note that in general there may be more than one $\pmb{G}_d-$rays at a given angle, but all of them land at the point $\theta(i_1, i_2, \cdots)\in\mathbb{S}^1$.
For instance, both the symbol sequences $(1,d+1,1,d+1,\cdots)$ and $(d+1,1,d+1,1,\cdots)$ yield $\pmb{G}_d-$rays at angle $0$, and both land at the point $1$.
Moreover, the image of the tail of a ray at angle $\theta$ under $\pmb{\cN}_d$ is the tail of a ray at angle $\pmb{\cN}_d(\theta)$.

\begin{defn}\label{dyn_ray_schwarz}
For $S\in\mathscr{S}_{\pmb{\cN}_d}$, the image of a $\pmb{G}_d-$ray at angle $\theta$ under the map $\psi_S:\D\longrightarrow T^\infty(S)$ (that conjugates $\pmb{\cN}_d$ to $S$) is called a $\theta-$dynamical ray of $S$.
\end{defn}
Evidently, the image of the tail of a dynamical $\theta-$ray (of $S$) under $S$ is the tail of a dynamical ray at angle $\pmb{\cN}_d(\theta)$.
We denote the set of all (pre)periodic points of $\pmb{\cN}_d:\mathbb{S}^1\to\mathbb{S}^1$ by $\mathrm{Per}(\pmb{\cN}_d)$.

\begin{prop}[Landing of pre-periodic rays]\label{per_rays_land}
Let $S\in\mathscr{S}_{\pmb{\cN}_d}$. Then the following statements hold true.
\begin{enumerate}
\item Let $\theta\in\mathrm{Per}(\pmb{\cN}_d)$. Then a dynamical ray of $S$ at angle $\theta$ lands on~$\Lambda(S)$. 

\item The dynamical rays of $S$ at fixed angles land at points of $\mathfrak{S}$. Conversely, each point of $\mathfrak{S}$ is the landing point of a fixed ray.

\item The dynamical rays of $S$ at periodic angles of period larger than two land at repelling or parabolic periodic points of $\Lambda(S)$. 
\end{enumerate}
\end{prop}
\begin{proof}
1) After possibly applying some iterate of $S$, we may assume that $\theta$ is periodic of period $k$ under $\pmb{\cN}_d$. Then, the angle $\theta$ corresponds to a sequence in $\{1,\cdots, d+1\}^{\N}$ that is periodic under the left-shift map; i.e., $\theta=\theta\left(\overline{i_1, i_2, \cdots, i_k}\right)$ (see Definition~\ref{dyn_ray_schwarz}). Consider the sequence of points $w_n:=\psi_S\left(\left(\rho_{i_1}\circ\cdots\circ\rho_{i_k}\right)^{\circ n}(0)\right)$, $n\geq1$, on the $\theta-$dynamical ray of $S$. 
Note that if $\gamma_n$ is the sequence of arcs on this $\theta-$dynamical ray bounded by $w_n$ and $w_{n+1}$, then $S^{\circ k}(\gamma_n)=\gamma_{n-1}$, and $\ell_{\mathrm{hyp}}(\gamma_n)=c$ for all $n\geq1$, where $c$ is a fixed positive constant, and $\ell_{\mathrm{hyp}}$ denotes the hyperbolic length in $T^\infty(S)$. We now consider two cases.

\noindent\textbf{Case a.} (The accumulation set of the $\theta-$ray is contained in $\cup_{n\geq 0} S^{-n}(\mathfrak{S})$.) Since the accumulation set of a ray is connected, and the set of iterated preimages of $\mathfrak{S}$ is totally disconnected, it follows that the $\theta-$ray lands. 

\noindent\textbf{Case b.} (The accumulation set of the $\theta-$ray intersects $\Gamma(S)\setminus\cup_{n\geq 0} S^{-n}(\mathfrak{S})$.) Let $w\in \Gamma(S)\setminus\cup_{n\geq 0} S^{-n}(\mathfrak{S})$ be an accumulation point of the $\theta-$ray. By our assumption, the map $S^{\circ k}$ admits an (anti-)holomorphic extension in a neighborhood of $w$. Moreover, since the arcs $\gamma_n$ of fixed hyperbolic length accumulate on the limit set $\Gamma(S)$, their Euclidean lengths must shrink to $0$. One can now apply the arguments of \cite[Theorem~18.10]{Mil06} to the current setting, and deduce that $S^{\circ k}(w)=w$. But the set of fixed points of $S^{\circ k}$ in $\Gamma(S)$ is totally disconnected (as the set of fixed points of the holomorphic second iterate $S^{\circ 2k}$ in $\Gamma(S)$ is discrete). As before, connectedness of the accumulation set of a ray implies that the $\theta-$ray lands.

2) Let $\theta$ be fixed under $\pmb{\cN}_d$; i.e., $\theta=\frac{j}{d+1}$, for some $j\in\{0,\cdots,d\}$. Since $\partial T^0(S)$ is locally connected, the image of the radial line at angle $\theta$ under $\psi_S$ lands at some point $w\in\Lambda(S)$. The fact that $\pmb{\cN}_d$ does not admit an anti-conformal extension to a neighborhood of an ideal boundary point of $\Pi$ implies that $S$ does not have an anti-conformal extension to a neighborhood of $w$. Hence, $w\in\mathfrak{S}$. We will show that the $\theta-$dynamical ray of $S$ lands at~$w$.

By the asymptotic development of $S^{\circ 2}$ (which is given by parabolic germs if $w$ is a double point and by Equation~\eqref{asymp_eqn_1} if $w$ is a cusp), the map $S^{\circ 2}$ admits an inverse branch $g$ in a relative neighborhood of $w$ with the following properties:\\
i) $g(w)=w$, and\\
ii) for $w'\in S^{-2}(T^0(S))\cap\mathrm{Domain}(g)$ is sufficiently close to $w$, the sequence $g^{\circ n}(w')$ converges to $w$ horocyclically (in the simply connected domain $T^\infty(S)$).\\
It is easy to see that the sequence $g^{\circ n}(w')$ is within a bounded distance from the dynamical $\theta-$ray of $S$, and hence the $\theta-$ray lands at~$w$.

For the converse statement, note that each $w\in\mathfrak{S}$ is the landing point of the curve $\{\psi_S(re^{2\pi i\theta}):r\in[0,1)\}$, for some angle $\theta$ fixed under $\pmb{\cN}_d$. By the previous paragraph, the dynamical ray of $S$ at angle $\theta$ lands at $w$.
 
3) We claim that the dynamical rays of $S$ of period larger than two do not land on $\mathfrak{S}$. Once this claim is established, the result follow by the classical Snail lemma.

Let us fix $w\in\mathfrak{S}$. Once again, by the asymptotic development of $S^{\circ 2}$ near $w$, the map $S^{\circ 2}\vert_{S^{-2}(\cD)}$ extends to an orientation-preserving local homeomorphism in a neighborhood of $w$. One can apply the arguments of \cite[Lemma~2.3]{Mil00} to this map to conclude that all rays landing at $w$ have the same period under $S^{\circ 2}$ (cf. \cite[Proposition~5.44]{LLMM1}). Since $w$ is the landing point of some fixed ray of $S$ (by the previous part), it follows that all rays landing at $w$ are fixed by $S^{\circ 2}$. Therefore, the dynamical rays of $S$ of period larger than two cannot land on $\mathfrak{S}$
\end{proof}

\begin{prop}\label{rep_para_landing_point}
Let $S\in\mathscr{S}_{\pmb{\cN}_d}$. Then, every repelling and parabolic periodic point of $S$ is the landing point of finitely many (at least one) dynamical rays. Moreover, all these rays are periodic of the same period under $S^{\circ 2}$.
\end{prop}
\begin{proof}
The proof of the corresponding statement for polynomials with connected Julia set (cf. \cite[\S 24.3]{Lyu23}) can be easily adapted for the current setting.
\end{proof}

Using the same techniques as in \cite[Appendix B]{GM93}, we can prove stability of dynamical rays landing at pre-repelling points whose orbit avoids critical points.
\begin{prop}\label{rep_stab}
Let $S\in\mathscr{S}_{\pmb{\cN}_d}$. Assume that a dynamical ray lands at a repelling pre-periodic point $z$ such that the orbit of $z$ avoids the critical points of $S$.
Then, the corresponding dynamical ray lands at the real-analytic continuation $z'$ (of the pre-periodic point $z$) for all $S'\in\mathscr{S}_{\pmb{\cN}_d}$ sufficiently close to~$S$.
\end{prop}

In light of Proposition~\ref{per_rays_land}, one can define a landing map $L:\mathrm{Per}(\pmb{\cN}_d)\to \Lambda(S)$ that associates to every (pre)periodic angle $\theta$  the landing point of the $\theta-$dynamical ray of $S$.

\begin{defn}\label{def_preper_lami}
For $S\in\mathscr{S}_{\pmb{\cN}_d}$, the \emph{rational lamination} of $S$ is defined as the equivalence relation on $\mathrm{Per}(\pmb{\cN}_d)\subset\R/\Z$ such that $\theta, \theta'\in\mathrm{Per}(\pmb{\cN}_d)$ are related if and only if $L(\theta)=L(\theta')$. We denote the rational lamination of $S$ by $\lambda(S)$.
\end{defn}

Rational laminations of Schwarz reflections in $\mathscr{S}_{\pmb{\cN}_d}$ are analogous to rational laminations of anti-polynomials with connected Julia set (cf. \cite{Kiw01}). In particular, $\lambda(S)$ is a $\pmb{\cN}_d-$invariant equivalence relation.

\begin{defn}\label{push_lami}
The push-forward $(\pmb{\mathcal{E}}_d)_{\ast}(\lambda(S))$ of the rational lamination of $S$ is defined as the image of $\lambda(S)\subset\mathrm{Per}(\pmb{\cN}_d)\times\mathrm{Per}(\pmb{\cN}_d)$ under $\pmb{\mathcal{E}}_d\times\pmb{\mathcal{E}}_d$. Clearly, $(\pmb{\mathcal{E}}_d)_{\ast}(\lambda(S))$ is an equivalence relation on $\Q/\Z$. Similarly, the pullback $(\pmb{\mathcal{E}}_d)^{\ast}(\lambda(f))$ of the rational lamination of an anti-polynomial $f$ with connected Julia set is defined as the preimage of $\lambda(f)\subset\Q/\Z\times\Q/\Z$ under $\pmb{\mathcal{E}}_d\times\pmb{\mathcal{E}}_d$.
\end{defn}

For Schwarz reflections as well as anti-polynomials, we will denote by $\widehat{\lambda(S)}$ (respectively, $\widehat{\lambda(f)}$) the smallest equivalence relation on $\R/\Z$ that contains the closed set $\overline{\lambda(S)}\subset\R/\Z\times\R/\Z$ (respectively, $\overline{\lambda(f)}\subset\R/\Z\times\R/\Z$).
Note that the push-forward operation defined in Definition~\ref{push_lami} can be applied to $\widehat{\lambda(S)}$ as well.

\subsection{Combinatorial mating of $\pmb{\cN}_d$ with periodically repelling anti-polynomials}

Recall that the filled Julia set of a periodically repelling anti-polynomial $f\in\mathcal{C}_d^-$ has empty interior. By Proposition~\ref{fatou_comp_prop}, the same statement holds for the non-escaping set of a periodically repelling Schwarz reflection $S\in\mathscr{S}_{\pmb{\cN}_d}$.
The proof of \cite[Lemma~4.17]{Kiw01} implies that if the Julia/limit set of such a map is locally connected, then its filled Julia set/non-escaping set dynamics is topologically modeled by the induced map
$$
m_{-d}:\faktor{\R/\Z}{\widehat{\lambda(f)}}\longrightarrow \faktor{\R/\Z}{\widehat{\lambda(f)}}
$$
or
$$
\pmb{\cN}_d:\faktor{\R/\Z}{\widehat{\lambda(S)}}\longrightarrow \faktor{\R/\Z}{\widehat{\lambda(S)}}.
$$
Further, if 
$$
\left(\pmb{\mathcal{E}}_d\right)_*(\lambda(S))=\lambda(f),
$$
then the above induced maps are topologically conjugate via the homeomorphism
$$
\pmb{\mathcal{E}}_d: \faktor{\R/\Z}{\widehat{\lambda(S)}}\longrightarrow \faktor{\R/\Z}{\widehat{\lambda(f)}}
$$
induced by the conjugacy $\pmb{\mathcal{E}}_d:\mathbb{S}^1\to\mathbb{S}^1$ between $\pmb{\cN}_d$ and $\overline{z}^d$. In particular, local connectedness of the Julia (respectively, limit) set of $f$ (respectively, $S$) and the condition $\left(\pmb{\mathcal{E}}_d\right)_*(\lambda(S))=\lambda(f)$ together imply that $f\vert_{\mathcal{K}(f)}$ and $S_\vert{K(S)}$ are topologically conjugate.

The above discussion motivates the following weaker definition of mating.

\begin{defn}\label{comb_mating_def}
We say that a periodically repelling Schwarz reflection $S\in\mathscr{S}_{\pmb{\cN}_d}$ is a \emph{combinatorial mating} of $\pmb{\cN}_d$ with a periodically repelling anti-polynomial $f\in\mathcal{C}_d^-$ if $\left(\pmb{\mathcal{E}}_d\right)_*(\lambda(S))=\lambda(f)$.
\end{defn}

\section{Geometrically finite polygonal Schwarz reflections}\label{geom_fin_polygonal_schwarz_sec}

\begin{defn}\label{geom_fin_def}
A Schwarz reflection map $S\in\mathscr{S}_d$ is said to be \emph{geometrically finite} if every critical point of $S$ in the
limit set $\Lambda(S)$ is pre-periodic.
\end{defn}

By Propositions~\ref{fatou_comp_prop} and~\ref{fatou_crit_prop}, each periodic Fatou component of a geometrically finite Schwarz reflection is an attracting/parabolic basin, or a basin of attraction of some singular point $x\in\mathfrak{S}$. Moreover, $S$ has no Cremer cycle.

\begin{prop}\label{geom_fin_limit_lc_pop}
Let $S\in\mathscr{S}_{\pmb{\cN}_d,gf}$. Then the following hold true.
\noindent\begin{enumerate}\upshape
\item $\Lambda(S)$ is locally connected.
\item $\Lambda(S)$ has zero area.
\end{enumerate}
\end{prop}
\begin{proof}
1) The proof is analogous to the proof of local connectivity of connected Julia sets of geometrically finite polynomials (cf. \cite[\S X]{DH1}, \cite[Propositions~6.4, 6.9]{LLMM1}).

2) We say that a point $w\in \Gamma(S)$ is a \emph{radial point} if there exists $\delta > 0$, an infinite sequence of positive integers $\lbrace n_k \rbrace$, and well-defined inverse branches of $S^{\circ n_k}$ defined on $B(S^{\circ n_k}(w), \delta)$ sending $S^{\circ n_k}(w)$ to $w$ for all $k$. The set of all radial points of $\Gamma(S)$ is called the \emph{radial limit set} of $S$. Arguments used in \cite[Propositions~6.2,6.7]{LLMM1} show that the radial limit set of $S$ is given by $\displaystyle\Gamma(S)\setminus \bigcup_{j=0}^\infty S^{-j}\left(\mathfrak{S}\cup\mathcal{P}\cup C_{\mathrm{preper}}\right)$, where $\mathcal{P}$ is the union of the parabolic cycles and $C_{\mathrm{preper}}$ is the set of strictly pre-periodic critical points of $S$. Finally, the nowhere dense structure of $\Lambda(S)$ can be transferred from moderate scales to microscopic scales at radial points (using the inverse branches of iterates of $S$), and this implies that no radial limit point is a point of Lebesgue density for $\Gamma(S)$ (cf. \cite[Proposition~25.23]{Lyu23} or \cite[Theorem~3.8]{Urb1}). The above facts, together with the Lebesgue density theorem, imply that $\Lambda(S)$ has zero area. 
\end{proof}

Local connectedness of limit sets of maps in $\mathscr{S}_{\pmb{\cN}_d,gf}$, combined with the arguments of \cite[Lemma~4.17]{Kiw01}, implies that for $S\in\mathscr{S}_{\pmb{\cN}_d,gf}$, the dynamics of $S$ on its limit set $\Lambda(S)$ is topologically conjugate to 
$$
\pmb{\cN}_d: \faktor{(\R/\Z)}{\widehat{\lambda(S)}}\longrightarrow \faktor{(\R/\Z)}{\widehat{\lambda(S)}}\ .
$$
In particular, $\widehat{\lambda(S)}$ is generated by co-landing of the dynamical rays of $S$. A straightforward adaptation of \cite[Theorem~1]{Kiw04} for Schwarz reflections yields the following properties of $\widehat{\lambda(S)}$ for $S\in\mathscr{S}_{\pmb{\cN}_d,gf}$.

\begin{prop}\label{geom_finite_laminations_prop}
Let $S\in\mathscr{S}_{\pmb{\cN}_d,gf}$. Then, $\widehat{\lambda(S)}$ satisfies the following properties.
\begin{enumerate}\upshape
\item $\widehat{\lambda(S)}$ is closed in $\R/\Z\times\R/\Z$.

\item Each equivalence class $A$ of $\widehat{\lambda(S)}$ is a finite subset of $\R/\Z$.

\item If $A$ is a $\widehat{\lambda(S)}-$equivalence class, then $\pmb{\cN}_d(A)$ is also a $\widehat{\lambda(S)}-$equivalence class.

\item If $A$ is a $\widehat{\lambda(S)}-$equivalence class, then $A\mapsto \pmb{\cN}_d(A)$ is \emph{consecutive reversing}; i.e., for every connected component $(s,t)$ of $\R/\Z\setminus A$, we have that $(\pmb{\cN}_d(s),\pmb{\cN}_d(t))$ is a connected component of $\R/\Z\ \setminus\  \pmb{\cN}_d(A)$.

\item $\widehat{\lambda(S)}-$equivalence classes are pairwise \emph{unlinked}; i.e., if $A$ and $B$ are two distinct equivalence classes of $\widehat{\lambda(S)}$, then there exist disjoint intervals $I_A, I_B\subset\R/\Z$ such that $A\subset I_A$ and $B\subset I_B$.

\item If $\gamma\subset\faktor{(\R/\Z)}{\widehat{\lambda(S)}}$ is a periodic simple closed curve, then the corresponding return map is not a homeomorphism.

\item If $c$ is a critical point of $P$ on the limit set $\Lambda(S)$, then $\tau_S^{-1}(c)\subset\mathrm{Per}(\pmb{\cN}_d)$, and $\tau_S^{-1}(c)\mapsto \pmb{\cN}_d(\tau_S^{-1}(c))$ has some degree $\delta>1$. Every other equivalence class of $\widehat{\lambda(S)}$ maps injectively onto its image equivalence class under $\pmb{\cN}_d$. Here, $\tau_S:\D\to T^\infty(S)$ is the normalized conformal map that conjugates $\pmb{\cN}_d$ to $S$.
\end{enumerate}
\end{prop}

To prove Theorem~\ref{thm:C}, we need the following technical lemma that guarantees the existence of continuous extensions of circle homeomorphisms as quasiconformal/David homeomorphisms of $\D$.

\begin{prop}\label{extension_david_prop}
\noindent
\begin{enumerate}\upshape
\item The homeomorphism $\pmb{\mathcal{E}}_d^{-1}:\mathbb{S}^1\to\mathbb{S}^1$, that conjugates $\overline{z}^d$ to $\pmb{\cN}_d$, admits a continuous extension to $\D$ as a David homeomorphism.

\item Let $B:\D\to\D$ be a degree $d$ Blaschke (respectively, anti-Blaschke product) with a parabolic fixed point on $\mathbb{S}^1$ that has an attracting direction in $\D$. Then there exists a homeomorphism $H:\mathbb{S}^1\to\mathbb{S}^1$ that conjugates $z^d$ (respectively, $\overline{z}^d$) to $B$ and admits a continuous extension to $\D$ as a David homeomorphism.

\item Let $n\geq 2$, and 
$$
\pmb{B}:\mathbb{S}^1\times\Z/n\Z\longrightarrow\mathbb{S}^1\times\Z/n\Z,\ \pmb{B}(z,i)=(B_i(z),i+1)
$$
be such that each $B_i$ is a degree $d_i$ anti-Blaschke product. Assume further that each 
$$
B_{i+n-1}\circ\cdots\circ B_{i+1}\circ B_i:\mathbb{S}^1\to\mathbb{S}^1
$$ 
is a degree $d$ Blaschke or anti-Blaschke product with a parabolic fixed point on $\mathbb{S}^1$ that has an attracting direction in $\D$, depending on whether $n$ is even or odd, where $d=d_0 d_1\cdots d_{n-1}$, $i\in\Z/n\Z$.
Then, there exists a homeomorphism $\pmb{H}:\mathbb{S}^1\times\Z/n\Z\longrightarrow\mathbb{S}^1\times\Z/n\Z$ conjugating 
$$
\pmb{P}:\mathbb{S}^1\times\Z/n\Z\longrightarrow\mathbb{S}^1\times\Z/n\Z,\ \pmb{P}(z,i)=(\overline{z}^{d_i},i+1)
$$
to $\pmb{B}$.
Moreover, $\pmb{H}$ continuously extend as a David homeomorphism of $\D\times\Z/n\Z$.

\item Let $n\geq 2$, and 
$$
\pmb{B}:\mathbb{S}^1\times\Z/n\Z\longrightarrow\mathbb{S}^1\times\Z/n\Z,\ \pmb{B}(z,i)=(B_i(z),i+1)
$$
be such that each $B_i$ is a degree $d_i$ anti-Blaschke product. Assume further that each 
$$
B_{i+n-1}\circ\cdots\circ B_{i+1}\circ B_i:\mathbb{S}^1\to\mathbb{S}^1
$$ 
is a degree $d$ Blaschke or anti-Blaschke product with an attracting fixed point in $\D$, depending on whether $n$ is even or odd, where $d=d_0 d_1 \cdots d_{n-1}$, $i\in\Z/n\Z$.
Then, there exists a homeomorphism $\pmb{H}:\mathbb{S}^1\times\Z/n\Z\longrightarrow\mathbb{S}^1\times\Z/n\Z$ conjugating 
$$
\pmb{P}:\mathbb{S}^1\times\Z/n\Z\longrightarrow\mathbb{S}^1\times\Z/n\Z,\ \pmb{P}(z,i)=(\overline{z}^{d_i},i+1)
$$
to $\pmb{B}$.
Moreover, $\pmb{H}$ continuously extends as a quasiconformal homeomorphisms of $\D\times\Z/n\Z$.
\end{enumerate}
\end{prop}
\begin{proof}
1) Follows from \cite[Theorem~4.14]{LMMN}.

2) Note that the parabolic (anti-)Blaschke product $B_d^-(z) :=\frac{(d+1)\overline{z}^d+(d-1)}{(d+1)+(d-1)\overline{z}^d}$ (or the parabolic Blaschke product $B_d^+(z)=\frac{(d+1)z^d+(d-1)}{(d+1)+(d-1)z^d}$) is conjugate to $B$ via a quasisymmetric homeomorphism $H_1:\mathbb{S}^1\to\mathbb{S}^1$ (cf. \cite[Proposition~6.8]{McM}). We extend $H_1$ to a quasiconformal homeomorphism of $\widehat{\C}$ (also called $H_1$).

By \cite[Example~4.2, Theorem~4.13]{LMMN}, there exists a circle homeomorphism $H_2:\mathbb{S}^1\to\mathbb{S}^1$ that conjugates $z^d$ or $\overline{z}^d$ to $B_d^\pm$, such that $H_2$ admits a continuous extension to $\D$ as a David homeomorphism. By \cite[Proposition~2.5]{LMMN}, the desired map $H$ is given by $H_1\circ H_2$.

3) Let us assume that $n$ is odd. The proof of the even case is completely analogous.

By the previous part, there exists a topological conjugacy $H_{n-1}:\mathbb{S}^1\to\mathbb{S}^1$ between $\overline{z}^d$ and $B_{n-2}\circ\cdots\circ B_{0}\circ B_{n-1}$. Since $B_{j}$ and $\overline{z}^{j}$ are circle coverings of the same degree for $j\in\Z/n\Z$, we can succesively lift $H_n$ to obtain circle homeomorphisms $H_{n-2},\cdots,H_0, \widetilde{H}_{n-1}$ such that the following diagram commutes:

\[ \begin{tikzcd}
\mathbb{S}^1 \arrow{r}{\overline{z}^{d_{n-1}}} \arrow[swap]{d}{\widetilde{H}_{n-1}} & \mathbb{S}^1 \arrow{r}{\overline{z}^{d_0}} \arrow[swap]{d}{H_0} & \mathbb{S}^1 \arrow{d}{H_1} \arrow{r}{\overline{z}^{d_{1}}}  & \mathbb{S}^1 \arrow{d}{H_{2}} \arrow{r}{\overline{z}^{d_{2}}} & \cdots \arrow{r}{\overline{z}^{d_{n-3}}} \arrow{d}{} & \mathbb{S}^1 \arrow[swap]{d}{H_{n-2}} \arrow{r}{\overline{z}^{d_{n-2}}} & \mathbb{S}^1 \arrow{d}{H_{n-1}} \\
\mathbb{S}^1 \arrow[swap]{r}{B_{n-1}} & \mathbb{S}^1 \arrow[swap]{r}{B_{0}}& \mathbb{S}^1 \arrow[swap]{r}{B_{1}} & \mathbb{S}^1 \arrow{r}[swap]{B_{2}} & \cdots \arrow[swap]{r}{B_{n-3}} & \mathbb{S}^1 \arrow[swap]{r}{B_{n-2}} & \mathbb{S}^1.
\end{tikzcd}
\]
It now follows that
$$
\left(B_{n-2}\circ\cdots\circ B_{0}\circ B_{n-1}\right) \circ \widetilde{H}_{n-1} = \left(B_{n-2}\circ\cdots\circ B_{0}\circ B_{n-1}\right) \circ H_{n-1},
$$
and hence $\widetilde{H}_{n-1}\circ H_{n-1}^{-1}$ is a deck transformation for the circle covering $B_{n-2}\circ\cdots\circ B_{0}\circ B_{n-1}$. 
Hence, possibly after post-composing $H_{n-2},\cdots,H_0, \widetilde{H}_{n-1}$ with deck transformations of appropriate $B_j$s, we can assume that $\widetilde{H}_{n-1}=H_{n-1}$. Hence, 
$$
\pmb{H}:\mathbb{S}^1\times\Z/n\Z \to\mathbb{S}^1\times\Z/n\Z,\ (z,i)\mapsto (H_i(z), i)
$$
is the desired topological conjugacy between $\pmb{P}$ and $\pmb{B}$.

Finally, as each $H_i$ conjugates $\overline{z}^d$ to a degree $d$ anti-Blaschke product with a parabolic fixed point on $\mathbb{S}^1$ that has an attracting direction in $\D$, the first part (of this proposition) allows us to conclude that $H_i$ can be continuously extended to David homeomorphisms of~$\D$. 

4) The proof, which uses the fact that any expanding, analytic, orientation-reversing covering of $\mathbb{S}^1$ is quasisymmetrically conjugate to $\overline{z}^n$ (where $n$ is the degree of the map), is similar to the previous part.
\end{proof}

\begin{proof}[Proof of Theorem~\ref{thm:C} (bijection part)]
The proof will be divided into three parts.
\smallskip

\noindent\textbf{Constructing a map $\Phi:\mathcal{C}^-_{d,gf}\longrightarrow \mathscr{S}_{\pmb{\cN}_d,gf}$}. Let $f\in\mathcal{C}^-_{d,gf}$ be geometrically finite. Our goal is to construct a Schwarz reflection in $\mathscr{S}_{\pmb{\cN}_d,gf}$ that is a conformal mating of $f$ and $\pmb{\cN}_d$.

By an antiholomorphic version of \cite[Corollary~9.7]{LMMN} (the proof of the orientation-preserving case applies mutatis mutandis to the orientation-reversing case thanks to the orientation-reversing version of Thurston Realization Theorem, cf. \cite[Proposition~6.1]{LLMM4}, \cite[Theorem~3.9]{Gey20}), there exists a postcritically finite map $f_0\in\mathcal{C}^-_d$ such that $f\vert_{\mathcal{J}(f)}$ is topologically conjugate to $f_0\vert_{\mathcal{J}(f_0)}$. We first construct a partially defined continuous map on the sphere from $f_0$ as follows.
\smallskip

\noindent$\bullet$ Let us denote the David extension of $\pmb{\mathcal{E}}_d^{-1}:\mathbb{S}^1\to\mathbb{S}^1$ (that conjugates $\overline{z}^d$ to $\pmb{\cN}_d$) to $\D$ by $\xi$. Further, let $\tau_{f_0}:\D\to\mathcal{B}_\infty(f_0)$ be a conformal map that conjugates $\overline{z}^d$ to $f_0$ (where $\mathcal{B}_\infty(f_0)$ is the basin of infinity of $f_0$).  We replace the action of $f_0$ on $\mathcal{B}_\infty(f_0)$ by $\tau_{f_0}\circ \xi^{-1}\circ\pmb{\cN}_d\circ \xi\circ\tau_{f_0}^{-1}$. 

Note that the map $\tau_{f_0}\circ \xi^{-1}\circ\pmb{\cN}_d\circ \xi\circ\tau_{f_0}^{-1}$ preserves the pullback of the standard complex structure on $\D$ by $\xi\circ\tau_{f_0}^{-1}$.
\smallskip

\noindent$\bullet$ Suppose that $U_0,\cdots,U_{n-1}$ is a cycle of bounded Fatou components of $f$, and $\widehat{U}_0,\cdots,\widehat{U}_{n-1}$ is the corresponding cycle of bounded Fatou components of $f_0$. One can choose conformal maps $\pmb{\kappa}:\D\times\Z/n\Z\to\bigsqcup_{i\in\Z/n\Z} U_i$ and $\widehat{\pmb{\kappa}}:\D\times\Z/n\Z\to\bigsqcup_{i\in\Z/n\Z} \widehat{U}_i$ such that $\pmb{\kappa}^{-1}\circ f\circ\pmb{\kappa}:\D\times\Z/n\Z\to\D\times\Z/n\Z$ is a fiberwise anti-Blaschke product $\pmb{B}$ satisfying the hypotheses of part (3) or part (4) of Proposition~\ref{extension_david_prop}, and $\widehat{\pmb{\kappa}}^{-1}\circ f_0\circ\widehat{\pmb{\kappa}}:\D\times\Z/n\Z\to\D\times\Z/n\Z$ is a fiberwise power map $\pmb{P}$ as in part (3) or part (4) of Proposition~\ref{extension_david_prop}. Hence, there exists a homeomorphism $\pmb{H}:\mathbb{S}^1\times\Z/n\Z\longrightarrow\mathbb{S}^1\times\Z/n\Z$ that conjugates $\pmb{P}$ to $\pmb{B}$, and admits a David/quasiconformal extension $\D\times\Z/n\Z$.
We replace the action of $f_0$ on $\bigsqcup_{i\in\Z/n\Z} \widehat{U}_i$ by the map $\widehat{\pmb{\kappa}}\circ\pmb{H}^{-1}\circ\pmb{B}\circ\pmb{H}\circ\widehat{\pmb{\kappa}}^{-1}$.

Again, the map $\widehat{\pmb{\kappa}}\circ\pmb{H}^{-1}\circ\pmb{B}\circ\pmb{H}\circ\widehat{\pmb{\kappa}}^{-1}$ preserves the pullback of the standard complex structure on $\D\times\Z/n\Z$ by $\pmb{H}\circ\widehat{\pmb{\kappa}}^{-1}$.

\smallskip

\noindent$\bullet$ Finally, let $U$ be a strictly pre-periodic Fatou component of $f$ with $f(U)=V$, and $\widehat{U}, \widehat{V}$ be the corresponding components of $f_0$. The maps $f:U\to V$ and $f_0:\widehat{U}\to\widehat{V}$ induce, via Riemann maps of the Fatou components, anti-Blaschke actions of $\D$, and these anti-Blaschke actions are quasisymmetrically conjugate on $\mathbb{S}^1$ (cf. \cite[Theorem~6.1]{McM}). This allows us to replace the action of $f_0$ on $\widehat{U}$ by the conformal model of the action of $f$ on $U$ as in the previous steps. We note that this step is only necessary for Fatou components $U$ containing a critical point.

The above recipe defines a continuous map $\widecheck{S}$ on a subset of the sphere and an invariant Beltrami coefficient on the union of the periodic Fatou components of $f_0$. We use $\widecheck{S}$ to spread this Beltrami coefficient out to all the pre-periodic Fatou components of $f_0$, and set it equal to zero on the Julia set of $f_0$. Since $f_0$ is subhyperbolic, one can now argue as in \cite[Lemma~7.1]{LMMN} to conclude that this invariant Beltrami coefficient $\mu$ satiesfies the David condition and that the conjugate of $\widecheck{S}$ by the corresponding integrating map $\Psi$ is antiholomorphic. We denote the conjugated map $\Psi\circ\widecheck{S}\circ \Psi^{-1}$ by $S$.

By construction, $S$ is defined on the complement of the simply connected domain $\Psi(\tau_{f_0}(\xi^{-1}(\Int{\Pi})))$, and fixes the boundary of its domain of definition pointwise. Thus, $S$ is a Schwarz reflection map whose escaping and non-escaping sets are given by $\Psi(\mathcal{B}_\infty(f_0)$ and $\Psi(K(f_0))$, respectively. The tiling set dynamics of $S$ is conformally conjugate to $\pmb{\cN}_d$ via $\xi\circ\tau_{f_0}^{-1}\circ\Psi^{-1}$; i.e., $S\in\mathscr{S}_{\pmb{\cN}_d,gf}$. It also follows from the construction of $S$ that its non-escaping set dynamics is topologically conjugate (conformally on the interior) to the filled Julia set dynamics of $f$. It is now easily checked that $S$ is a conformal mating of $f$ and $\pmb{\cN}_d$.

We also note that by \cite[Theorem~2.12]{LMMN} and the fact that $\mathcal{B}_\infty(f_0)$ is  John domain, we have that the limit set of $S$ is conformally removable.

We claim that such an $S\in\mathscr{S}_{\pmb{\cN}_d,gf}$ is unique. Indeed, if $S_1$ were another conformal mating of $f$ and $\pmb{\cN}_d$, then the tiling set dynamics and non-escaping set dynamics of $S$ and $S_1$ are conformally conjugate. Further, these conjugacies match continuously along the limit sets to yield a topological conjugacy between $S$ and $S_1$ that is conformal away from $\Lambda(S)$. By conformal removability of $\Lambda(S)$, this conjugacy is a M{\"o}bius map; i.e., $S$ and $S_1$ are M{\"o}bius conjugate. The normalization of the maps in $\mathscr{S}_{\pmb{\cN}_d}$ now implies that $S=S_1$. 
Alternatively, one can employ classical pullback arguments to deduce uniqueness of the map $S$ (see \cite[\S 8.2]{LLMM2} for an implementation of this principle to prove rigidity of geometrically finite Schwarz reflections in the C\&C family).

Thus, we have a well-defined map $\Phi:\mathcal{C}^-_{d,gf}\longrightarrow\mathscr{S}_{\pmb{\cN}_d,gf}$ that sends $f$ to the unique conformal mating $S$ of $f$ and $\pmb{\cN}_d$.
\smallskip

\noindent\textbf{Injectivity of $\Phi$.} Suppose that $\Phi(f)=\Phi(f_1)=S$. Then, the basin of infinity dynamics and filled Julia set dynamics of $f$ and $f_1$ are conformally conjugate. Further, these conjugacies match continuously along the Julia sets to yield a topological conjugacy between $f$ and $f_1$ that is conformal away from $\mathcal{J}(f)$. By conformal removability of $\mathcal{J}(f)$ (see \cite[Theorem~9.2]{LMMN}), this conjugacy is a M{\"o}bius map. (Equivalently, injectivity of $\Phi$ follows from well-known results about rigidity of geometrically finite polynomials, cf. \cite{McM,SulMc}.) By the normalization of B{\"o}ttcher coordinates for maps in $\mathcal{C}^-_d$, this M{\"o}bius conjugacy between the monic, centered anti-poynomials $f, f_1$ is tangent to the identity at $\infty$. Hence, the M{\"o}bius conjugacy must be the identity map; i.e., $f=f_1$.
\smallskip

\noindent\textbf{Surjectivity of $\Phi$.} Let $S\in\mathscr{S}_{\pmb{\cN}_d,gf}$ and $\lambda(S)$ be the rational lamination of $S$. By Proposition~\ref{geom_finite_laminations_prop} and the orientation-reversing version of \cite[Theorem~9.6]{LMMN}, there exists a postcritically finite anti-polynomial $f_0\in\mathcal{C}^-_d$ with $\widehat{\lambda(f_0)}= (\pmb{\mathcal{E}}_d)_{\ast}(\widehat{\lambda(S))}$. Having $f_0$ at our disposal, we can now produce a geometrically finite anti-polynomial $f\in\mathcal{C}^-_d$ by performing quasiconformal/David surgeries on the bounded Fatou components of $f_0$ (as in the first step of the proof) so that the filled Julia set dynamics of $f$ is topologically conjugate (conformally on the interior) to the non-escaping set dynamics of $S$. Evidently, the map $f$ is the desired preimage of $S$ under $\Phi$.
\end{proof}




\section{Strata of the connectedness locus $\mathscr{S}_{\pmb{\cN}_d}$}

\subsection{Fixed ray lamination and a coarse partition}\label{partition_sec}

In this section, we will explain that the connectedness locus of the space of degree $d$ regular polygonal Schwarz reflections admits a finite partition.
This coarse partition gives a natural stratification of the connectedness locus.

\subsection*{Fixed ray lamination}

The Nielsen map $\pmb{\cN}_d:\mathbb{S}^1\to\mathbb{S}^1$ has exactly $d+1$ fixed points $0,\cdots,\frac{d}{d+1}$. We denote the set of these fixed points by $\mathrm{Fix}(\pmb{\cN}_d)$.


We define a {\em fixed ray lamination} as a geodesic lamination $\mathcal{L} \subset \D$ so that
\begin{itemize}
\item each leaf $\ell = (a,b)$ of $\mathcal{L}$ has end points $a, b \in \mathrm{Fix}(\pmb{\cN}_d) \subset \T \cong \partial \D$;
\item two leaves in $\mathcal{L}$ have disjoint closures in $\overline{\D}$.
\end{itemize}

Given a fixed ray lamination $\mathcal{L}$, we call each complementary component $G \subset \D\setminus \mathcal{L}$ a {\em gap}.
Let $G$ be a gap. Then $\partial G \cap \T$ has length equal to $k \cdot \frac{1}{d+1}$ for some multiple $k$.
We shall call this number $k$ the {\em degree} of the gap, and is denoted by $\deg(G)$.
From the definition, we have
$$
d+1 = \sum_{G} \deg(G),
$$
where the sum is over all gaps of $\mathcal{L}$.

Let us denote the number of fixed points of $\pmb{\cN}_d$ in $\Int(\partial G \cap \T)$ by $\cusp(G)$, and call it the {\em cusp number} of the gap.
We also denote the number of leaves on $\partial G$ by $\tan(G)$, and call it the {\em tangent number} of the gap.
An easy computation shows that
$$
\deg(G) = \cusp(G) + \tan(G).
$$

Let $G_1, G_2$ be two gaps.
We say that they are {\em adjacent} if $\partial G_1 \cap \partial G_2 \neq \emptyset$.
Note that if two gaps are adjacent, then $\partial G_1 \cap \partial G_2$ is a leaf of $\mathcal{L}$.
It is easy to see that there are exactly $\tan(G)$ number of gaps adjacent to $G$.

\subsection*{Fixed ray laminations associated with quadrature domains}

Recall that a polygonal Schwarz reflection map $S\in\mathscr{S}_{\pmb{\cN}_d}$ (normalized as in Section~\ref{subsec:bdp}) induces the action of the Nielsen map $\pmb{\cN}_d$ on the ideal boundary $I(T^\infty(\sigma))\cong \R/\Z$ of the escaping set, and hence has exactly $d+1$ fixed points on its ideal boundary.

Let $\mathcal{L}$ be a fixed ray lamination.
We say it is associated with a marked degree $d$ polygonal Schwarz reflection $S$ if
\begin{itemize}
\item $\ell = (a, b)$ is a leaf of $\mathcal{L}$ if and only $a, b \in I(T^\infty(S))$ correspond to the same point on $\partial T^\infty(S)$.
\end{itemize}

Recall that $\mathscr{S}_{\pmb{\cN}_d}$ is the connectedness locus of degree $d$ regular polygonal Schwarz reflections. As mentioned in Subsection~\ref{subsec:bdp}, the domain of a Schwarz reflection $S\in\mathscr{S}_{\pmb{\cN}_d}$ is a pinched polygon, and each cut-point of this pinched polygon is the touching point of precisely two Jordan quadrature domains defining $S$. 
Thus, we have a coarse partition
$$
\mathscr{S}_{\pmb{\cN}_d} = \bigcup_{\mathcal{L}} \mathscr{S}_{\pmb{\cN}_d, \mathcal{L}},
$$
where $\mathscr{S}_{\pmb{\cN}_d, \mathcal{L}}$ consists of all marked degree $d$ polygonal Schwarz reflections with connected limit set and fixed ray lamination $\mathcal{L}$, and the union is over all possible fixed ray laminations.

The following theorem gives the topological structure of the quadrature domains for $S \in \mathscr{S}_{\pmb{\cN}_d, \mathcal{L}}$.
\begin{theorem}[Topological structure of the quadrature domain]
Let $\mathcal{L}$ be a fixed ray lamination. Let $S : \overline{\cD} \longrightarrow \hat\C \in \mathscr{S}_{\pmb{\cN}_d, \mathcal{L}}$.
Then 
\begin{enumerate}[label*=\arabic*)]
\item the components of $\cD$ are in one-to-one correspondence with the gaps of $\mathcal{L}$:
$$
\cD = \bigcup_{G} \Omega_G;
$$
\item two quadrature domains $\Omega_G, \Omega_{G'}$ share at most one boundary point; moreover, they share a boundary point if and only if $G, G'$ are adjacent;
\item the touching point of two quadrature domains $\Omega_G, \Omega_{G'}$ is a non-singular point for both $\partial\Omega_G$ and $\partial\Omega_{G'}$ (i.e., it is a double point of~$\partial\cD$);
\item  the quadrature domain $\Omega_G$ has $\cusp(G)$ number of cusps on its boundary; and
\item the uniformizing rational map of the quadrature domain $\Omega_G$ has global degree $\deg(G)$.
\end{enumerate}
\end{theorem}
\begin{proof} 
Let $\psi:\D\to T^\infty(S)$ be a conformal conjugacy between $\pmb{\cN}_d$ and $S$. Also recall that $\Pi$ denotes the ideal $(d+1)-$gon in $\D$ with ideal vertices at the $(d+1)-$st roots of unity.
By definition, $\cD=\widehat{\C}\setminus\psi(\overline{\Pi})$, where the closure is taken in $\overline{\D}$. By injectivity of $\psi$ on $\D$, the only points on $\overline{\Pi}$ that may be identified by $\psi$ are the $(d+1)-$st roots of unity. By design, each such identification is recorded by a leaf of the fixed ray lamination of $\mathcal{L}$. 

1) and 2) It follows from the above discussion that the cut-points of the pinched polygon $\overline{\cD}$ correspond bijectively to the leaves of $\cL$ and hence the components of $\cD$ correspond bijectively to the gaps of $\cL$. The touching structure of the components of $\cD$ also follow from above.

3) As each $\Omega_G$ is a Jordan quadrature domain, the only singularities on $\partial\Omega_G$ are conformal cusps. Let $\phi_G:\D\to\Omega_G$ be a uniformizing rational map. Then, each conformal cusp on $\partial\Omega_G$ is a critical value of $\phi_G$ with an associated critical point on $\mathbb{S}^1$. By the injectivity of $\phi_G$ on $\D$, such a cusp is an inward pointing cusp for $\Omega_G$. This shows that the touching point of $\Omega_G$ and $\Omega_{G'}$ cannot be a conformal cusp for the boundary of either domain. The result follows.

4) Note that $\pmb{\cN}_d$ does not admit an anti-conformal extension in a neighborhood of any ideal vertex of $\Pi$, while it does admit such extensions around all other boundary points of $\Pi$. It follows that $S$ does not admit an anti-conformal extension in a neighborhood of any point of $\psi(\omega^j)$ (where $\omega=\exp{(\frac{2\pi i}{d+1})}$, and $j\in\{0,\cdots,d\}$), but it does extend anti-conformally around other boundary points of $\partial\cD$. The only obstructions for an anti-conformal extension of $S$ across a boundary point of $\cD$ is a cusp or a double point. By part 3) of this theorem, the pairs of end-points of the leaves of $\cL$ are in one-to-one correspondence with the double points on $\partial\cD$. Hence, the other ideal vertices of $\Pi$ correspond bijectively to the cusps on $\partial\cD$. The result is now a consequence of the above analysis and part (1) of this theorem.

5) The map $\pmb{\cN}_d$ carries each arc of $\mathbb{S}^1\setminus\mathrm{Fix}(\pmb{\cN}_d)$ to the complement of its closure. Hence, $\pmb{\cN}_d(\partial G\cap\mathbb{S}^1)$ covers $\partial G\cap\mathbb{S}^1$ exactly $(\deg(G)-1)$ times. It follows that $\Lambda(S)\cap\Omega_G$ covers itself $(\deg(G)-1)$ times under the map $S\vert_{\Omega_G}$. Since the limit set is completely invariant, we conclude that $S\vert_{\Omega_G}:\left(S\vert_{\Omega_G}\right)^{-1}(\Omega_G)\to\Omega_G$ is a degree $(\deg(G)-1)$ branched covering. This observation and Proposition~\ref{simp_conn_quad_prop} together imply that the degree of the uniformizing rational map of $\Omega_G$ is equal to $\deg(G)$.
\end{proof}


\subsection*{Coarse partitions in low degree}
There are four fixed ray laminations in degree $2$: the trivial lamination $\cL_0$, and the three symmetric laminations $\cL_1:=\{\{1/3,2/3\}\}$, $\cL_2:=\{\{0,1/3\}\}$, $\cL_3:=\{\{0,2/3\}\}$. The corresponding space $\mathscr{S}_{\pmb{\cN}_2, \mathcal{L}_1}$ (respectively, $\mathscr{S}_{\pmb{\cN}_2, \mathcal{L}_0}$) is given by Schwarz reflections in a fixed cardioid and its circumcircles (respectively, by Schwarz reflections in deltoid-like curves). These families were studied in \cite{LLMM1,LLMM2}.

\subsection{Compactness of connectedness locus $\mathscr{S}_{\pmb{\cN}_d}$}\label{compact_sec}

Recall that $\mathscr{S}_{\pmb{\cN}_d}$ stands for the space of degree $d$ normalized Schwarz reflection maps whose tiling set dynamics are conformally conjugate to the Nielsen map $\pmb{\cN}_d$ of the regular ideal polygon reflection group $\pmb{G}_d$. The domain of definition of such a Schwarz reflection is the closure of a disjoint collection of simply connected quadrature domains $\Omega_1,\cdots, \Omega_k$ such that
$\cup_{i=1}^k \overline{\Omega_i}$ is connected and simply connected.

Since there are only finitely many fixed ray laminations, we can assume (after possibly passing to a subsequence) that any sequence in $\mathscr{S}_{\pmb{\cN}_d}$ entirely lies in $\mathscr{S}_{\pmb{\cN}_d, \cL}$, where $\cL$ is a particular fixed ray lamination. We enumerate the gaps of $\cL$ as $\mathcal{G}_1,\cdots,\mathcal{G}_k$.
For maps $S:\overline{\cD}\to\widehat{\C}$ in $\mathscr{S}_{\pmb{\cN}_d, \cL}$, we denote the components of $\cD$ as $\Omega_1,\cdots,\Omega_k$ such that $\Omega_i$ corresponds to $\mathcal{G}_i$. 

We topologize $\mathscr{S}_{\pmb{\cN}_d}$ as follows.
\begin{defn}\label{topo_def_1}
We say that a sequence $\displaystyle\{S_n:\overline{\cD^n}=\bigcup_{r=1}^k\overline{\Omega_{r}^{n}}\to\widehat{\C}\}\subset\mathscr{S}_{\pmb{\cN}_d, \cL}$ converges to $\displaystyle S:\overline{\cD}=\bigcup_{r=1}^{k}\bigcup_{j=1}^{l_r}\overline{\Omega_{r,j}}\to\widehat{\C}$ if 
\begin{enumerate}
\item $\{\Omega_{r,j}:\ j\in\{1,\cdots,l_r\} \}$ is the collection of all Carath{\'e}odory limits of the sequence of domains $\{\Omega_{r}^{n}\}_n$, $r\in\{1,\cdots,k\}$, and 

\item for $n$ large enough, the antiholomorphic maps $S_n$ converge uniformly to $S$ on compact subsets of $\cD$.
\end{enumerate}
\end{defn}

\begin{remark}
For each Schwarz reflection $\displaystyle S_n:\overline{\cD^n}=\bigcup_{r=1}^k\overline{\Omega_{r}^{n}}\to\widehat{\C}$, there are $k$ Riemann uniformization maps $\phi_{n,r}: \D \longrightarrow \Omega_{r}^{n}, r = 1,..., k$ which extend as rational maps of $\widehat\C$ (see \S \ref{subsec:qmd}).
For $S:\overline{\cD}=\bigcup_{r=1}^{k}\bigcup_{j=1}^{l_r}\overline{\Omega_{r,j}}\to\widehat{\C}$, there are $\sum_{r=1}^k l_r$ Riemann uniformization maps $\phi_{r, j}: \D \longrightarrow \Omega_{r, j}$ which extend as rationals maps of $\widehat\C$.
It follows from the definition of Carath{\'e}odory limits that $S_n \to S$ if and only if there exist $M_{n, r, j} \in \Aut(\D)$ so that
$$
\phi_{n,r} \circ M_{n, r, j} \to \phi_{r, j}
$$
compactly away from finitely many points in $\widehat\C$ (where each map is viewed as maps on $\widehat\C$).
\end{remark}

\begin{theorem}\label{thm:compact}
The connectedness locus $\mathscr{S}_{\pmb{\cN}_d}$ is compact.
\end{theorem}

\begin{proof}
Let $\{S_n:\overline{\cD^n}=\cup_{r=1}^k\overline{\Omega_{r}^{n}}\to\widehat{\C}\}\subset\mathscr{S}_{\pmb{\cN}_d, \cL}$. By definition, there exist conformal maps
$$
\psi_n:\D\rightarrow T^\infty(S_n)
$$
that conjugate $\pmb{\cN}_d$ to $S$ (wherever defined). We can normalize $\psi_n$, possibly after conjugating $S_n$ by M{\"o}bius maps, such that $\psi_n(0)=\infty$, and $\psi_n(z)=\frac1z+O(z)$ near the origin. 

Note that $\psi_n(\Pi)=T^0(S_n)$. Since $\Pi$ contains a round disk around the origin, it follows from the Koebe one-quarter theorem that each $\psi_n(\Pi)$ contains a definite round disk centered at $\infty$. In other words, $\overline{\cD_n}\subset B(0,R)$, for some fixed $R>0$. 
Also note that the maps $\psi_n\vert_{\D}$ form a normal family. After possibly passing to a subsequence, we may assume that $\psi_n$ converges normally to a conformal map $\psi_\infty$ on $\D$. Let $B_s$ be an open ball compactly contained in $\rho_s(\Pi)$ centered at $\rho_s(0)$, for $s\in\{1,\dots,d+1\}$. Normal convergence of $\psi_n$ to $\psi_\infty$ implies that for $n$ large, the conformal disks $\Omega_{r}^{n}$ contain some fixed open set $\psi_\infty(B_s)$ (see Figure~\ref{compactness_fig}). It follows that for each $r\in\{1,\cdots,k\}$, the sequence of conformal disks $\{\Omega_{r}^{n}\}_n$ has at least one non-trivial Carath{\'e}odory limit. We will now argue that there are finitely many such Carath{\'e}odory limits.
\begin{figure}[ht]
\captionsetup{width=0.96\linewidth}
\begin{tikzpicture}
\node[anchor=south west,inner sep=0] at (0,0) {\includegraphics[width=0.975\textwidth]{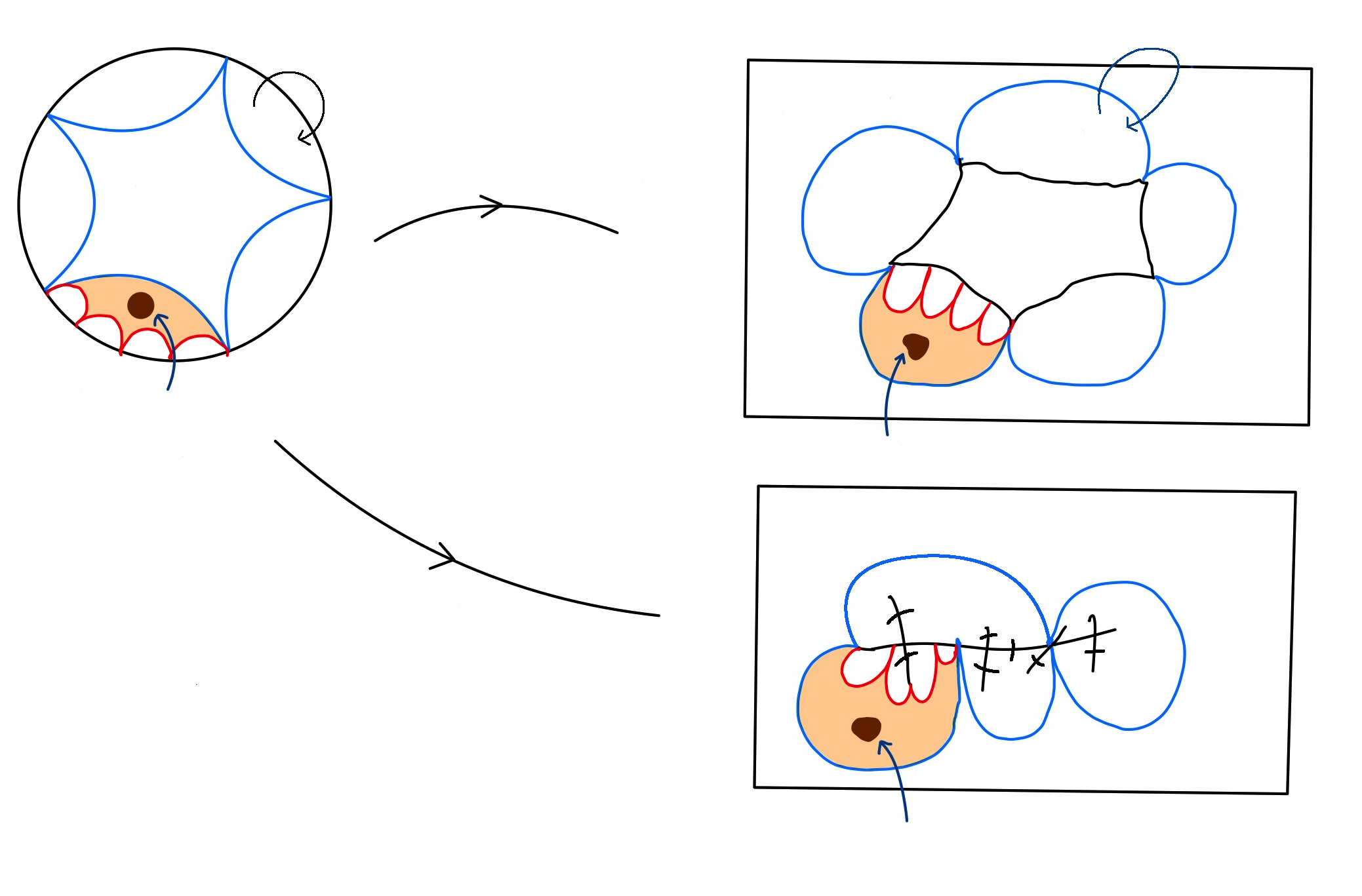}};
\node at (1.5,6.24) {\begin{huge}$\Pi$\end{huge}};
\node at (0.6,4.8) {\begin{Large}$\D$\end{Large}};
\node at (1.6,4.28) {\begin{large}$B_s$\end{large}};
\node at (3.2,7.6) {\begin{large}$\pmb{\cN}_d$\end{large}};
\node at (4,2.64) {\begin{Large}$\psi_\infty$\end{Large}};
\node at (4.4,5.7) {\begin{large}$\psi_n(0)=\infty$\end{large}};
\node at (4.75,6.56) {$\psi_n:\D\to T^\infty(S_n)$};
\node at (8.1,0.4) {\begin{large}$\psi_\infty(B_S)$\end{large}};
\node at (10.8,3.2) {\begin{normalsize}$\psi_\infty(\Pi)$\end{normalsize}};
\node at (10.5,8) {\begin{large}$S_n$\end{large}};
\node at (7.42,7.2) {\begin{normalsize}$\psi_n(\Pi)$\end{normalsize}};
\node at (8.16,4) {\begin{normalsize}$\psi_n(B_s)$\end{normalsize}};
\node at (9.28,5.92) {\begin{large}$K(S_n)$\end{large}};
\end{tikzpicture}
\caption{Illustrated is the proof of Theorem~\ref{thm:compact}.}
\label{compactness_fig}
\end{figure}
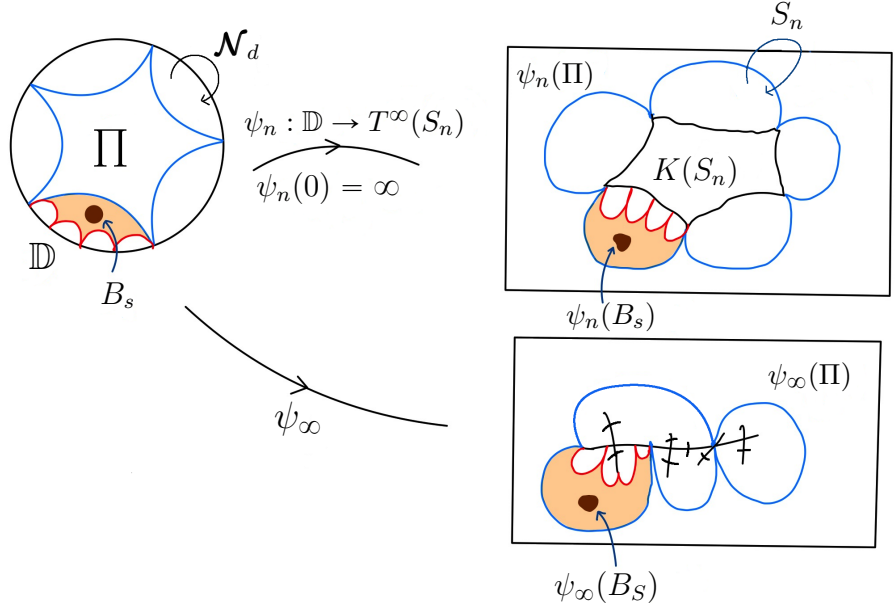

By conformality of $\psi_\infty$, the set $\psi_\infty(\partial\Pi\cap\D)$ is a union of $d+1$ disjoint bi-infinite geodesics in the hyperbolic metric of $\psi_\infty(\D)$. Since each $\psi_n(\partial\Pi)=\partial\cD_n$ is a finite union of real-algebraic curves of uniformly bounded degree contained in the fixed compact set $\overline{B}(0,R)$, one concludes that the ends of the above bi-infinite geodesics have unique limit points on $\partial\psi_\infty(\D)$. Thus, $\psi_\infty$ extends continuously to $\partial\Pi$ and $\psi_\infty(\overline{\Pi})$ has only finitely many identifications on its boundary (i.e., it is a topological polygon possibly with some vertices identified). Specifically, only the ideal vertices of $\Pi$ may be identified under $\psi_\infty$.

It follows from the above discussion that $\widehat{\C}\setminus\psi_\infty(\overline{\Pi})$ has finitely many (at least one) components, and these are precisely the Carath{\'e}odory limits of the sequences of domains $\{\Omega_{r}^{n}\}_n$, $r\in\{1,\cdots,k\}$ (with all possible choices of base points). We denote these components as $\Omega_1,\cdots,\Omega_l$.
As the Riemann uniformizations of $\Omega_{r}^{n}$ are rational maps, it follows from the Carath{\'e}odory kernel convergence theorem that the Riemann uniformization of each $\Omega_j$, $j\in\{1,\cdots,l\}$, is also a rational map. Thus, each such limit is a simply connected quadrature domain.

The description of Schwarz reflection maps given in Proposition~\ref{simp_conn_quad_prop} now implies that for $n$ large enough, $S_n$ converges to the Schwarz reflection map $S$ associated with the quadrature domains $\Omega_1,\cdots,\Omega_l$. We also note that 
$T^0(S)$ is equal to $\psi_\infty(\Pi)$. The equivariance property of $\psi_n$ implies that $\psi_\infty$ conjugates $\pmb{\cN}_d$ to $S$, wherever defined. Hence, the escaping set of $S$ is $\psi_\infty(\D)$, and the dynamics of $S$ on its escaping set is conformally conjugate to $\pmb{\cN}_d$ via $\psi_\infty^{-1}$. Hence, $S:\overline{\cD}\to\widehat{\C}$ lies in $\mathscr{S}_{\pmb{\cN}_d}$ and is a limit point of the sequence $\{S_n:\overline{\cD_{n}}\to\widehat{\C}\}$.
\end{proof}


Since fixed ray laminations are equivalence relations on $\mathbb{S}^1$, there is a natural partial order $\geq$ on them. 

\begin{cor}\label{limit_dominates_cor}
If $\{S_n\}\subset\mathscr{S}_{\pmb{\cN}_d, \cL_1}$ converges to some $S\in\mathscr{S}_{\pmb{\cN}_d, \cL_2}$, then $\cL_2\geq \cL_1$.
\end{cor}
\begin{proof}
If $\psi_n(a)=\psi_n(b)$ for some $a,b\in\partial\Pi\cap\mathbb{S}^1$ and for all $n\in\N$, then $\psi_\infty(a)=\psi_\infty(b)$.
\end{proof}

\begin{remark}
By using quasi-PCF degeneration studied in \cite{L23}, one can show that the converse of Corollary \ref{limit_dominates_cor} is also true.
Suppose $\cL_2\geq \cL_1$, then there exists a sequence $\{S_n\}\subset\mathscr{S}_{\pmb{\cN}_d, \cL_1}$ that converges to some $S\in\mathscr{S}_{\pmb{\cN}_d, \cL_2}$.
Some explicit examples are constructed in section \S \ref{discont_sec}, which are the key ingredients for proving discontinuity.
\end{remark}

\subsection{Discontinuity and change of signature}\label{discont_sec}

Recall from Section~\ref{geom_fin_polygonal_schwarz_sec} that there exists a natural bijective map $\Phi$ from geometrically finite polynomials $\mathcal{C}^-_{d,gf}$ onto geometrically finite Schwarz reflections $\mathscr{S}_{\pmb{\cN}_d,gf}$.

Let $H_0$ and $H_1$ stand for the main hyperbolic component and the period two hyperbolic component contained in the $(0,\frac{1}{d+1})-$limb of the Multicorn $\mathcal{M}_d$, which is the connectedness locus of unicritical anti-polynomials $f_c(z)=\overline{z}^d+c$. Then, $H_1$ bifurcates from $H_0$ along an arc, and barring a unique double parabolic parameter, all points on this arc are simple parabolic parameters (see \cite[Theorem~3.6]{HS}, \cite[\S 5]{MNS}, and \cite[\S 2.2]{IM21}).

We denote the trivial fixed ray lamination by $\cL_0$, and the fixed ray lamination $\{\{0,\frac{1}{d+1}\}\}$ by $\cL_1$.

\begin{lem}\label{discont_lem_1}
Let $\pmb{c_\infty}$ be a simple parabolic parameter on $\partial H_0\cap\partial H_1$. Then $\Phi$ is discontinuous at $\pmb{c_\infty}$.
\end{lem}
\begin{proof}
Suppose that $\{c_n\}\subset H_1$ converges to $\pmb{c_\infty}\in \partial H_0\cap\partial H_1$.
By our assumption, the dynamical rays at angles $0$ and $\frac{1}{d+1}$ land at a common point of the Julia set of $f_{c_n}$. Hence, $\Phi(f_{c_n})$ lies in $\mathscr{S}_{\pmb{\cN}_d, \cL_1}$. By Corollary~\ref{limit_dominates_cor} and unicriticality of the maps, any subsequential limit of $\{\Phi(f_{c_n})\}$ also lies in~$\mathscr{S}_{\pmb{\cN}_d, \cL_1}$.

On the other hand, as $f_{\pmb{c_\infty}}$ has a simple parabolic fixed point and hence a Jordan curve Julia set, it follows that $\Phi(f_{\pmb{c_\infty}})$ lies in $\mathscr{S}_{\pmb{\cN}_d, \cL_0}$. Hence, the sequence $\{\Phi(f_{c_n})\}$ does not converge to $\Phi(f_{\pmb{c_\infty}})$.
\end{proof}

\begin{figure}[ht]
\captionsetup{width=0.96\linewidth}
\includegraphics[width=0.45\linewidth]{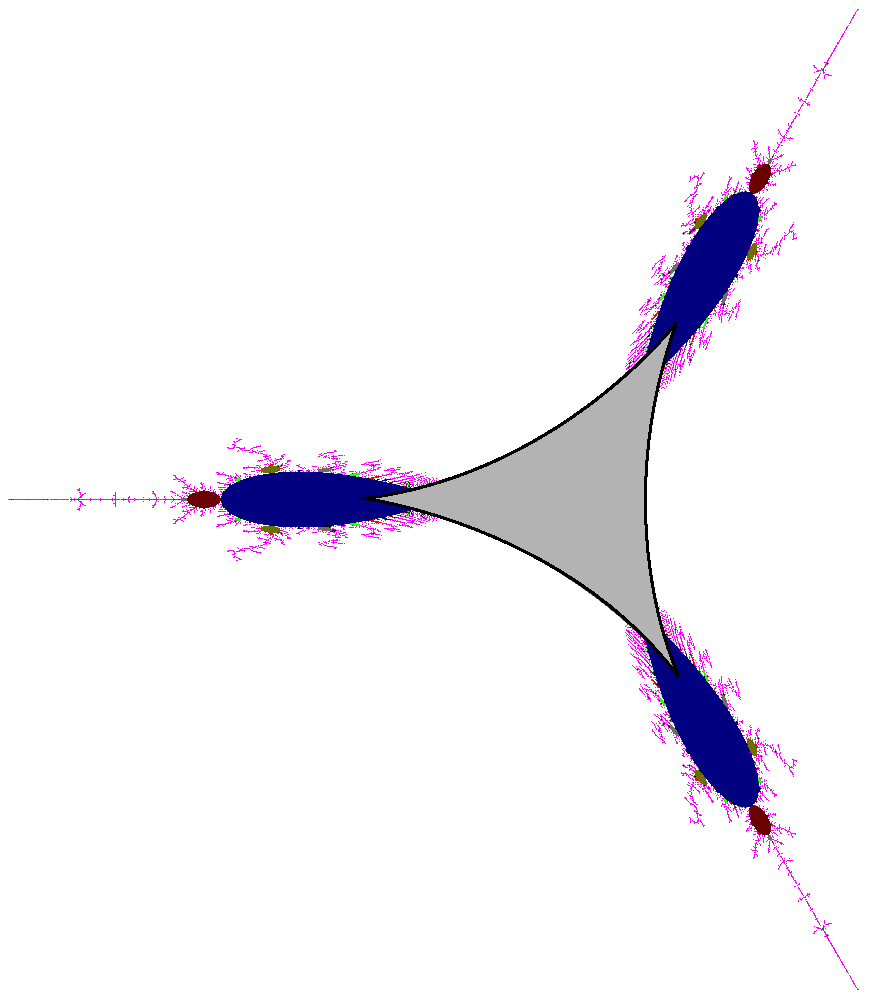}\hspace{8mm} \includegraphics[width=0.45\linewidth]{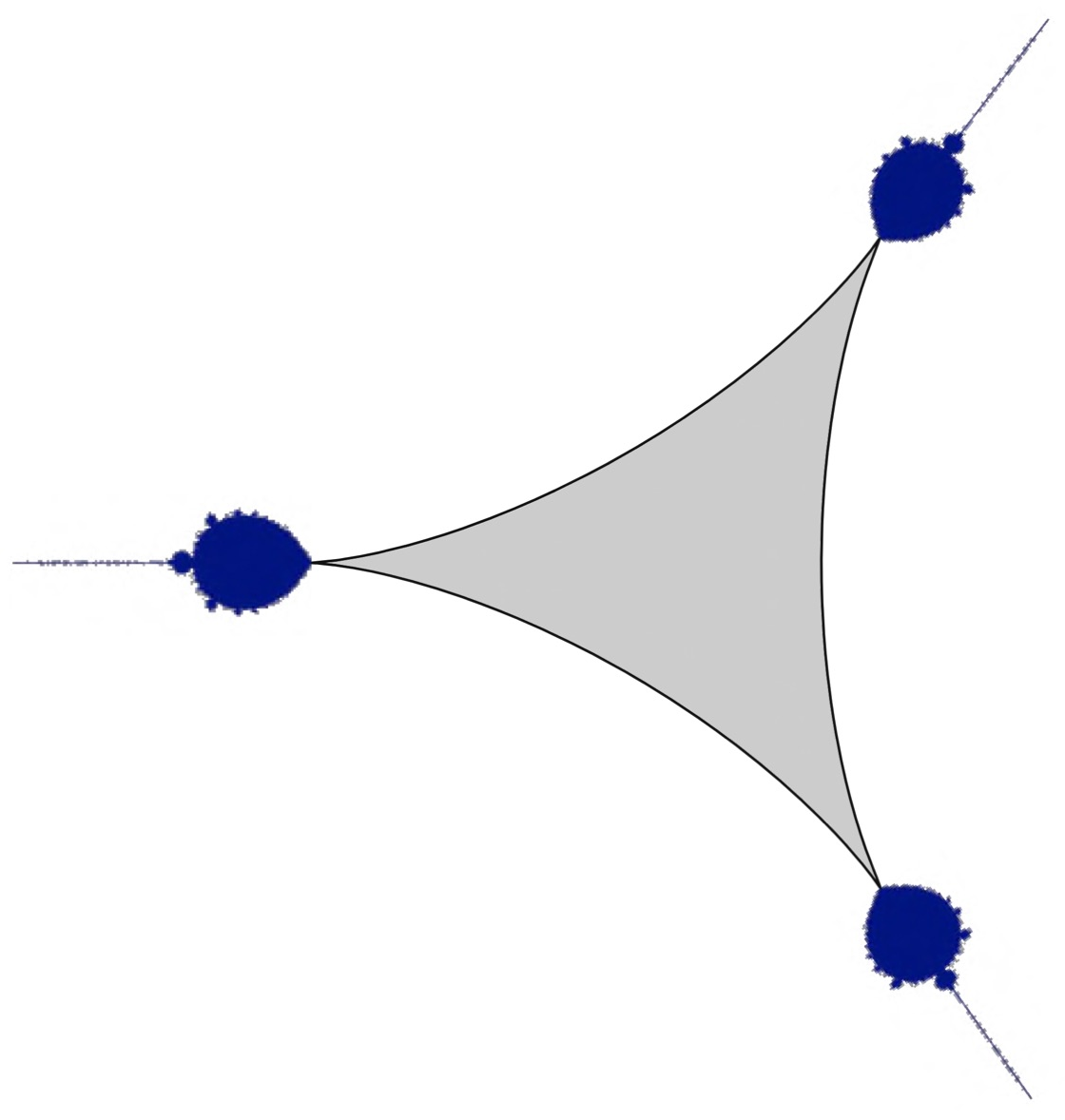}
\caption{The Tricorn $\mathcal{C}_2^-$ (i.e., the connectedness locus of quadratic anti-polynomials) is displayed on the left, and a schematic picture of $\mathscr{S}_{\pmb{\cN}_2}$ (i.e., the connectedness locus of quadratic polygonal Schwarz reflections) is depicted on the right. In both pictures, the central grey region is the period one hyperbolic component and the three symmetric blue bulbs are the period two hyperbolic components. The bifurcation structures of the period two hyperbolic components from the period one component are different in the two connectedness loci, and this causes discontinuity of $\Phi, \Phi^{-1}$.}
\label{discont_fig}
\end{figure}

\begin{lem}\label{discont_lem_2}
Let $\pmb{c_0}$ be the unique double parabolic parameter on $\partial H_0\cap\partial H_1$. Then $\Phi^{-1}$ is discontinuous at $\Phi(f_{\pmb{c_0}})$.
\end{lem}
\begin{proof}
Let $\pmb{c_\infty}$ be a simple parabolic parameter on $\partial H_0\cap\partial H_1$, and $\{c_n\}$ be a sequence in $H_1$ converging to $\pmb{c_\infty}$. Then, the multiplier of the unique attracting $2$-cycle of $f_{c_n}$ converges to $1$. By definition of $\Phi$, the multiplier of the unique attracting $2$-cycle of $\Phi(f_{c_n})$ also converges to $1$. This implies that the sequence $\{\Phi(f_{c_n})\}$ converges to $\Phi(f_{\pmb{c_0}})\in\mathscr{S}_{\pmb{\cN}_d}$, while $\{c_n\}$ converges to $\pmb{c_\infty}$. Hence, $\Phi^{-1}$ is discontinuous at $\Phi(f_{\pmb{c_0}})$.
\end{proof}

We are now ready to complete the proof of Theorem~\ref{thm:C}.
\begin{proof}[Proof of Theorem~\ref{thm:C} (discontinuity part)]
The discontinuity of $\Phi$ and $\Phi^{-1}$ follow from Lemmas~\ref{discont_lem_1} and~\ref{discont_lem_2}, respectively. 
\end{proof}

\section{Combinatorial classes for polynomials and anti-polynomials}\label{comb_class_poly_sec}
Let $\mathcal{C}^\pm_{d, r}$ be the space of monic and centered polynomials (or anti-polynomials) with connected Julia set and all cycles repelling (also called `periodically repelling').
In this section, we discuss a natural way to organize the data for renormalizations of polynomials, or anti-polynomials in $\mathcal{C}^\pm_{d, r}$.

We will use puzzle pieces to construct a decomposition of the dynamical plane.
To better organize these combinatorial data, we will inductively construct puzzles depending on their {\em level} of renormalizations and the {\em depth} of the puzzle pieces.

\subsection{Level zero renormalization puzzles}
To start our construction, we consider rays landing at the fixed points.
If the polynomial or the anti-polynomial is periodically repelling, the following lemma guarantees that these rays cut through the Julia set.
\begin{lem}\label{lem:cj}
Let $f \in \mathcal{C}^\pm_{d, r}$.
Then the union of rays landing at all fixed points together with their landing points cut the Julia set of $f$.
\end{lem}
\begin{proof}
If $f$ is a polynomial, then there are $d$ fixed points in $\C$.
Since $f$ has connected Julia set and is periodically repelling, these fixed points are distinct and all lie on the Julia set.
Since there are $d-1$ fixed rays, at least one fixed point $p$ is the landing point of some periodic rays of period $\geq 2$.
Then these rays landing at $p$ together with $p$ cut the Julia set of $f$.

If $f$ is an anti-polynomial, then there are $d$ fixed points in $\C$ by Lefschetz fixed point theorem, as all fixed points are repelling (see \cite[Lemma 6.1]{LM14}).
Since there are $d+1$ fixed rays, at least two fixed rays must land together at a fixed point $p$.
Then these two rays landing at $p$ together with $p$ cut the Julia set of $f$.
\end{proof}

To define level zero renormalization puzzles, we let $\phi$ be the Riemann mapping from $\C - \overline{\D}$ to $\C - K(f)$ (where $K(f)$ is the filled Julia set of $f$), normalized so $\phi'(\infty) = 1$.
Consider the disk $D$ bounded by the image of the circle $\{\vert z\vert =2\}$ under $\phi$.
This disk encloses $K(f)$ and is bounded by an equipotential of the complement.

By Lemma \ref{lem:cj}, this disk is cut into $q>1$ pieces by the external rays landing at the fixed points of $f$ that are also cut points.
These pieces form the {\em puzzle pieces} of depth zero for level zero renormalization.
Each piece is a closed disk, whose boundary consists of fixed points, segments of external rays landing at those fixed points, and part of $\partial D$.

The puzzle pieces at depth $k+1$ are defined inductively as the closures of the components of $f^{-1}(\Int(P))$, where $P$ ranges over all puzzle pieces at depth~$k$.

The puzzle pieces at depth $k$ have disjoint interiors and cover the Julia set.
As the depth increases, the puzzles become successively finer.
More precisely, a piece at depth $k+1$ is contained in a unique piece at depth $k$.

We will use the notation $P_{k, f}^0(z)$ to denote the puzzle piece at depth $k$ for level zero renormalization, that contains $z$ in its interior.
Sometimes we simply write it as $P_{k}^0(z)$ or $P_k(z)$ if the map $f$ and the level of renormalizations are not ambiguous.

We remark that the puzzle piece $P_{k, f}^0(z)$ is well-defined for $z\in J(f)$ if $f^k(z)$ is not a fixed point.
If $f^k(z)$ is a fixed point, then $z$ might be on the common boundary of multiple puzzle pieces at depth $k$.

\begin{defn}\label{defn:lzr}
Let $c$ be a critical point for $f$.
We say that $c$ is {\em level zero renormalizable} if $c$ is not on the boundary of any puzzle piece and there exists an integer $p \geq 1$ so that 
$$
P_k(c) = P_k(f^p(c)) \text{ for all } k\geq 0.
$$
The smallest such integer $p$ is called the {\em renormalization period}.

We say that $f$ is {\em level zero renormalizable} if some critical point of $f$ is {\em level zero renormalizable}, and {\em level zero non-renormalizable} otherwise.
\end{defn}

\begin{lem}\label{lem:rn}
If $c$ is level zero renormalizable for a map $f \in \mathcal{C}^\pm_{d, r}$, then there exist open sets $U\ni c, V\ni c$ with $U$ compactly contained in $V$ so that 
$$
f^p: U \longrightarrow V
$$
is a polynomial-like map with connected small Julia set $J(c) := \bigcap_k P_k(c)$.
\end{lem}
\begin{proof}
Since $P_k(c)$ is a nested decreasing sequence of closed disks, 
there exists some integer $N$ so that for any critical point $c'$, either
\begin{itemize}
\item $c' \in P_k(c)$ for all $k \geq N$; or
\item $c' \notin P_k(c)$ for all $k\geq N$.
\end{itemize}
Let $k \geq N$.
Consider the map
$$
f^p: P_{k+p}(c) \longrightarrow  P_k(f^p(c)) = P_k(c).
$$
Note that $P_{k+p}(c) \subseteq P_k(c)$.
Since all cycles of $f$ are repelling, we can find a small neighborhood $U, V$ of $P_{k+p}(c)$ and $P_k(c)$ so that
\begin{itemize}
\item a critical point $c' \in U$ if and only if $c' \in P_{k+p}(c)$;
\item $U$ is compactly contained in $V$;
\item $f^p: U \longrightarrow V$ is a branched covering map.
\end{itemize}
Then any point in $J(c) \subseteq U$ is non-escaping for the map $f^p: U \longrightarrow V$.
Let $c' \in U$ be a critical point.
Then $c' \in P_k(c)$ for all $k \geq N$, so $c'$ is non-escaping.
Therefore, $f^p: U \longrightarrow V$ is a polynomial-like map with connected small Julia set $J(c)$.
\end{proof}

\begin{cor}\label{cor:lgeq2}
If $c$ is level zero renormalizable, then its renormalization period $p$ is at least $2$.
\end{cor}
\begin{proof}
Suppose for contradiction that $p = 1$.
Then we have a polynomial-like map $f:U \longrightarrow V$.
By Lemma \ref{lem:cj}, the external rays landing at the fixed points of $f: U \longrightarrow V$ would cut through the small Julia set $J(c) = \bigcap_k P_k(c)$, which is a contradiction.
\end{proof}

\begin{cor}
If $c$ is level zero renormalizable, then $f^p: U \longrightarrow V$ is a simple renormalization, i.e., the small Julia sets $J(c), J(f(c)), ..., J(f^{p-1}(c))$ intersect possibly only at their tip points.
\end{cor}
\begin{proof}
For sufficiently large $k$, we may assume that the puzzle pieces 
$$
P_k(c), P_k(f(c)), ..., P_k(f^{p-1}(c))
$$ 
have disjoint interior.
Therefore, $J(f^i(c)) = \bigcap_k P_k(f^{i}(c))$ cannot intersect at a cut point.
\end{proof}






\subsection{Level $m$ renormalization puzzles}
Let $f \in \mathcal{C}^\pm_{d, r}$ be a polynomial or anti-polynomial with connected Julia set and all cycles repelling.
Suppose $f$ is level zero renormalizable.
Let $c$ be a critical point that is level zero renormalizable with renormalization period $p$.
By Lemma \ref{lem:rn}, we have a polynomial-like map $f^p: U \longrightarrow V$ with small Julia set $J(c)$. 
The fixed points of $f^p$ in $J(c)$ are periodic points of $f$ with period $p$.
Note that the polynomial-like map $f^p: U \longrightarrow V$ has connected Julia set and is periodically repelling.
Thus, the external rays landing at the fixed points of $f^p$ in $J(c)$ cut the small Julia set.

To define the {\em level one renormalization puzzle}, we consider 
\begin{itemize}
\item the external rays landing at all fixed points of $f$; as well as 
\item the external rays landing at the periodic points of $f$ associated with the fixed points of the polynomial-like maps $f^p: U \longrightarrow V$ for all level zero renormalizable critical points.
\end{itemize}
Then the closed disk $D$ bounded by the image of the circle $\{\vert z\vert =2\}$ under the Riemann mapping $\phi: \C - \overline{\D} \longrightarrow \C - K(f)$ is cut by these external rays into finitely many pieces.
These pieces form the {\em puzzle pieces} of depth zero for level one renormalization.

Similar as before, the puzzle pieces at depth $k+1$ are defined inductively by pulling back the puzzle pieces at depth $k$.
We use the notation $P^1_k(z)$ to denote the puzzle piece at depth $k$ for level one renormalization that contains $z$ in its interior.

Just as in Definition \ref{defn:lzr}, we define a critical point $c$ to be {\em level one renormalizable}
if $c$ is not on the boundary of any puzzle piece and there exists a number $p \geq 1$ so that 
$$
P^1_k(c) = P^1_k(f^p(c)) \text{ for all } k.
$$
If $c$ is level one renormalizable with renormalization period $p_1$, then the same proof of Lemma \ref{lem:rn} shows that there exists a polynomial-like map
$$
f^{p_1}: U^1 \longrightarrow V^1
$$
with small Julia set $J^1(c)= \bigcap_k P^1_k(c)$.

Inductively, if $f$ is level $m$ renormalizable, then we can define the level $m+1$ renormalization puzzle for $f$.
We remark that the choice of depth zero puzzle pieces for level $n$ renormalization depends on the polynomial or the anti-polynomial $f$.

To make our notations uniform in the future, if $f$ is not level $m$ renormalizable, then we define level $m+1$ renormalization puzzle for $f$ equal to the level $m$ renormalization puzzle.

In this way, for a periodically repelling polynomial or an anti-polynomial $f$ with connected Julia set, we can associate an infinite sequence of puzzles
$$
\mathfrak{P}^m(f) = \{P^m_k\}, m=0, 1, \cdots.
$$
We say that $f$ is {\em finitely renormalizable} if this sequence of puzzle $\mathfrak{P}^m(f)$  is eventually constant, and {\em infinitely renormalizable} otherwise.

It is clear that we have the following basic nesting property.
\begin{lem}
If $c$ is level $m$ renormalizable, then $c$ is level $j$ renormalizable for all $j \leq m$.
The corresponding renormalization periods satisfy
$$
p_0 \mid p_1 \mid ... \mid p_m.
$$
Further, the small Julia sets are nested; i.e.,
$$
J^m(c) \subseteq J^{m-1}(c) \subseteq ... \subseteq J^0(c).
$$
\end{lem}

\subsection{Combinatorial distance}
\begin{defn}
    Let $f, g \in \mathcal{C}^\pm_{d, r}$. 
    We say that $g$ is {\em aligned} with $f$ at depth $k$ of level $m$ if for each point $p \in \partial P^m_{k,f} \cap J(f)$, there exists a point $q \in J(g)$ so that the angles of the external rays landing at $p$ are the same as the angles of those landing at $q$.
\end{defn}
From the definition, it is clear that if $g$ is aligned with $f$ at depth $k$ of level $m$, then $g$ is aligned with $f$ at depth $k'$ of level $m'$ for all $k' \leq k$ and $m' \leq m$.

Given a puzzle piece $P^m_{k,f}$ for $f$, its boundary consists of segments of external rays, their landing points and segments of an equipotential.
If $g$ is aligned with $f$ at depth $k$ of level $m$, we can construct a puzzle piece for $g$ bounded by the corresponding segments of external rays (with the same angle), their landing points and segments of the equipotential.
We call this the {\em induced puzzle piece} for $g$.
We use the notation $P^m_{k, g|f}(z)$ to represent the induced puzzle piece for $g$ at depth $k$ of level $m$ that contains $z$.


We remark that these induced puzzle pieces for $g$ may be different from the level $m$ renormalization puzzle $\mathfrak{P}^m(g)$ for $g$.
Indeed, it is possible for a non-renormalizable polynomial $g$ to align with a renormalizable polynomial $f$ at some depth $k$ and level $m$ with $m \geq 1$.
Then there are more induced puzzle pieces at depth zero of level $1$ than those in $\mathfrak{P}^1(g)$. { As an example, one can take $f$ to be the quadratic Feigenbaum polynomial and pick any non-renormalizable $g$ in the limb of the Mandelbrot set cut out by the parameter rays at angles $2/5$ and $3/5$.}

We define level $m$ renormalization distance from $g$ to $f$ as
$$
d_m(f, g) = e^{-(k_m(f, g)+1)\cdot (m+1)},
$$
where 
$$
k_m(f,g) = \inf \{j: g \text{ is not aligned with $f$ at depth $j$ of level $m$}\}.
$$
We define
$$
d(f, g) = \sum_{m\geq 0} d_m(f, g).
$$
We remark that $d$ is not a distance function as $d(f, g)$ may be different from~$d(g, f)$.

Recall that two periodically repelling polynomials or anti-polynomials are in the same combinatorial class if they have the same rational lamination, and $\widehat{\mathcal{C}^\pm_{d, r}}$ consists of combinatorial classes of periodically repelling polynomials or anti-polynomials.
\begin{lem}\label{lem:rl}
Let $f, g \in\mathcal{C}^\pm_{d, r}$. Then $d(f, g) = 0$ if and only if $f, g$ have the same rational lamination; i.e., $[f] = [g] \in \widehat{\mathcal{C}^\pm_{d, r}}$.
\end{lem}
\begin{proof}
It is clear that if $f, g$ have the same rational lamination, then the induced puzzle for $g$ agrees with the renormalization puzzle for $f$ at all levels and depths. Thus, $d(f, g) = 0$.

Conversely, suppose that $s, t \in \Q$ form a leaf of the rational lamination for $f$; i.e., they coland at some pre-periodic point $x$ for $f$.
If $x$ is on the boundary of some puzzle piece, then $s, t$ also coland for $g$ as $d(f, g) = 0$.
So $s, t$ form a leaf for the rational lamination for $g$ as well.

Thus, we assume that $x$ is not on the boundary of any puzzle piece.
Assume that $x$ has pre-period $l$ and period $p$, and $y = f^l(x)$.
Then $f^p: \Int(P^m_{k+p, f}(y)) \longrightarrow \Int(P^m_{k, f}(y))$ is a proper map for all $m$ and $k$.
We can choose $m$ large enough so that the level $m$ renormalization period for any critical point is larger than $p$.
(Here we use the convention that the renormalization period is $\infty$ if the critical point is not renormalizable.)
Thus, we can then choose $N$ large enough so that $\Int(P^m_{k+p, f}(y))$ contains no critical points for $k \geq N$.
So $f^p: \Int(P^m_{k+p, f}(y)) \longrightarrow \Int(P^m_{k, f}(y))$ has degree $1$ for all $k \geq N$.
Therefore, $\bigcap_k P^m_{k, f}(y) = \{ y\}$, so $\bigcap_k P^m_{k, f}(x) = \{ x\}$.
Hence, we have a sequence of angles $s_k, t_k$ colanding at boundary points of some puzzle pieces with $s_k \to s$ and $t_k \to t$.
Since $d(f, g) = 0$, the angles $s_k, t_k$ define a leaf for the rational lamination of $g$, for all $k$.
As rational laminations are closed, the angles $s, t$ also define a leaf of the rational lamination for $g$.
Therefore, the rational lamination of $f$ is contained in the rational lamination of $g$.
The other inclusion can be proved similarly.
\end{proof}

\subsection{Combinatorial continuity}
In this section, we will prove a combinatorial continuity statement for maps in $\mathcal{C}^\pm_{d, r}$.
The proof is divided into two cases depending on whether there are critical points on boundaries of puzzle pieces or not.

\subsubsection{No boundary critical points}
We first present the proof where no critical points are on the boundaries of puzzle pieces.
\begin{theorem}\label{thm:ncp}
Let $f \in \mathcal{C}^\pm_{d, r}$. Suppose that no critical point is on the boundary of any puzzle piece. 
\begin{enumerate}
\item If $f_n \to f$ in $\mathcal{C}^\pm_{d, r}$, then $d(f, f_n) \to 0$.
\item If $d(f, f_n) \to 0$, then for any accumulation point $g\in \mathcal{C}^\pm_{d, r}$ of $f_n$, we have that $d(f, g) = 0$.
\end{enumerate}
\end{theorem}

\begin{remark}
Suppose that for some fixed level $m$, the distance $d_m(f, f_n) \to 0$ as $n\to\infty$. We remark that it is not true that any accumulation point $g\in \mathcal{C}^\pm_{d, r}$ of $f_n$ satisfies $d_m(f, g) = 0$.

For example, let $f$ be the quadratic Feigenbaum polynomial, and $f_n$ be the tuning of the basilica polynomial with a sequence of real polynomials $g_n$ that goes to the Chebyshev polynomial.
The limit $g = \lim_n f_n$ is the basilica polynomial tuned with the Chebyshev polynomial.
Then $d_0(f, f_n) = 0$ for all $n$, but $d_0(f, g) \neq 0$ since two rays (lying on the boundary of level zero puzzles) that land at distinct Julia points for $f$ coland at the critical point of $g$.

\end{remark}

\begin{lem}[Separation lemma]\label{lem:sl}
Let $f \in \mathcal{C}^\pm_{d, r}$. Suppose that no critical point is on the boundary of any puzzle piece.
Let $c$ be a critical point for $f$, and $P^m_k(c)$ be the level $m$ puzzle piece at depth $k$ that contains $c$.
Then there exist $l, j$ so that $P^{l}_{j}(c)$ is compactly contained in $\Int(P^m_k(c))$.
\end{lem}
\begin{proof}
Suppose for contradiction that there exists a Julia point $x$ on the boundary of $P^m_k(c)$ so that $x \in P^l_j(c)$ for all $j, l$.
{Let $q$ be the pre-period and $p$ be the period of the pre-periodic point $x$. The assumption $x \in P^l_j(c)$ for all $j, l$ implies that $f^{q+np}(x)=f^q(x)\in P^l_j(f^{q+np}(c))$ for all $j, l$, and $n$. It follows that $P^l_j(f^{q}(c)) = P^l_j(f^{q+rp}(c))$, where $r$ is the valence of $f^q(x)$ as a branch point of $J(f)$.
}
If for some $l$, $P^l_j(f^{q}(c))$ contains no critical points for all large $j$, then $\bigcap_j P^l_j(f^{q}(c)) = f^{q}(x)$; {i.e., $f^q(c)=f^q(x)$}.
This is a contradiction to the fact that $c$ is not on the boundary of any puzzle piece.
Otherwise, there exists some critical point $c'$ (potentially equal to $c$) which is infinitely renormalizable, and $f^q(x)$ is in the small Julia set $J^l(c')$ for all $l$. But this is not possible as the period of $f^q(x)$ is a fixed integer.
\end{proof}

We are now ready to prove Theorem \ref{thm:ncp}.
\begin{proof}[Proof of Theorem \ref{thm:ncp}]
To prove the first statement, it suffices to show $d_m(f, f_n) \to 0$ for all $m$.
Let $P$ be a puzzle piece of level $m$ and depth $k$.
Since there are no critical points of $f$ on the boundaries of any puzzle pieces, any points in $\partial P$ is eventually repelling and its forward iterate avoids any critical points.
Thus, by stability of landing external rays (see \cite{GM93, GT21}), there exists a small neighborhood $U_{m,k}$ of $f$ so that any map $g \in U_{m,k}$ is aligned with $f$ at depth $k$ of level $m$.
Therefore $d_m(f, f_n) \to 0$ as $n \to \infty$.

For the second statement, by passing to a subsequence, we may assume that $f_n \to g \in \mathcal{C}^\pm_{d, r}$. 
It suffices to show that $d_m(f, g) = 0$ for all $m$.
Let $m$ be the minimal level where $d_m(f, g) \neq 0$, and let $k:= k_m(f, g)$.

We claim that $k > 0$.
By way of contradiction, suppose that $k=0$.
Let $\{t_i\}$ be the set of angles of external rays landing at boundaries of the puzzle pieces at depth $0$ of $\mathfrak{P}^m(f)$.
Note that the Julia points on boundaries of puzzle pieces at depth $0$ for $f$ are periodic points, so $t_i$ are periodic under the map $m_d(t) = dt$ { or $m_{-d}(t)=-dt$}.
Since $d(f, f_n) \to 0$, the external rays with angles in $\{t_i\}$ have the same landing patterns in the dynamical planes of $f_n$ and $f$ for all sufficiently large $n$.
By our assumption, the landing patterns of the external rays with angles in $\{t_i\}$ in the dynamical plane of $g$ are different from the common landing patterns of these rays in the dynamical planes of $f$ and $f_n$ (with sufficiently large $n$).
Thus, by stability of landing external rays, some angle must land at critical points.
But this is not possible as each $t_i$ is periodic.

Let $\widetilde{P}_{k, f}^m$ be a puzzle piece for $f$ which does not correspond to an induced puzzle piece for $g$.
Since $k > 0$, $\widetilde{P}_{k-1, f}^m:=f(\widetilde{P}_{k, f}^m)$ is a puzzle piece at depth $k-1$.
It corresponds to some induced puzzle piece $\widetilde{P}_{k-1, g|f}^m$ for $g$.
We claim that some critical value $v$ of $g$ must lie on $\partial \widetilde{P}_{k-1, g|f}^m$.
Indeed, otherwise, the Julia points on the boundary of $g^{-1}(\widetilde{P}_{k-1, g|f}^m)$ are pre-repelling and avoid critical points. {Since $\widetilde{P}_{k, f}^m$ does not correspond to a puzzle piece for $g$,} this is a contradiction by stability of landing rays.

Let $c_g$ be the critical point of $g$ that is mapped to $v$.
Since $d(f, f_n) \to 0$, the puzzle pieces of $f_n$ at depth $k$ and level $m$ align with $f$ for all large $n$.
Since no critical point of $f$ is on the boundary of puzzle pieces, for all large $n$, no critical point of $f_n$ is on the boundary of puzzle pieces of $f_n$ at depth $k$ and level $m$.
Since $f_n \to g$, there exists a sequence of critical points $c_n$ for $f_n$ with $c_n \to c_g$.
Let $P_{k, f_n}^m(c_n)$ be the puzzle piece that contains $c_n$, which corresponds to some puzzle piece $P_{k, f}^m(n)$ for $f$.
Since there are only finitely many puzzle pieces at depth $k$ and level $m$, we may assume that $P_{k, f}^m(n)$ is eventually constant.
Denote this puzzle piece by $P_{k, f}^m$. Note that it contains some critical point $c$ of $f$.
One can verify that $P_{k, f}^m$ does not correspond to an induced puzzle piece for $g$.
Since $f(P_{k, f}^m)$ corresponds to an induced puzzle piece for $g$, there exist two angles $\alpha, \beta$ landing at distinct points $a, b \in \partial P_{k, f}^m$ for $f$, but they land at a critical point for $g$.
By Lemma \ref{lem:sl}, there exists some puzzle piece $P^{m_1}_{k_1, f}(c)$ that is compactly contained in $\Int(P_{k, f}^m) = \Int(P_{k, f}^m(c))$.
Inductively, we obtain a sequence 
$$
\cdots \Subset \Int(P^{m_2}_{k_2, f}(c)) \Subset \Int(P^{m_1}_{k_1, f}(c)) \Subset \Int(P^{m}_{k, f}(c)).
$$
Therefore, we can find a sequence of pairs of angles $t_i, s_i$ so that 
\begin{itemize}
\item $t_i, s_i$ coland at a point on $\partial P^{m_i}_{k_i, f}(c)$;
\item the union of external rays $\overline{\mathcal{R}_f(t_i)} \cup \overline{\mathcal{R}_f(s_i)}$ separates the two points $a, b \in \partial P_{k, f}^m(c)$.
\end{itemize}
Therefore, $\mathcal{R}_{f_n}(t_i)$ and $\mathcal{R}_{f_n}(s_i)$ land at the same point for all sufficiently large $n$. 
On the other hand, since $\alpha, \beta$ coland for $g$, the union $\overline{\mathcal{R}_g(\alpha)} \cup \overline{\mathcal{R}_g(\beta)}$ separates the external rays $\mathcal{R}_g(t_i)$ and $\mathcal{R}_g(s_i)$.
It follows that both $\mathcal{R}_g(t_i)$ and $\mathcal{R}_g(s_i)$ land at $c_g$ (see \cite[Lemma 9.9]{GT21}).
This is a contradiction, as there can be at most finitely many external rays landing at a pre-periodic~point.
\end{proof}



\subsubsection{With boundary critical points}
To deal with the situation where critical points lie on boundaries of puzzle pieces, we need to modify the definition of combinatorial distances.

Let $f \in \mathcal{C}^\pm_{d, r}$ with puzzles $\mathfrak{P}^m$.
Let $c$ be a critical point on the boundary of some puzzle piece.
Then $c$ is pre-periodic.
Denote by $\mathcal{O}$ the periodic cycle associated to $c$.
We define a modification of $\mathfrak{P}^m$ by removing all external rays landing on $\mathcal{O}$ and their iterated preimages in the definition of puzzle pieces.

As the first step of the modification, we denote by $\widetilde{\mathfrak{P}}^m$ the modified puzzle after removing all external rays associated with all critical points on the boundaries of puzzle pieces.


\begin{defn}\label{pers_renorm_def}
We say that $c$ is {\em statically renormalizable} with respect to the modified puzzle $\widetilde{\mathfrak{P}}^{m_0}$ if $c$ is renormalizable with respect to the modified puzzle $\widetilde{\mathfrak{P}}^{m_0}$ and
$$
\bigcap_k \widetilde{P}^m_k(c) = \bigcap_k \widetilde{P}^{m_0}_k(c)
$$ 
for all $m \geq m_0$.
We denote by
$$
J^\infty(c) = J^{m_0}(c) := \bigcap_k \widetilde{P}^{m_0}_k(c) = \bigcap_m \bigcap_k \widetilde{P}^m_k(c)
$$
the {\em static small Julia set} for $c$.
Similarly, we say that a critical point $c$ is {\em statically pre-renormalizable} if $c$ is eventually mapped to some static small Julia set.
\end{defn}

    We remark that if $c$ is statically renormalizable, then $c$ is renormalizable with respect to $\widetilde{\mathfrak{P}}^{m}$ for all $m$.
    However, for all $m \geq m_0$, the renormalization with respect to $\widetilde{\mathfrak{P}}^{m}$ stays the same, as the condition $\bigcap_k \widetilde{P}^m_k(c) = \bigcap_k \widetilde{P}^{m_0}_k(c)$ implies that all higher level puzzle pieces do not cut through the small Julia set $J^{m_0}(c)$.
    This is a consequence of the fact we generate the puzzle by rays landing at the fixed points of renormalizations.

\begin{lem}\label{lem:fr}
Let $f \in \mathcal{C}^\pm_{d, r}$.
Let $c$ be a statically renormalizable critical point with respect to $\widetilde{\mathfrak{P}}^{m_0}$.
Then the corresponding polynomial-like map
$$
f^{p_{m_0}}: U^{m_0} \longrightarrow V^{m_0}
$$
is finitely renormalizable.
\end{lem}
\begin{proof}
If $f^{p_{m_0}}: U^{m_0} \longrightarrow V^{m_0}$ is infinitely renormalizable, then as $m \to \infty$, there are infinitely many puzzle pieces of $\mathfrak{P}^m$ at depth $0$ that cut through the small Julia set.
The modified puzzles can only omit finitely many of them.
Therefore, as $m \to \infty$, there exist an unbounded number of modified puzzle pieces of $\widetilde{\mathfrak{P}}^m$ at depth $0$ that cut through the small Julia set $J^\infty(c)$. 
Since $c$ is a statically renormalizable, there are no modified puzzle pieces of $\widetilde{\mathfrak{P}}^m$ at depth $0$ that cut through $J^\infty(c)$.
This is a contradiction.
\end{proof}

The second step of the modification is not canonical.
Let $c$ be a statically renormalizable critical point for $\widetilde{\mathfrak{P}}^{m_0}$, with static small Julia set $J^\infty(c)$.
We choose periodic cut cycles in $J^\infty(c)$ disjoint from the critical orbit of $f$.
We add puzzle pieces induced by the external rays landing at these cut points as depth zero puzzles in $\widetilde{\mathfrak{P}}^{m}$ for all $m \geq m_0$.
Pulling back by dynamics, we add depth $k$ puzzles in $\widetilde{\mathfrak{P}}^{m}$.
By Lemma \ref{lem:fr}, we can choose these periodic cut cycles so that no critical point in $J^\infty(c)$ is renormalizable with respect to the new modified puzzle.
By making such modifications for each static renormalization critical point, we obtain a new puzzle system.
We denote it by $\widehat{\mathfrak{P}}$.
We denote by $\widehat{d}(f, g)$ the distance from $g$ to $f$ with respect to the modified puzzle $\widehat{\mathfrak{P}}$.

By our construction, any static critical point is only finitely renormalizable with respect to the modified puzzle $\widehat{\mathfrak{P}}$. 
A similar argument as Lemma \ref{lem:rl} gives
\begin{lem}\label{lem:rl2}
Let $f, g\in \mathcal{C}^\pm_{d, r}$. 
Then $\widehat{d}(f, g) = 0$ if and only if $f, g$ have the same rational lamination.
\end{lem}

By construction, no critical point is on the boundary of puzzle pieces in $\widehat{\mathfrak{P}}$. Thus, the proof of Theorem~\ref{thm:ncp} applies to this setting as well, and yields:
\begin{theorem}\label{thm:ccf}
Let $f, f_n \in \mathcal{C}^\pm_{d, r}$.
\begin{enumerate}
\item If $f_n \to f$ in $\mathcal{C}^\pm_{d, r}$, then $\widehat{d}(f, f_n) \to 0$.
\item If $\widehat{d}(f, f_n) \to 0$, 
then for any accumulation point $g\in \mathcal{C}^\pm_{d, r}$ of $f_n$, we have that $\widehat{d}(f, g) = 0$.
\end{enumerate}
\end{theorem}

\subsection{Approximation by post-critically finite maps}
In this subsection, we will prove the following theorem about approximating any combinatorial class using post-critically finite maps.
\begin{theorem}\label{thn:approxbypcf}
Let $f\in \mathcal{C}^\pm_{d, r}$. Then there exists a sequence of post-critically finite polynomials $f_n \in \mathcal{C}^\pm_{d, r}$ so that $\widehat{d}(f, f_n) \to 0$.
\end{theorem}
\begin{proof}
    Let $\lambda(f)$ be the rational lamination of $f$.
    By \cite[\S 6.8]{Kiw04}, there exists a sequence of post-critically finite $f_n$ converging to {a periodically repelling map} $g$ so that $\lambda(g) = \lambda(f)$.
    Since $f$ is periodically repelling, $f_n$ can be chosen to be periodically repelling as well.
    Then by Theorem \ref{thm:ccf}, we have  {$\widehat{d}(g, f_n) \to 0$. Since $\lambda(g) = \lambda(f)$ (and the construction of the puzzles can be uniquely recovered from the lamination), it follows that $\widehat{d}(f, f_n) \to 0$.}
\end{proof}

\section{Combinatorial classes for polygonal Schwarz reflections}\label{comb_class_schwarz_sec}

Recall that for a Schwarz reflection map $S:\overline{\cD}\to\widehat{\C}$, the set $\mathfrak{S}$ stands for the singular (fixed) points on the boundary $\partial\cD$. The local dynamics near a singular point is parabolic in nature. In fact, for $S\in\mathscr{S}_{\pmb{\cN}_d}$, these singular fixed points have attracting directions in the tiling set. Indeed, the points in $\mathfrak{S}$ correspond, under the conformal map $\psi:\D\to T^\infty(S)$ that conjugates $\pmb{\cN}_d$ to $S$, to the (parabolic) fixed points of $\pmb{\cN}_d$. The set of (parabolic) fixed points of $\pmb{\cN}_d$, denoted by $\mathrm{Fix}(\pmb{\cN}_d)$, consists of the $(d+1)-$st roots of unity. It is easy to see that points in $\mathrm{Fix}(\pmb{\cN}_d)$ attract nearby points in $\D$ under iterations of $\pmb{\cN}_d$.

We define the space $\mathscr{S}_{\pmb{\cN}_d,r}\subset\mathscr{S}_{\pmb{\cN}_d}$ as the collection of those maps $S$ for which 
\begin{enumerate}
\item no point in $\mathfrak{S}$ has an attracting direction in $K(S)$, and 
\item each periodic point of $S$ not in $\mathfrak{S}$ is repelling.
\end{enumerate}

\subsection{Renormalization puzzles}
{Recall that $T^0(S)$ is the fundamental tile for the Schwarz reflection map $S$, and connected components of $S^{-n}(T^0(S))$ are called tiles of rank $n$. As in Proposition~\ref{basic_top_prop}, we denote the union of the tiles of rank $0$ through $k$ by $E^k$. Clearly, $E^0=T^0(S)$.
We define the {\em depth zero equipotential} for $S$ as $\partial E^0$.
Similarly, the {\em depth $k$ equipotential} is defined as $\partial E^k$.}
Note that these equipotentials enclose $K(S)$; however, unlike the polynomial or anti-polynomial setting, an equipotential for a Schwarz reflection always intersects the limit set $\Lambda(S)$.
We also remark that any intersection point of an equipotential of depth zero with the limit set $\Lambda(S)$ is a fixed point of $S$.

To define level zero renormalization puzzles for a Schwarz reflection $S \in \mathscr{S}_{\pmb{\cN}_d,r}$, we consider the union $\bigcup \overline{\pmb{\cR}}$ of closures of all dynamical rays landing at all fixed points of $S$ (see Section~\ref{subsec:dl} for the definition of dynamical rays for Schwarz reflections in $\mathscr{S}_{\pmb{\cN}_d}$).
A similar counting argument as in Lemma \ref{lem:cj} shows that this closure cuts the limit set $\Lambda(S)$.
{We define puzzle pieces at depth zero for level zero renormalization as closures of the connected components of 
$$
\widehat\C - \overline{E^0} - \bigcup \overline{\pmb{\cR}}
$$
intersecting $\Lambda(S)$.}
Similarly as in the polynomial case, the puzzle pieces at depth $k+1$ are defined inductively as closures of the components of $S^{-1}(\Int(P))$, where $P$ ranges over all puzzle pieces at depth $k$.
We will use a similar notation as in the polynomial case, where $P_k(z)$ represents the depth $k$ puzzle piece that contains $z$ in its interior.

Renormalization can be defined similarly.
\begin{defn}\label{defn:lzr1}
Let $c$ be a critical point for $S$.
We say that $c$ is {\em level zero renormalizable} if there exists an integer $p \geq 1$ so that 
$$
P_k(c) = P_k(S^p(c)) \text{ for all } k\geq 0.
$$
The smallest such integer $p$ is called the {\em renormalization period}.

We say that $S$ is {\em level zero renormalizable} if \st{any} some critical point of $S$ is {\em level zero renormalizable}, and {\em level zero non-renormalizable} otherwise.
\end{defn}

Similar to the polynomial setting, we have the following lemma (c.f. Lemma \ref{lem:rn}).
\begin{lem}\label{lem:rns}
If $c$ is level zero renormalizable with renormalization period $p>2$, then there exist open sets $U\ni c, V\ni c$ with $U$ compactly contained in $V$ so that 
$$
S^p: U \longrightarrow V
$$
is a polynomial/anti-polynomial-like map with small Julia set $J(c) := \bigcap_k P_k(c)$.
\end{lem}
\begin{proof}
We comment on the nuance of the assumption that the renormalization period $p$ is larger than two.
Indeed, the same argument as in Lemma \ref{lem:rn} gives a branched covering between puzzles $P_{k+p}(c) \subseteq P_k(S^p(c)) = P_k(c)$,
$$
S^p: P_{k+p}(c) \longrightarrow  P_k(c).
$$
The subtlety here is to thicken the puzzle pieces slightly to ensure compact containment between them.
Since $S$ is assumed to be periodically repelling, this can be done if there are no singular points on the boundaries of the puzzle pieces.

Now let $x$ be a singular point. Then there are at most two different nested sequences of puzzle pieces $P_n$ with $x \in \partial P_n$ {(recall that at most two quadrature domains can touch at a point)}.
Thus, we have $S^2: P_{n+2} \longrightarrow P_{n}$.
Therefore, if the renormalization period $p$ is larger than two, no singular point can lie on the boundary, and the lemma follows.
\end{proof}

A modification of the proof of Corollary \ref{cor:lgeq2} also gives the following lemma.
\begin{lem}\label{lem:lgeq2S}
If $c$ is level zero renormalizable, then its renormalization period $p$ is at least $2$.
\end{lem}
\begin{proof}
By way of contradiction, let us assume that the renormalization period $p$ is equal to $1$.
The proof is identical to Corollary \ref{cor:lgeq2} if there are no singular points on $J(c) := \bigcap_k P_k(c)$.
Thus, we may assume that there is a singular point $x \in J(c)$.
Then for all large $n$, we have a map $S: P_{n+1}(c) \longrightarrow P_n(c)$ of degree {$e>1$}.
Since we assume that the map is periodically repelling, we can extend the map $S: P_{n+1}(c) \longrightarrow P_n(c)$ to a topological degree $e$ branched covering $F$ of the sphere so that $x$ has (Lefschetz fixed point) index $-1$ and there is a unique fixed point in $\widehat\C - P_n(c)$ with index $+1$.
Thus, by Lefschetz–Hopf fixed point theorem, there exists another fixed point $z \neq x$ in $P_n(c)$.
{The arguments used in the proof of Corollary~\ref{cor:lgeq2} now show that the dynamical rays landing at the fixed
points of $S: P_{n+1}(c) \longrightarrow P_n(c)$ cut through the small Julia set $J(c)$.}
\end{proof}

We can define level $m$ renormalization in this setting as in the polynomial or anti-polynomial case.
The following lemma says that except possibly level zero renormalizations, all subsequent renormalizations are actually polynomial/anti-polynomial-like maps.
\begin{lem}\label{lem:rnsS}
If $c$ is level $m$ renormalizable with $m \geq 1$, then there exist open sets $U\ni c, V\ni c$ with $U$ compactly contained in $V$ so that 
$$
S^{p_0p_1...p_m}: U \longrightarrow V
$$
is a polynomial/anti-polynomial-like map with small Julia set $J^m(c):=~\bigcap_k P^m_k(c)$.
\end{lem}
\begin{proof}
Note that the arguments of Lemma \ref{lem:lgeq2S} imply that $p_0, p_1 \geq 2$, so ${p_0p_1...p_m}>2$.
Thus the lemma follows from arguments used in the proof of Lemma \ref{lem:rns}.
\end{proof}

Recall that $\widehat{\mathscr{S}_{\pmb{\cN}_d, r}}$ consists of combinatorial classes of periodically repelling Schwarz reflections in $\mathscr{S}_{\pmb{\cN}_d}$.
The same proof as Lemma \ref{lem:rl} also gives the following lemma.
\begin{lem}\label{lem:rl3}
Let $S, R \in\mathscr{S}_{\pmb{\cN}_d,r}$. Then $d(S, R) = 0$ if and only if $S, R$ have the same rational lamination; i.e., $[S] = [R] \in \widehat{\mathscr{S}_{\pmb{\cN}_d, r}}$.
\end{lem}

\begin{theorem}\label{thm:ncpS}
Let $S \in \mathscr{S}_{\pmb{\cN}_d,r}$. Suppose that no critical point is on the boundary of any puzzle piece. 
\begin{enumerate}
\item If $S_n \to S$ in $\mathscr{S}_{\pmb{\cN}_d,r}$, then $d(S, S_n) \to 0$.
\item If $d(S, S_n) \to 0$, then for any accumulation point $R\in \mathscr{S}_{\pmb{\cN}_d,r}$ of $S_n$, we have that $d(S, R) = 0$.
\end{enumerate}
\end{theorem}
\begin{proof}
    The first statement follows from the same argument as in Theorem \ref{thm:ncp}, as an external ray landing at a repelling periodic point, or a repelling pre-periodic point whose orbit does not contain a critical point is stable under perturbation (see Proposition \ref{rep_stab}).

    By Lemma \ref{lem:rnsS}, all renormalization of level $m \geq 1$ are (anti-)polynomial-like maps.
    Thus, the argument of the Separation Lemma~\ref{lem:sl} applies.
    Now the same argument in Theorem \ref{thm:ncp} applies to show the second statement.
\end{proof}

If there are critical points on the boundary of a puzzle piece, then we can define modified puzzles $\widehat{\mathfrak{P}}^m$ for $S$ exactly as in the (anti-)polynomial setting.
Then, no critical points lie on the boundary of puzzle pieces for $\widehat{\mathfrak{P}}$. {This gives the following combinatorial continuity result, where $\widehat{d}(S_1, S_2)$ stands for the distance from $S_2$ to $S_1$ with respect to the modified puzzle $\widehat{\mathfrak{P}}$.}
\begin{theorem}\label{thm:scc}
Let $S, S_n \in \mathscr{S}_{\pmb{\cN}_d,r}$.
\begin{enumerate}
\item If $S_n \to S$ in $\mathscr{S}_{\pmb{\cN}_d,r}$, then $\widehat{d}(S, S_n) \to 0$.
\item If $\widehat{d}(S, S_n) \to 0$ 
then for any accumulation point $R\in \mathscr{S}_{\pmb{\cN}_d,r}$ of $S_n$, we have that $\widehat{d}(S, R) = 0$.
\end{enumerate}
\end{theorem}

\section{Combinatorial straightening}\label{comb_homeo_sec}
In this section, we shall assemble all the ingredients to prove Theorem~\ref{thm:CL}.
\begin{proof}[Proof of Theorem \ref{thm:CL}]
To construct the map $\Phi$, we first note that if $f$ is post-critically finite, then there exists unique conformal mating $S \in \mathscr{S}_{\pmb{\cN}_d,r}$ of $f$ with $\pmb{\cN}_d$ (by Theorem~\ref{thm:C}).
Let $[f]$ and $[S]$ be the combinatorial class of $f$ and $S$ respectively.
We define $\Phi([f]) = [S]$.

Let $f \in \mathcal{C}^-_{d, r}$ be a periodically repelling anti-polynomial.
By Theorem \ref{thn:approxbypcf}, there exists a sequence of post-critically finite anti-polynomials $f_n$ with $\widehat{d}(f, f_n) \to 0$. Let  $S_n \in \mathscr{S}_{\pmb{\cN}_d,r}$ be the unique conformal mating of $f_n$ with $\pmb{\cN}_d$.
By the compactness result Theorem \ref{thm:compact}, after passing to a subsequence, we may assume that $S_n$ converges.
Denote this limit by $S$.
Using a similar argument in \cite[Lemma 6.34]{Kiw04}, one can show that $S$ cannot have parabolic cycles or irrationally neutral cycles.
Thus $S \in \mathscr{S}_{\pmb{\cN}_d,r}$.
We define $\Phi([f]) = [S]$.
Since $S_n \to S$, by Theorem \ref{thm:scc}, $\widehat{d}(S, S_n) \to 0$.

We claim that the rational lamination $\left(\pmb{\mathcal{E}}_d\right)_*(\lambda(S)) = \lambda(f)$.
Indeed, since $\left(\pmb{\mathcal{E}}_d\right)_*(\lambda(S))$ is $m_{-d}$-invariant, there exists a periodically repelling anti-polynomial $g$ with $\lambda(g) = \left(\pmb{\mathcal{E}}_d\right)_*(\lambda(S))$ (see \cite{Kiw04}).
Since $\left(\pmb{\mathcal{E}}_d\right)_*(\lambda(S_n)) = \lambda(f_n)$, and the construction of the puzzles can be uniquely recovered from the lamination, we conclude that $\widehat{d}(g, f_n) \to 0$.
Therefore, by Theorem \ref{thm:ccf}, $\widehat{d}(f, g) = 0$, and hence $\lambda(f) = \lambda(g)$ by Lemma \ref{lem:rl2}.
This also shows that the map $\Phi$ is well-defined.


We now show that $\Phi$ is continuous.
Suppose $f_n \to f$. 
Then $\widehat{d}(f, f_n) \to 0$.
Let $S_n \in \Phi([f_n])$ and let $S\in \mathscr{S}_{\pmb{\cN}_d,r}$ be an accumulation point of $S_n$.
By Theorem \ref{thm:scc}, $\widehat{d}(S, S_n) \to 0$.
The same argument as above shows that $\left(\pmb{\mathcal{E}}_d\right)_*(\lambda(S)) = \lambda(f)$.
Therefore $S \in \Phi([f])$, proving continuity of $\Phi$.

A similar argument can be applied to construct a well-defined continuous inverse $\Phi^{-1}$, and the theorem follows.
\end{proof}

\section{Rigidity of at most finitely renormalizable Schwarz reflections}\label{fin_renorm_rigid_sec}
In this section, we will show that a periodically repelling, finitely renormalizable Schwarz reflection is rigid, and deduce Theorem~\ref{thm:A}.
More precisely, we have
\begin{theorem}\label{thm:rs}
Let $S \in \mathscr{S}_{\pmb{\cN}_d,r}$ be finitely renormalizable.
Then
\begin{enumerate}
\item the limit set $\Lambda(S)$ is locally connected;
\item $S$ has no invariant line field on $\Lambda(S)$;
\item if $\widetilde{S} \in \mathscr{S}_{\pmb{\cN}_d,r}$ has the same rational lamination as $S$, then $\widetilde{S} = S$. In particular, the combinatorial class of $S$ is a singleton.
\end{enumerate}
\end{theorem}

The proof uses rigidity of {\em complex box mappings}, which is a refinement of the notion of \emph{generalized polynomial-like maps} (cf. \cite{LM93}, \cite{Lyu97}). Although the proof is almost identical to the polynomial setting (see \cite{KvS09}, cf. \cite{CDKvS22}), there is a subtle difference from the polynomial setting, as the complement $\widehat\C - \overline{\mathcal{D}}$ of the domain of a Schwarz reflection $S: \overline{\mathcal{D}} \longrightarrow \widehat\C$ may be pinched (i.e., not Jordan), which requires some additional treatment.
\begin{defn}(\cite[Definition 1.2]{KvS09})
A holomorphic (or anti-holomorphic) map $F: U \longrightarrow V$ between open sets $U \subseteq V$ in $\C$ is a complex box mapping if
\begin{enumerate}
\item $F$ has finitely many critical points;
\item $V$ is a union of finitely many pairwise disjoint Jordan disks;
\item let $V'$ be a component of $V$, then either $V'$ is a component of $U$ or $V' \cap U$ is a union of Jordan disks with pairwise disjoint closures which are compactly contained in $V'$;
\item let $U'$ be a component of $U$, then $F(U')$ is a component of $V$.
\end{enumerate}
\end{defn}

\begin{proof}[Proof of Theorem \ref{thm:rs}]
Since $S$ is finitely renormalizable, we can choose the level $m$ puzzle $\mathfrak{P} = \mathfrak{P}^m$ so that $S$ is non-renormalizable with respect to $\mathfrak{P}$.
The same argument as in \cite[\S 2.2]{KvS09} allows us to construct a finite collection of open puzzle pieces $V$ that contains all critical points, such that the first return map into $V$ defines a complex box mapping $F:U \longrightarrow V$.

Let $\{a_1,..., a_r\}$ be the collection of pinching points for $\widehat\C - \overline{\mathcal{D}}$ where $\overline{\mathcal{D}}$ is the domain for $S$.
Note that each $a_i$ is a fixed point for $S$.
We may assume that $\{a_1,..., a_r\} \cap \overline{V} = \emptyset$.

Note that for each depth $k$, there are {at most} two puzzle pieces attached at $a_i$.
Since $S$ is non-renormalizable with respect to $\mathfrak{P}$, the puzzle pieces attached at $a_i$ contain no critical points at all sufficiently large depth $k$.
Since the local dynamics {on $\Lambda(S)$ near} these fixed points are parabolic repelling, these puzzle pieces shrink to points as the depth $k \to \infty$.

Now choose $k_0$ large enough so that these puzzle pieces have pairwise disjoint closure except at $a_i$ and their closure are also disjoint from $\overline{V}$.
Let $W$ be the union of all depth $k_0$ puzzle pieces attached to $a_i$ for some $i$. 

Then rigidity of complex box mappings \cite[Theorem 1.5]{KvS09} and the above definition of $W$ imply that all puzzle pieces in $V\cup W$ shrink to points, and it carries no measurable invariant line field on its Julia set.
Note that the set of points which are never mapped into $V \cup W$ is hyperbolic, so all puzzle pieces shrink to points and its area is $0$.
This proves $(1)$ and $(2)$ in Theorem \ref{thm:rs}.

Let $\widetilde{S} \in \mathscr{S}_{\pmb{\cN}_d,r}$ be a normalized polygonal Schwarz reflection that has the same lamination as $S$.
Then $\widetilde{S}$ is also finitely renormalizable.
Since the limit sets $\Lambda(S)$ and $\Lambda(\widetilde{S})$ are locally connected, and all cycles of $S, \widetilde{S}$ are repelling, {the maps $S\vert_{K(S)}$ and $\widetilde{S}\vert_{K(\widetilde{S})}$} are topologically conjugate.
By the same argument as in \cite[Proposition 4.1]{KvS09}, the corresponding complex box mappings $F:U \longrightarrow V$ and $\widetilde{F}: \widetilde{U} \longrightarrow \widetilde{V}$ are quasiconformally conjugate.
By spreading around the conjugacy (see the Spreading Principle \cite[\S 4]{KvS09}), we can promote this {local quasiconformal} conjugacy to a quasiconformal map $\Psi: \widehat\C \longrightarrow \widehat\C$ that conjugates the dynamics of $S$ and $\widetilde{S}$.
Since the external dynamics of $S, \widetilde{S}$ are conformally conjugate, and the Julia set $\Lambda(S)$ is a dendrite that does not support any invariant line field, $\Psi$ is a conformal map.
This proves $(3)$ in Theorem \ref{thm:rs}.
\end{proof}

\begin{proof}[Proof of Theorem \ref{thm:A}]
The theorem now follows from Theorem \ref{thm:CL} and Theorem \ref{thm:rs}.
\end{proof}


\section{Correspondences associated with polygonal Schwarz reflections}\label{polygonal_corr_sec}

Let $S:\overline{\mathcal{D}}\to\widehat{\C}$ be a conformal mating of a degree $d$ anti-polynomial $f$ with connected Julia set and the Nielsen map $\pmb{\cN}_d$. Recall that 
$$
\cD= \bigcup_{i=1}^k \Omega_i,
$$
where $\Omega_i$ are Jordan quadrature domains, and $S\vert_{\overline{\Omega_i}}$ is the Schwarz reflection map of $\Omega_i$.
In this section, we will associate anti-holomorphic correspondences with such matings, and show that they can also be interpreted as matings of anti-polynomials and ideal $(d+1)-$gon reflection groups. 

Let $R_i:\overline{\D}\to\overline{\Omega_i}$ be uniformizing rational maps. To make book-keeping easier, we denote the domain of $R_i$ (which is a copy of the Riemann sphere) by $\widehat{\C}_i$ and the unit disk in $\widehat{\C}_i$ by $\D_i$. We also set $\mathcal{I}:=\{1,\cdots,k\}$.

We recall that $\partial\cD$ meets the limit set $\Lambda(S)$ in the finite set $\mathfrak{S}$, which is the set of singular points on $\partial\cD$.
We define 
$$
\widetilde{\mathfrak{S}}_i\ := (R_i\vert_{\partial\D_i})^{-1}\left(\mathfrak{S} \cap \partial\Omega_i\right).
$$

\subsection{Construction of the correspondence}\label{corr_construct_subsec}

Let us now consider the disjoint union 
$$
\mathfrak{U}\ :=\ \bigsqcup_{i\in\mathcal{I}} \widehat{\C}_{i}\ \cong\ \widehat{\C}\times\mathcal{I},
$$
and define the maps
$$
\pmb{R}:\ \mathfrak{U}\longrightarrow \widehat{\C},\quad (z,i)\mapsto R_{i}(z),
$$
and
$$
\pmb{\eta^-}\ : \mathfrak{U}\longrightarrow \mathfrak{U},\quad (z,i)\mapsto (\eta^-(z),i).
$$
Note that $\pmb{R}$ is a branched covering of degree $d+1$, and $\pmb{\eta^-}$ is a homeomorphism.

The conformal mating $S$ gives rise to an anti-holomorphic correspondence on $\mathfrak{U}$ as follows. For $i,j\in\mathcal{I}$, define
$$
\mathfrak{C}_{i,j}:=\ \{(z,w)\in\widehat{\C}_i\times\widehat{\C}_j: \frac{R_j(w)-R_{i}(\eta^-(z))}{w-\eta^-(z)}=0\},\qquad \mathrm{if}\ i=j,
$$
and
$$
\mathfrak{C}_{i,j}:=\ \{(z,w)\in\widehat{\C}_i\times\widehat{\C}_j: R_j(w)-R_{i}(\eta^-(z))=0\},\qquad \mathrm{if}\ i\neq j.
$$
The union of the real-algebraic curves $\mathfrak{C}_{i,j}$ can be written succinctly as
\begin{equation}
\{(\mathfrak{u}_1,\mathfrak{u}_2)\in\mathfrak{U}\times\mathfrak{U}: \frac{\pmb{R}(\mathfrak{u}_2)-\pmb{R}(\pmb{\eta^-}(\mathfrak{u}_1))}{\mathfrak{u}_2-\pmb{\eta^-}(\mathfrak{u}_1)}=0\}.
\label{corr_gen_eqn}
\end{equation}
(The division in Equation~\eqref{corr_gen_eqn} makes sense since the numerator and the denominator can be viewed as points of $\widehat{\C}$.)
The first and second coordinate projection maps $\pi_1^i$ and $\pi_2^j$ from $\mathfrak{C}_{i,j}$ onto $\widehat{\C}_i$ and $\widehat{\C}_j$ define an anti-holomorphic correspondence (i.e., a multi-valued map with anti-holomorphic local branches) from $\widehat{\C}_i$ onto $\widehat{\C}_j$ (cf. \cite{DS06}):
\[
\begin{tikzcd}
 & \mathfrak{C}_{i,j} \arrow{dl}[swap]{\pi_1^i} \arrow{dr}{\pi_2^j} \\
\widehat{\C}_i && \widehat{\C}_j.
\end{tikzcd}
\]
Combining all these anti-holomorphic correspondences for various $i,j\in\mathcal{I}$, we obtain an anti-holomorphic correspondence on $\mathfrak{U}$ defined by the (in general reducible) curve $\sum_{i,j}\mathfrak{C}_{i,j}$. We denote this correspondence by $\mathfrak{C}^\circledast$.

The following observations show that $\mathfrak{C}^\circledast$ is obtained by lifting $S$ and its appropriate backward branches via the rational maps $R_i$. 
\smallskip

\noindent$\bullet$ Fix $z\in\overline{\D_i}$. Then, $S(R_i(z))=R_i(\eta^-(z))$, and hence,
\begin{equation}
((z,,i),(w,j))\in\mathfrak{C}^\circledast \iff R_j(w)=R_i(\eta^-(z))=S(R_i(z)),\ (w,j)\neq \pmb{\eta^-}(z,i).
\label{f_lift_eqn}
\end{equation}
\smallskip

\noindent$\bullet$ Now fix $z\in\widehat{\C}_i\setminus\overline{\D_i}$. Then, $S(R_i(\eta^-(z)))=R_i(z)$; i.e.,
\begin{equation}
((z,,i),(w,j))\in\mathfrak{C}^\circledast \iff R_j(w)=R_i(\eta^-(z))=S^{-1}(R_i(z)),\ (w,j)\neq \pmb{\eta^-}(z,i),
\label{f_inv_lift_eqn}
\end{equation}
\noindent where $S^{-1}$ is a suitable backward branch of $S$.

To capture our desired mating phenomenon, we need to pass to a correspondence induced by $\mathfrak{C}^\circledast$ on a quotient of $\mathfrak{U}$. To this end, we endow $\mathfrak{U}$ with the following finite equivalence relation:
\begin{center}
For $z\in\widetilde{\mathfrak{S}}_i\subset\widehat{\C}_i$ and $w\in\widetilde{\mathfrak{S}}_j\subset\widehat{\C}_j$,
\smallskip

$(z,i)\sim_{\mathrm{w}} (w,j)\iff R_i(z)=R_j(w)$.
\end{center}
\smallskip

The fact that $\overline{\cD}$ is the quotient of $\overline{\D}$ by a finite lamination and that $R_i\vert_{\partial\D_i}$ is injective (for all $i\in\mathcal{I}$) imply that 
$$
\mathfrak{W}\ :=\ \faktor{\mathfrak{U}}{\sim_{\mathrm{w}}}
$$
has a trivial fundamental group (it can be thought of as a compact, simply connected, noded Riemann surface). 
By definition, the map $\pmb{R}$ descends to a map from $\mathfrak{W}\longrightarrow\widehat{\C}$. This induced map, which we denote by $\widecheck{R}$, is also a branched covering. Abusing notation, we denote the image of a set $X\subset\mathfrak{U}$ (respectively, a point $p\in\mathfrak{U}$) under the quotient map $\mathfrak{U}\longrightarrow\mathfrak{W}$ by $X$ (respectively, $p$). We further set 
$$
\mathfrak{D}:=\bigcup_{i\in\mathcal{I}} \D_i,\quad \widetilde{\mathfrak{S}}:=\bigcup_{i\in\mathcal{I}}\widetilde{\mathfrak{S}}_{i}.
$$

\begin{lem}\label{eta_descends_lem}
The homeomorphism $\pmb{\eta^-}:\mathfrak{U}\to\mathfrak{U}$ descends to a homeomorphism $\widecheck{\eta}:\mathfrak{W}\longrightarrow\mathfrak{W}$.
\end{lem}
\begin{proof}
Let us suppose that $(z,i)\sim_{\mathrm{w}} (w,j)$; i.e., $z\in\widetilde{\mathfrak{S}}_i, w\in\widetilde{\mathfrak{S}}_j,$ and $R_i(z)=R_j(w)$. We need to show that $\pmb{\eta^-}(z,i)\sim_{\mathrm{w}} \pmb{\eta^-}(w,j)$. However, this is obvious since $\pmb{\eta^-}(z,i)=(\eta^-(z),i)=(z,i)\in\widetilde{\mathfrak{S}}_i$ and $\pmb{\eta^-}(w,j)=(\eta^-(w),j)=(w,j)~\in~\widetilde{\mathfrak{S}}_j$. 
\end{proof} 

Thus, the correspondence $\mathfrak{C}^\circledast$ on $\mathfrak{U}$ also descends to a correspondence on $\mathfrak{W}$. We denote this correspondence by $\mathfrak{C}$.

\subsection{Dynamical partition for $\mathfrak{C}$}\label{inv_partition_corr_subsec}

The invariant partition of the dynamical plane of $S$, given by $K(S)$ and $T^\infty(S)$, can be pulled back by $\widecheck{R}$ to produce an invariant partition of the dynamical plane of the correspondence $\mathfrak{C}$.
We define
$$
\mathcal{K}:=\widecheck{R}^{-1}(K(S)),\quad \mathcal{T}:= \widecheck{R}^{-1}(T^\infty(S)).
$$

\begin{prop}\label{corr_partition_prop}
\noindent\begin{enumerate}\upshape
\item $\widecheck{\eta}(\mathcal{T})=\mathcal{T}$, and $\widecheck{\eta}(\mathcal{K})=\mathcal{K}$.

\item Let $(\mathfrak{u}_1,\mathfrak{u}_2)\in\mathfrak{C}$. Then $\mathfrak{u}_1\in\mathcal{T}$ (respectively, $\mathfrak{u}_1\in\mathcal{K}$) if and only if $\mathfrak{u}_2\in\mathcal{T}$ (respectively, $\mathfrak{u}_2\in\mathcal{K}$).
\end{enumerate}
\end{prop}
\begin{proof}
1) It suffices to show that $\widecheck{\eta}(\mathcal{K})=\mathcal{K}$. 

Let us first assume that $(z,i)\in\overline{\D_i}\cap\mathcal{K}$. Then, $R_i(z)\in K(S)$ and $R_i(\eta^-(z))=S(R_i(z))$. As $K(S)$ is invariant under $S$, it follows that $R_i(\eta^-(z))\in K(S)$. We conclude that $\widecheck{\eta}(z,i)=(\eta^-(z),i)\in\mathcal{K}$.

Next let $(z,i)\in\mathcal{K}\setminus\overline{\D_i}$. Then, $R_i(z)\in K(S)$ and $S(R_i(\eta^-(z)))=R_i(z)$. As $K(S)$ is also backward invariant under $S$, it follows that $R_i(\eta^-(z))\in K(S)$, and hence $\widecheck{\eta}(z,i)=(\eta^-(z),i)\in\mathcal{K}$.

2) Let $(\mathfrak{u}_1,\mathfrak{u}_2)\in\mathfrak{C}$, where $\mathfrak{u}_1=(z,i)$ and $\mathfrak{u}_2=(w,j)$. It suffices to show that if $(z,i)\in\mathcal{T}$ (respectively, if $(z,i)\in\mathcal{K}$), then $(w,j)\in\mathcal{T}$ (respectively, $(w,j)\in\mathcal{K}$). 

To this end, first suppose that $z\in\overline{\D_i}$, which implies that $R_j(w)=R_i(\eta^-(z))=S(R_i(z))$. Now let $(z,i)\in\mathcal{T}$ (respectively, $(z,i)\in\mathcal{K}$). The $S$-invariance of $T^\infty(S)$ (respectively, of $K(S)$) implies that $\widecheck{R}(w,j)=R_j(w)\in T^\infty(S)$ (respectively, $\widecheck{R}(w,j)=R_j(w)\in K(S)$). Hence, $(w,j)\in\mathcal{T}$ (respectively, $(w,j)\in\mathcal{K}$). 

Next suppose that $z\in\widehat{\C}_i\setminus\overline{\D_i}$, which implies that $S(R_j(w))=R_i(z)$. Now let $(z,i)\in\mathcal{T}$ (respectively, $(z,i)\in\mathcal{K}$). The backward invariance of $T^\infty(S)$ (respectively, of $K(S)$) under $S$ implies that $\widecheck{R}(w,j)=R_j(w)\in T^\infty(S)$ (respectively, $\widecheck{R}(w,j)=R_j(w)\in K(S)$). Hence, $(w,j)\in\mathcal{T}$ (respectively, $(w,j)\in\mathcal{K}$). 
\end{proof}

\subsection{Polynomial structure of $\mathfrak{C}$}\label{corr_poly_subsec}

Note that $\widecheck{\eta}$ maps $\overline{\mathfrak{D}}$ onto $\mathfrak{W} \setminus \mathfrak{D}$, and preserves $\widetilde{\mathfrak{S}}$. 
We set $\mathcal{K}_1:=\mathcal{K}\cap \overline{\mathfrak{D}}$ and $\mathcal{K}_2:=\mathcal{K} \setminus \mathfrak{D}$.

\begin{lem}\label{lifted_ne_top_lem}
We have that $\mathcal{K}_2=\widecheck{\eta}(\mathcal{K}_1)$, and $\mathcal{K}_1\cap\mathcal{K}_2=\widetilde{\mathfrak{S}}$.
\end{lem}
\begin{proof}
The first statement follows from the fact that $\widecheck{\eta}(\mathcal{T})=\mathcal{T}$. 

For the second statement, first observe that
$$
\mathcal{K}_1\cap\mathcal{K}_2=\{\mathfrak{u}\in\partial \mathfrak{D}: \widecheck{R}(\mathfrak{u})\in K(S)\}.
$$
The result is now a consequence of the fact that $\widecheck{R}(\widetilde{\mathfrak{S}})=\mathfrak{S}\subset K(S)$ and $\widecheck{R}(\partial \mathfrak{D} \setminus\ \widetilde{\mathfrak{S}})=\partial\cD \setminus \mathfrak{S}\subset T^\infty(S)$.
\end{proof}

\begin{prop}\label{lifted_ne_dyn_gen_prop}
\noindent\begin{enumerate}\upshape
\item $\mathcal{K}_2$ is forward invariant, and hence, $\mathcal{K}_1$ is backward invariant under~$\mathfrak{C}$.

\item $\mathfrak{C}$ has a forward branch carrying $\mathcal{K}_1$ onto itself with degree $d$, and this branch is conformally conjugate to $S:K(S)\to K(S)$. 

\item $\mathfrak{C}$ has a backward branch carrying $\mathcal{K}_2$ onto itself with degree $d$, and this branch is topologically conjugate to $S:K(S)\to K(S)$.
\end{enumerate}
\end{prop}
\begin{proof}
1) This follows immediately from the definition of $\mathfrak{C}$ and Lemma~\ref{lifted_ne_top_lem}. 
 
2) The forward branch of $\mathfrak{C}$ carrying $\mathcal{K}_1$ onto itself (with degree $d$) acts as 
$$
\mathfrak{b}:\ \left(z,i\right)\mapsto \left(\widecheck{R}\vert_{\overline{\mathfrak{D}}}\right)^{-1}(\widecheck{R}(\widecheck{\eta}(z,i)))= \left(\left(R_{j}\vert_{\overline{\D_{j}}}\right)^{-1}\left(R_{i}(\eta^-(z))\right),j\right),
$$
where $(z,i)\in\mathcal{K}_1$, and $R_{i}(\eta^-(z))\in\overline{\Omega_j}$.
It is easy to see from the construction that $\widecheck{R}:\mathcal{K}_1\longrightarrow K(S)$ is a homeomorphism. We claim that $\widecheck{R}\vert_{\mathcal{K}_1}$ is the desired conjugating map. To this end, note that 
$$
S(\widecheck{R}(z,i))=S(R_i(z))=R_{i}(\eta^-(z)),
$$
and hence
$$
\widecheck{R}\left(\mathfrak{b}(z,i)\right)=\widecheck{R}\left(\left(\widecheck{R}\vert_{\overline{\mathfrak{D}}}\right)^{-1}(\widecheck{R}(\widecheck{\eta}(z,i)))\right)=R_{i}(\eta^-(z))=S(\widecheck{R}(z,i)).
$$

It follows that $\widecheck{R}\vert_{\mathcal{K}_1}\circ\mathfrak{b}=S\circ\widecheck{R}\vert_{\mathcal{K}_1}$.

3) This follows from the previous part and the fact that $\widecheck{\eta}$ conjugates a forward branch of the correspondence on $\mathcal{K}_1$ to a backward branch of the correspondence on $\mathcal{K}_2$.
\end{proof}

\subsection{Group structure of $\mathfrak{C}$}\label{corr_group_subsec}

The fact that $S: T^\infty(S)\cap\overline{\cD}\longrightarrow T^\infty(S)$ is conformally conjugate to the Nielsen map $\pmb{\cN}_d$ implies that $S$ has no critical value in $T^\infty(S)$. Hence, in light of the description of $S$ in terms of $\eta^-$ and $R_i$, we conclude that no $R_i$ has a critical value in $\mathcal{T}$.
Thus, $\widecheck{R}:\mathcal{T}\to T^\infty(S)$ is a covering map. Simple connectedness of $T^\infty(S)$ now implies that $\mathcal{T}$ is the union of $d+1$ many disjoint open topological disks $U_0,\cdots,U_d$.

We define a conformal automorphism
$$
\widecheck{\tau}:\mathcal{T}\to\mathcal{T}
$$ 
satisfying the conditions
\begin{enumerate}
\item $\widecheck{\tau}(U_j)=U_{j+1},\ j\in\Z/(d+1)\Z$,

\item $\widecheck{R}\circ\widecheck{\tau}=\widecheck{R}$, and

\item $\widecheck{\tau}^{d+1}=\mathrm{id}$,

\item $\widecheck{R}^{-1}(\widecheck{R}(z))=\{z,\widecheck{\tau}(z),\cdots,\widecheck{\tau}^d(z)\}\ \forall\ z\in \mathcal{T}$.
\end{enumerate}
By construction, the forward branches of $\mathfrak{C}$ on $\mathcal{T}$ are given by the anti-conformal automorphisms $\widecheck{\tau}\circ\widecheck{\eta},\cdots,\widecheck{\tau}^d\circ\widecheck{\eta}$. 

Recall the notation $T^0\equiv T^0(S)=\widehat{\C}\setminus\left(\cD\cup \mathfrak{S}\right)$. We call $T^0$ the rank zero tile of $T^\infty(S)$, and connected components of $S^{-n}(T^0)$ tiles of rank $n$. This yields a natural tessellation of the lifted tiling set $\mathcal{T}$.  A connected component of the pre-image of a rank $n$ tile of $T^\infty(S)$ under $\widecheck{R}$ is called a rank $n$ tile of~$\mathcal{T}$.

\begin{prop}\label{grand_orbit_group_gen_prop}
\noindent\begin{enumerate}\upshape
\item The grand orbits of the correspondence $\mathfrak{C}$ on $\mathcal{T}$ are equal to the orbits of $\langle\widecheck{\eta}\rangle\ast\langle\widecheck{\tau}\rangle$. 
\item The group $\langle\widecheck{\eta}\rangle\ast\langle\widecheck{\tau}\rangle$ acts properly discontinuously on $\cT$, and $\cT/\mathfrak{C}$ is biholomorphic to $\D/\pmb{G}_d$ (where $\pmb{G}_d$ is the regular ideal $(d+1)-$gon reflection group).
\end{enumerate}
\end{prop}
\begin{proof}
1) We have already observed that the forward branches of $\mathfrak{C}$ on $\mathcal{T}$ are given by $\widecheck{\tau}\circ\widecheck{\eta},\cdots,\widecheck{\tau}^d\circ\widecheck{\eta}$. Furthermore, $\widecheck{\tau}=(\widecheck{\tau}^{2}\circ\widecheck{\eta})\circ(\widecheck{\tau}\circ\widecheck{\eta})^{-1}$, and hence $\langle\widecheck{\tau}\circ\widecheck{\eta},\cdots,\widecheck{\tau}^d\circ\widecheck{\eta}\rangle=\langle\widecheck{\eta},\widecheck{\tau}\rangle$. It now remains to justify that $\langle\widecheck{\eta},\widecheck{\tau}\rangle$ is the free product of the cyclic groups $\langle\widecheck{\eta}\rangle$ and $\langle\widecheck{\tau}\rangle$.

To this end, first note that any relation in $\langle\widecheck{\eta},\widecheck{\tau}\rangle$ other than $\widecheck{\eta}^{2}=\mathrm{id}$ and $\widecheck{\tau}^{d+1}=\mathrm{id}$ can be reduced to one of the form 
\begin{equation}
(\widecheck{\tau}^{k_1}\circ\widecheck{\eta})\circ\cdots\circ(\widecheck{\tau}^{k_r}\circ\widecheck{\eta})=\mathrm{id}
\label{group_relation_1}
\end{equation} 
or 
\begin{equation}
(\widecheck{\tau}^{k_1}\circ\widecheck{\eta})\circ\cdots\circ(\widecheck{\tau}^{k_r}\circ\widecheck{\eta})=\widecheck{\eta},
\label{group_relation_2}
\end{equation}
where $k_1,\cdots,k_r\in\{1,\cdots,n\}$.
\smallskip

\noindent\textbf{Case 1:} Let us first assume that there exists a relation of the form (\ref{group_relation_1}) in $\langle\widecheck{\eta},\widecheck{\tau}\rangle$. It is readily seen using Relation~\eqref{f_inv_lift_eqn} that $(\widecheck{\tau}^{k_p}\circ\widecheck{\eta})$ maps the interior of a tile of rank $s$ in $\mathcal{T}\setminus \mathfrak{D}$ to the interior of a tile of rank $(s+1)$ in $\mathcal{T}\setminus \mathfrak{D}$. Hence, the group element on the left of Relation~(\ref{group_relation_1}) maps the tile of rank $0$ to a tile of rank $r$. Clearly, such an element cannot be the identity map.
\smallskip

\noindent\textbf{Case 2:} Now we consider a relation of the form (\ref{group_relation_2}) in $\langle\widecheck{\tau},\widecheck{\eta}\rangle$. Each $(\widecheck{\tau}^{k_p}\circ\widecheck{\eta})$ maps $\mathcal{T}\setminus \mathfrak{D}$ to itself. Hence, the group element on the left of Relation~(\ref{group_relation_2}) maps $\mathcal{T}\setminus \mathfrak{D}$ to itself, while $\widecheck{\eta}$ maps $\mathcal{T}\setminus \mathfrak{D}$ to $\mathcal{T}\cap \overline{\mathfrak{D}}$. This shows that there cannot exist a relation of the form (\ref{group_relation_2}) in $\langle\widecheck{\tau},\widecheck{\eta}\rangle$.

We conclude that $\widecheck{\eta}^{2}=\mathrm{id}$ and $\widecheck{\tau}^{d+1}=\mathrm{id}$ are the only relations in $\langle\widecheck{\eta},\widecheck{\tau}\rangle$, and hence $\langle\widecheck{\eta},\widecheck{\tau}\rangle=\langle\widecheck{\eta}\rangle\ast\langle\widecheck{\tau}\rangle\cong\Z/2\Z\ast\Z/(d+1)\Z$.

2) Recall that $\langle\widecheck{\tau}\rangle\leqslant\langle\widecheck{\eta}\rangle\ast\langle\widecheck{\tau}\rangle$ acts transitively on the components $U_0,\cdots,U_d$ of $\cT$. Hence it suffices to study the action of the stabilizer subgroup
$$
\mathrm{Stab}(U_0):=\{\gamma\in\langle\widecheck{\eta}\rangle\ast\langle\widecheck{\tau}\rangle: \gamma(U_0)=U_0\}
$$
of $U_0$ in $\langle\widecheck{\eta}\rangle\ast\langle\widecheck{\tau}\rangle$. It is easily seen that $\mathrm{Stab}(U_0)$ is generated by $\{\widecheck{\tau}^j\circ\widecheck{\eta}\circ \widecheck{\tau}^{-j},\ j\in\{0,\cdots,d\}\}$. 

We denote the rank zero tile $\widecheck{R}^{-1}(T^0)\cap U_0$ of $\cT$ in the component $U_0$ by $\widetilde{T^0_0}$. As $\widecheck{R}:U_0\to T^\infty(S)$ is a conformal isomorphism, it follows that $\widetilde{T^0_0}$ is a regular ideal $(d+1)-$gon in the conformal disk $U_0$. Further, each of the above generators of $\mathrm{Stab}(U_0)$ is an anti-conformal involution of $U_0$ which fixes a side of $\widetilde{T^0_0}$ pointwise. This implies that the $\mathrm{Stab}(U_0)-$action on $U_0$ is conformally conjugate to the $\pmb{G}_d-$action on $\D$. The result now follows.
\end{proof}

We are now ready to prove a part of Theorem~\ref{all_corr_thm}.

\begin{proof}[Proof of Theorem~\ref{all_corr_thm} (ideal polygon reflection group case)]
Let $f$ be a degree $d$ geometrically finite or finitely renormalizable, periodically repelling anti-polynomial with connected Julia set. By Theorem~\ref{conf_mating_polygonal_thm}, there exists a polygonal Schwarz reflection $S:\overline{\cD}\to\widehat{\C}$ that is a conformal mating of $f$ with the Nielsen map $\pmb{\cN}_d$. 

We claim that the desired correspondence $\mathfrak{C}$ on $\mathfrak{W}$ is the one constructed in Section~\ref{corr_construct_subsec}. To this end, note that the required $\mathfrak{C}-$invariant partition of $\mathfrak{W}$ is given in Proposition~\ref{corr_partition_prop}. Finally, the desired dynamical properties of $\mathfrak{C}$ on $\cT$ follow from Proposition~\ref{grand_orbit_group_gen_prop}, and those on $\cK$ follow from Lemma~\ref{lifted_ne_top_lem}, Proposition~\ref{lifted_ne_dyn_gen_prop}, and the fact that $S\vert_{K(S)}$ is topologically conjugate (conformally on the interior) to $f\vert_{\cK(f)}$.
\end{proof}

\section{Anti-Farey Schwarz reflections}\label{schwarz_anti_farey_sec}

We now look at degenerate anti-polynomial-like maps having the anti-Farey map $\pmb{\cF}_d$ (see Section~\ref{nielsen_farey_sec}) as their external map. Since $\pmb{\cF}_d:\overline{\D}\setminus\Int{\mathcal{H}}\to\overline{\D}$ fixes $\partial\mathcal{H}$ pointwise and since $\partial\mathcal{H}$ is a Jordan curve touching the unit circle at exactly one point, it follows that any such degenerate anti-polynomial-like map is the restriction of a Schwarz reflection map $S:\overline{\cD}\to\widehat{\C}$ of a simply connected quadrature domain $\cD$. Moreover, the tiling set dynamics of this Schwarz reflection map is conformally conjugate to $\pmb{\cF}_d$. 

Following \cite{LMM23}, we will denote by $\mathscr{S}_{\pmb{\cF}_d}$ the space of Schwarz reflection maps $S:\overline{\cD}\to\widehat{\C}$ such that $\cD$ is a Jordan quadrature domain and the tiling set dynamics of $S$ is conformally conjugate to $\pmb{\cF}_d$. We denote by $\mathfrak{U}_{d+1}$ the space of degree $d+1$ polynomials $p$ such that $p\vert_{\overline{\D}}$ is injective and $p$ has a unique (non-degenerate) critical point on $\mathbb{S}^1$.

\begin{prop}\cite[Proposition 3.3]{LMM23}\label{mating_equiv_cond_prop}
Let $R$ be a rational map of degree $d+1$ that is injective on $\overline{\D}$, $\cD:=R(\D)$, and $S$ the Schwarz reflection map associated with $\cD$. Then the following are equivalent.
\begin{enumerate}\upshape
\item There exists a conformal conjugacy $\psi$ between 
$$
\quad \pmb{\cF}_d:\overline{\D}\setminus \Int{\mathcal{H}}\longrightarrow \overline{\D}\quad \mathrm{and}\quad S:T^\infty(S)\setminus\Int{T^0(S)}\longrightarrow T^\infty(S).
$$
In particular, $T^\infty(S)$ is simply connected.

\item After possibly conjugating $S$ by a M{\"o}bius map and pre-composing $R$ with an element of $\mathrm{Aut}(\D)$, the uniformizing map $R$ can be chosen to be a polynomial with a unique critical point on $\mathbb{S}^1$; i.e., $R$ can be chosen to be a member of $\mathfrak{U}_{d+1}$. Moreover, $K(S)$ is connected. 

\item $\cD$ is a Jordan domain with a unique conformal cusp on its boundary. Moreover, $S$ has a unique critical point in its tiling set $T^\infty(S)$, and this critical point maps to $\Int{T^0(S)}$ with local degree $d+1$.
\end{enumerate}
\end{prop}

\begin{remark}
That fact that rational uniformizing map $R$ of the quadrature domain of a Schwarz reflection $S\in\mathscr{S}_{\pmb{\cF}_d}$ can be chosen to be a polynomial follows from the existence of a multiplicity $d$ critical point of $S$ in its tiling set (recall that $\pmb{\cF}_d$ has such a critical point). 
\end{remark}

\begin{prop}\cite[Proposition~3.10]{LMM23}\label{anti_farey_slice_compact_prop}

\noindent The moduli space $\left[\mathscr{S}_{\pmb{\cF}_d}\right]:=\  \faktor{\mathscr{S}_{\pmb{\cF}_d}}{\mathrm{PSL}_2(\C)}$ is compact.
\end{prop}

We normalize $S\in\mathscr{S}_{\pmb{\cF}_d}$ so that the conformal map $\psi:\D\to T^\infty(S)$ that conjugates $\pmb{\cF}_d$ to $S$ sends $0$ to $\infty$, and has asymptotics $z\mapsto 1/z+O(z)$ as $z\to 0$.

We say two maps $S_1, S_2\in\mathscr{S}_{\pmb{\cF}_d}$ are \emph{combinatorially equivalent}, denoted by $S_1\sim S_2$, if they have the same rational lamination (cf. Definition~\ref{def_preper_lami}). As in the case for the space $\mathscr{S}_{\pmb{\cN}_d}$, we denote by $\mathscr{S}_{\pmb{\cF}_d, r}\subset\mathscr{S}_{\pmb{\cF}_d}$ the subset of periodically repelling Schwarz reflections, and by $\widehat{\mathscr{S}_{\pmb{\cF}_d, r}} := \mathscr{S}_{\pmb{\cF}_d, r}/\sim$ the space of combinatorial classes.
Let $\mathscr{S}_{\pmb{\cF}_d, fr}\subset\mathscr{S}_{\pmb{\cF}_d}$ (respectively, $\mathscr{S}_{\pmb{\cF}_d, gf}\subset\mathscr{S}_{\pmb{\cF}_d}$) be the subset of at most finitely renormalizable periodically repelling Schwarz reflections (respectively, geometrically finite Schwarz reflections).

The proofs of Theorems~\ref{thm:CL},~\ref{thm:A}, and~\ref{thm:C} work mutatis mutandis for the space $\mathscr{S}_{\pmb{\cF}_d}$ of Schwarz reflections and yield the following result.

\begin{theorem}\label{farey_antipoly_thm}
\noindent\begin{enumerate}
\item There is a dynamically natural homeomorphism 
$$
\Phi: \widehat{\mathcal{C}^-_{d, r}} \longrightarrow \widehat{\mathscr{S}_{\pmb{\cF}_d, r}},
$$
where `dynamically natural' means that the circle homeomorphism conjugating $\pmb{\cF}_d$ to $\overline{z}^d$ carries the rational lamination of an anti-polynomial to the rational lamination of a Schwarz reflection.

\item There is a dynamically natural homeomorphism 
$$
\Phi: \mathcal{C}^-_{d, fr}  \longrightarrow \mathscr{S}_{\pmb{\cF}_d, fr}.
$$
Here, dynamically natural means that for any $f\in \mathcal{C}^-_{d, fr}$, $\Phi(f)$ is the unique conformal mating of $f$ with the anti-Farey map $\pmb{\cF}_d$.

\item There is a dynamically natural bijection 
$$
\Phi: \mathcal{C}^-_{d, gf} \longrightarrow \mathscr{S}_{\pmb{\cF}_d, gf}.
$$ 
Here, dynamically natural has the same meaning as in item (2) of this theorem.
\end{enumerate}
\end{theorem}

{
\begin{remark}
We believe that the bijection $\Phi: \mathcal{C}^-_{d, gf} \longrightarrow \mathscr{S}_{\pmb{\cF}_d, gf}$ is not continuous in general. Although the proof of discontinuity given in Subsection~\ref{discont_fig} does not apply to the family $\mathscr{S}_{\pmb{\cF}_d}$ (essentially because all Schwarz reflections in $\mathscr{S}_{\pmb{\cF}_d}$ are defined by a single quadrature domain), one possible way of proving discontinuity in the current setting is to use the arguments of \cite[Lemma~8.9]{LLMM3}.
\end{remark}
}

We end this section with the proof of Theorem~\ref{all_corr_thm} in the case of the anti-Hecke group.

\begin{proof}[Proof of Theorem~\ref{all_corr_thm} (anti-Hecke group case)]
Let $f$ a degree $d$ anti-polynomial as in the statement of the theorem. By Theorem~\ref{conf_mating_antiFarey_thm}, there exists a Schwarz reflection $S=\Phi(f):\overline{\cD}\to\widehat{\C}$ in $\mathscr{S}_{\pmb{\cF}_d}$ such that the non-escaping set dynamics of $S$ is topologically conjugate (conformally on the interior) to $f\vert_{\cK(f)}$ and the tiling set dynamics of $S$ is conformally conjugate to $\pmb{\cF}_d$.

By Propositions~\ref{simp_conn_quad_prop} and~\ref{mating_equiv_cond_prop}, after possibly M{\"o}bius conjugating $S$ (which amounts to replacing $\cD$ by a M{\"o}bius image of it), we can assume that $p(\D)=\cD$, for some polynomial $p\in\mathfrak{U}_{d+1}$. We note that $S\equiv p\circ \eta^-\circ (p\vert_{\overline{\D}})^{-1}$.

We now define the antiholomorphic correspondence $\mathfrak{C}$ on $\widehat{\C}$ by the relation
\begin{equation*}
(z,w)\in\mathfrak{C}\iff \frac{p(w)-p(\eta^-(z))}{w-\eta^-(z)}=0,
\end{equation*} 
and set $\cK:=p^{-1}(K(S))$, $\cT:=p^{-1}(T^\infty(S))$. By \cite[Proposition~2.5]{LMM23}, the sets $\cK$ and $\cT$ are completely invariant under the action of the correspondence. 

Note that by construction of $S$, the tiling set $T^\infty(S)$ is simply connected, and it contains exactly one critical point of $S$. More precisely, this critical point lies at $p(0)$, and it maps to $\infty$ with local degree $d+1$ under $S$. Thus, by \cite[Proposition~2.18]{LMM23}, the dynamics of $\mathfrak{C}$ on $\cT$ is conformally orbit-equivalent to the action of the anti-Hecke group $\mathbbm{G}_d$ on $\D$.

Next, we set $\cK_1:=\cK\cap\overline{\D}$ and $\cK_2:=\cK\setminus\D$. Since $p$ carries $\mathbb{S}^1\setminus\{1\}$ onto $\partial\cD\setminus\{\pmb{s}\}\subset T^\infty(S)$ (where $\pmb{s}$ is the unique cusp of $\partial\cD$), it follows that $\mathbb{S}^1\setminus\{1\}$ is disjoint from $\cK$. As $\pmb{s}\in K(S)$, it follows that $(p\vert_{\mathbb{S}^1})^{-1}(\pmb{s})$ is the unique point of intersection of $\cK_1$ and $\cK_2$. 

By \cite[Proposition~2.6]{LMM23}, $\mathfrak{C}$ has a forward branch carrying $\cK_1$ onto itself with degree $d$, and this branch is topologically conjugate to $S\vert_{K(S)}$ such that the conjugacy is conformal on the interior. The result now follows from the fact that $S\vert_{K(S)}$ is topologically conjugate (conformally on the interior) to $f\vert_{\cK(f)}$.
\end{proof}

\begin{LARGE}\part{Holomorphic world}\end{LARGE}\label{part_two}
\bigskip

\section{Factor Bowen-Series maps for genus zero orbifolds}\label{fbs_sec}

We will now describe a collection of circle coverings that were introduced in \cite{MM2}.
Let us first define a class of orbifolds that give rise to these external maps:\\
 \noindent $\mathfrak{F}:=$ hyperbolic orbifolds $\Sigma$ of genus zero with
	\begin{enumerate}
	\item at least one puncture,
	\item at most one order two orbifold point,
	\item at most one order $\nu\geq 3$ orbifold point.
	\end{enumerate}
We set
\begin{align*}
n=
\begin{cases}
\nu \quad \textrm{if } \Sigma\in\mathfrak{F}\ \textrm{ has an order } \nu\geq 3 \textrm{ orbifold point},\\
1 \quad \mathrm{otherwise}.
\end{cases}
\end{align*}
Each $\Sigma\in\mathfrak{F}$ has an $n-$fold cyclic cover $\widetilde{\Sigma}$. Specifically, if $\Sigma$ has $\delta_1\geq1$ punctures and $\delta_2\in\{0,1\}$ order two orbifold points, then $\widetilde{\Sigma}$ is a genus zero orbifold with $n(\delta_1-1)+1$ punctures and $n\delta_2$ order two orbifold points. Topologically, $\widetilde{\Sigma}$ is obtained by skewering the surface $\Sigma$ along an infinite geodesic (in $\Sigma$) connecting the order $\nu$ orbifold point and a cusp, and gluing $n$ copies of it cyclically.

The surface $\widetilde{\Sigma}$ is uniformized by a Fuchsian group $\Gamma$ which admits a (closed) ideal $m-$gon $\Pi$ as a fundamental domain. This fundamental domain $\Pi$ can be chosen so that it has ideal vertices at $1$ and $\exp(2\pi i/n)$ which are identified, and such that $\Pi$ is symmetric under rotation by $2\pi/n$ around the origin. Here, 
\begin{align*}
m=
\begin{cases}
2n(\delta_1-1) \quad \textrm{ when }\ \delta_2=0,\\ 
n(2\delta_1-1) \quad \textrm{ when }\ \delta_2=1.
\end{cases}
\end{align*}

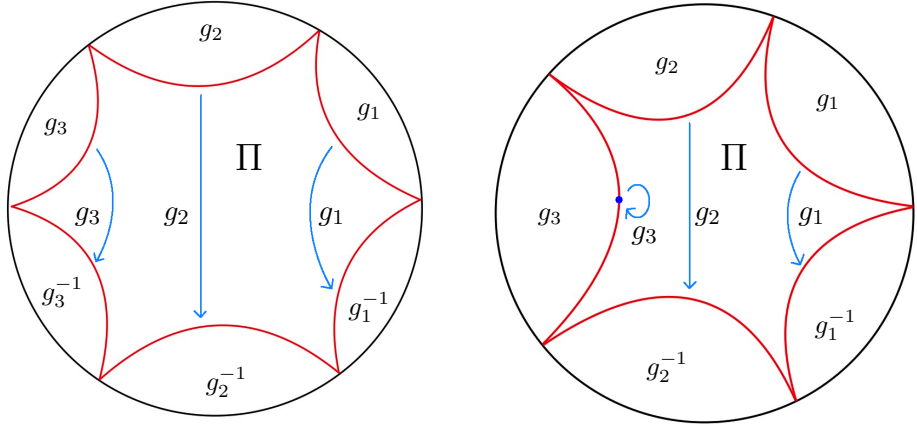
\begin{figure}[ht]
\captionsetup{width=0.96\linewidth}
	\begin{tikzpicture}
		\node[anchor=south west,inner sep=0] at (0.5,0) {\includegraphics[width=0.45\linewidth]{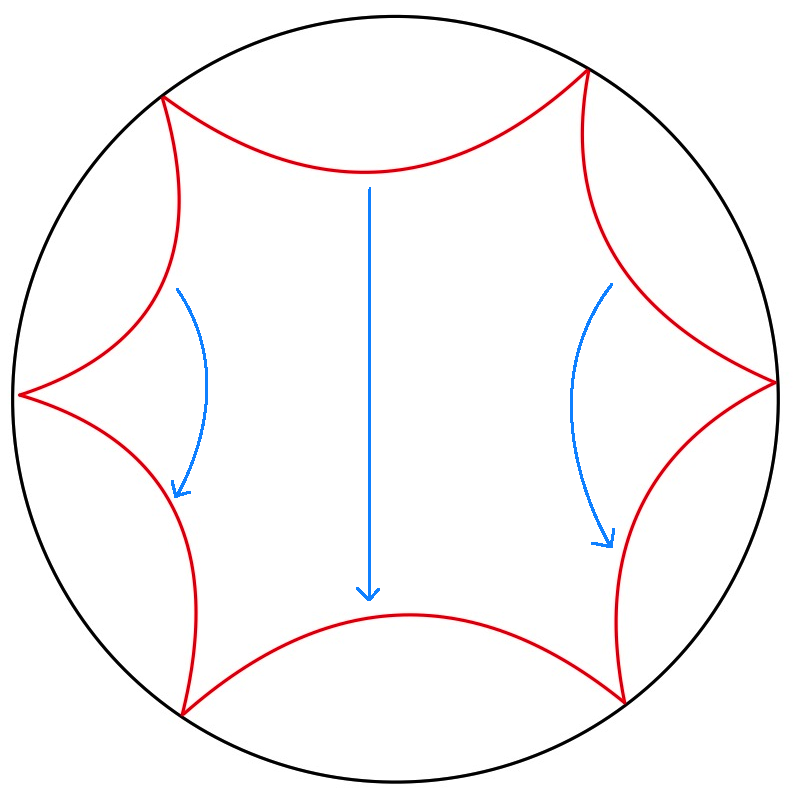}};
		\node[anchor=south west,inner sep=0] at (7,0) {\includegraphics[width=0.45\linewidth]{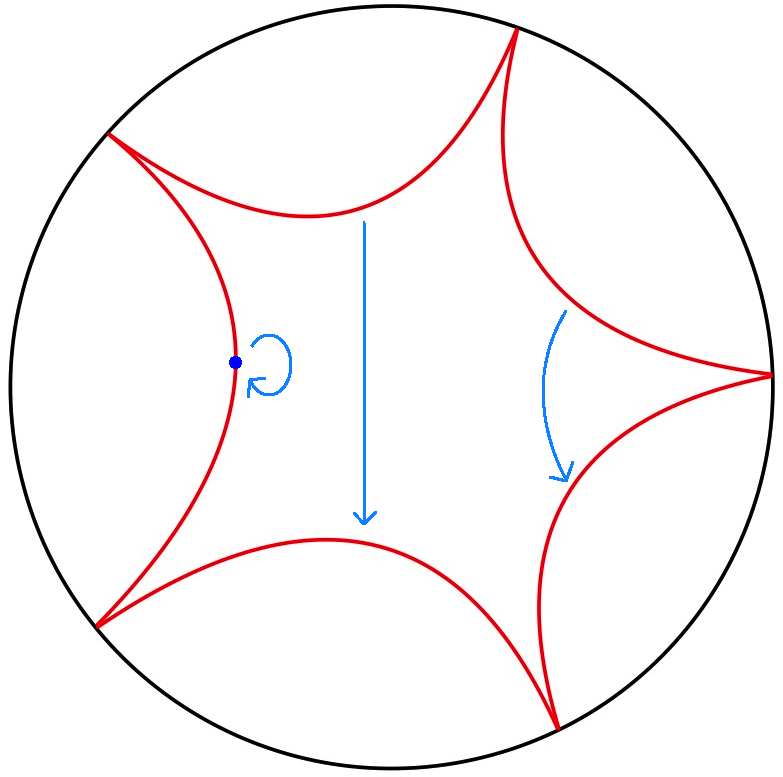}};
		\node at (3.8,3.6) {\begin{Large}$\Pi$\end{Large}};
		\node at (4.9,2.8) {$g_1$};
		\node at (2.84,2.8) {$g_2$};
		\node at (1.66,2.8) {$g_3$};
		\node at (5.4,4.2) {\begin{small}$g_{1}$\end{small}};
		\node at (5.4,1.6) {\begin{small}$g_{1}^{-1}$\end{small}};
		\node at (3.3,5.25) {\begin{small}$g_{2}$\end{small}};
		\node at (3.5,0.6) {\begin{small}$g_{2}^{-1}$\end{small}};
		\node at (1.24,4) {\begin{small}$g_{3}$\end{small}};
		\node at (1.32,1.8) {\begin{small}$g_{3}^{-1}$\end{small}};	
			
		\node at (10.25,3.6) {\begin{Large}$\Pi$\end{Large}};
		\node at (11.28,2.8) {$g_1$};
		\node at (9.9,2.8) {$g_2$};
		\node at (9.06,2.55) {$g_3$};
		\node at (11.5,4.32) {\begin{small}$g_{1}$\end{small}};
		\node at (11.6,1.36) {\begin{small}$g_{1}^{-1}$\end{small}};
		\node at (9.36,4.8) {\begin{small}$g_{2}$\end{small}};
		\node at (9.36,0.8) {\begin{small}$g_{2}^{-1}$\end{small}};
		\node at (7.8,2.8) {\begin{small}$g_{3}$\end{small}};
		
	\end{tikzpicture}
	\caption{Pictured are the preferred fundamental domains $\Pi$ and the actions of the associated Bowen-Series maps $A^{\textrm{BS}}_{\Sigma}:\overline{\D}\setminus\Int{\Pi}\to\overline{\D}$ for a four times punctured sphere $\Sigma$ (left) and a thrice punctured sphere $\Sigma$ with an order two orbifold point (right). Note that $\Sigma=\widetilde{\Sigma}$ in these cases.}
	\label{punc_sphere_orbifold_bs_fig}
\end{figure}

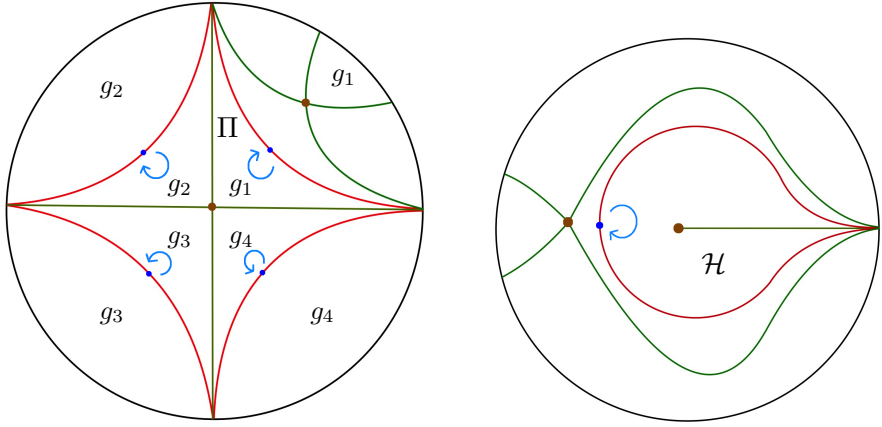
\begin{figure}[ht]
\captionsetup{width=0.98\linewidth}
\begin{tikzpicture}
\node[anchor=south west,inner sep=0] at (0.5,0) {\includegraphics[width=0.45\linewidth]{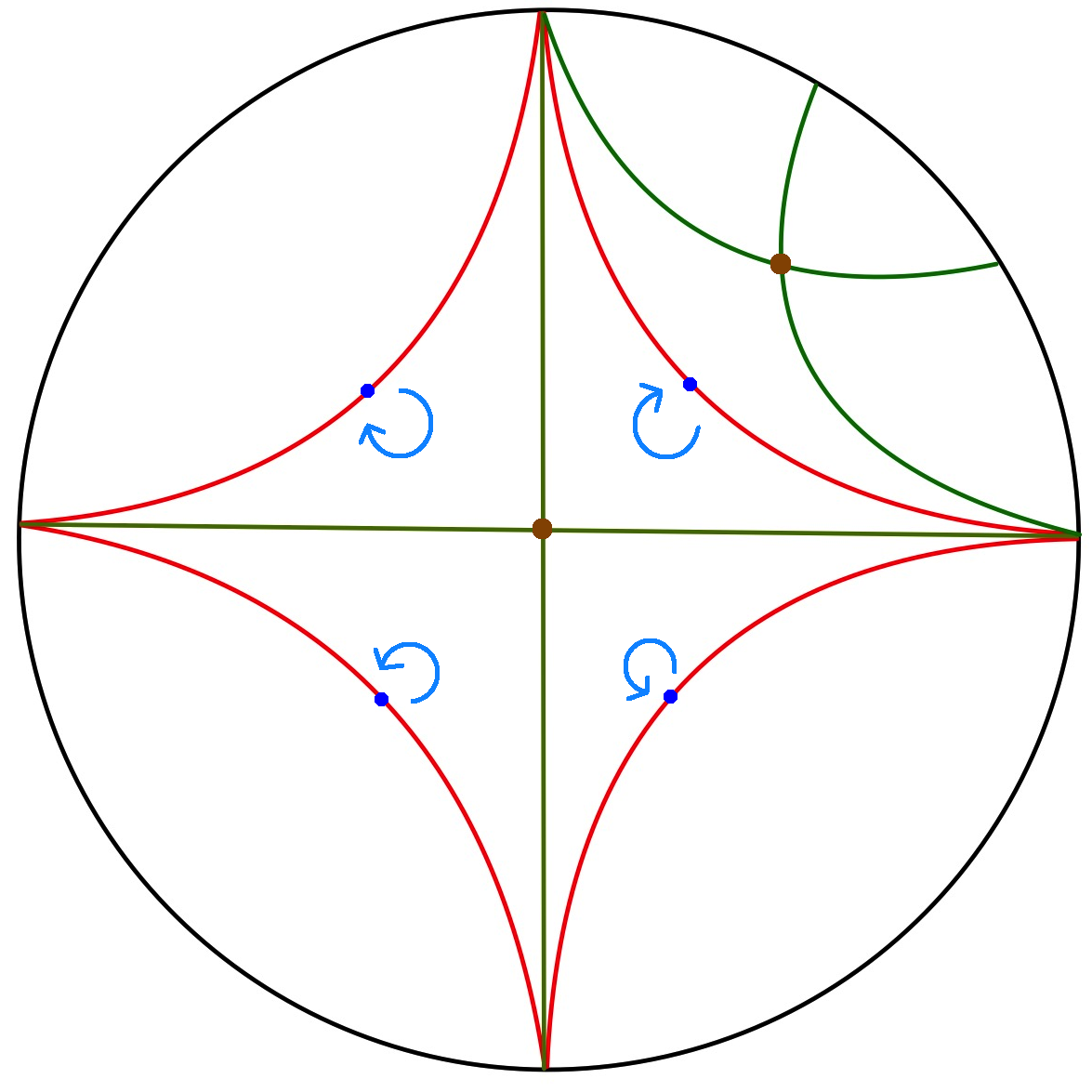}};
\node[anchor=south west,inner sep=0] at (7,0) {\includegraphics[width=0.42\linewidth]{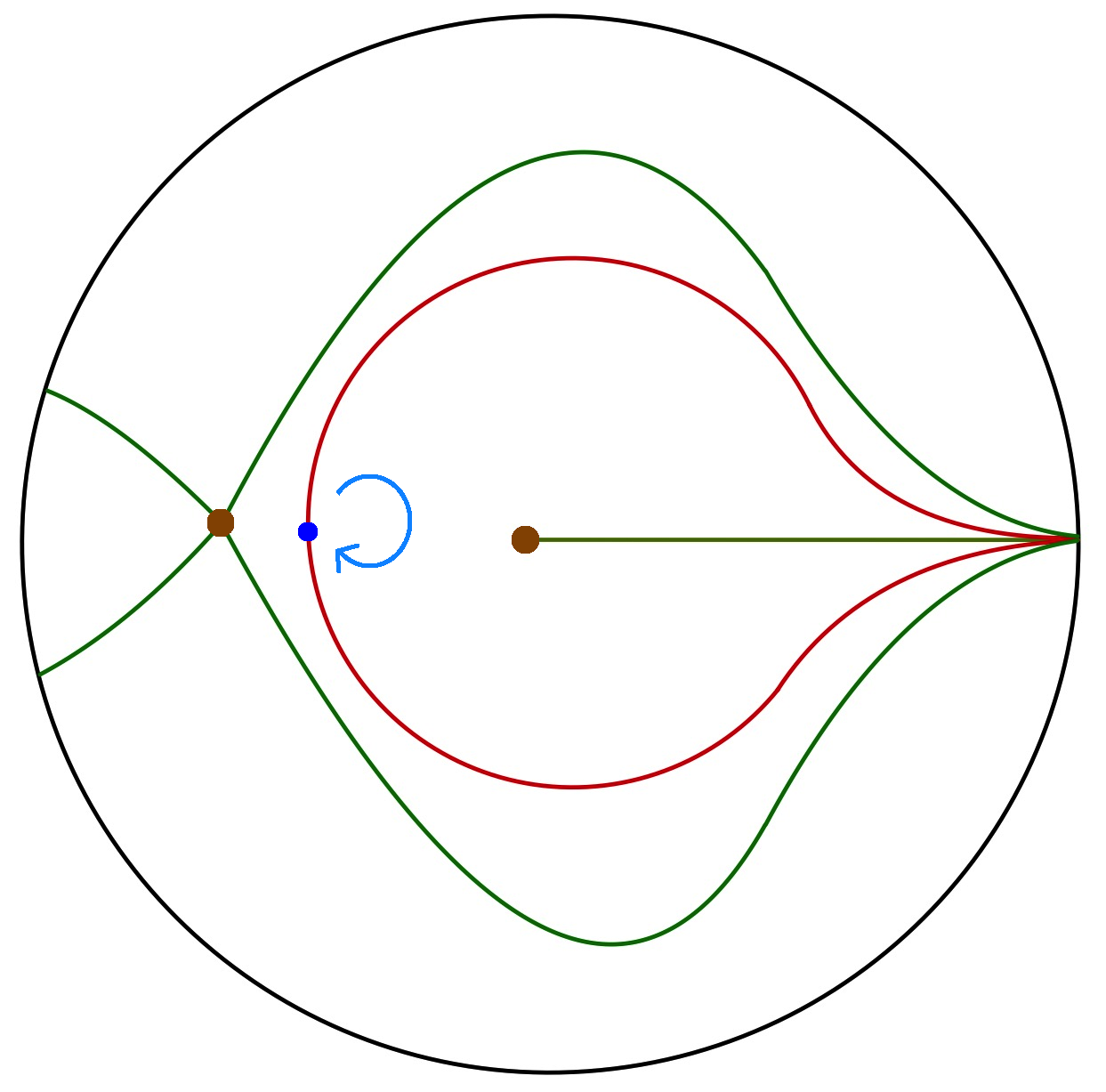}};
\node at (5.08,4.66) {\begin{small}$g_{1}$\end{small}};
\node at (2,4.5) {\begin{small}$g_{2}$\end{small}};
\node at (2,1.5) {\begin{small}$g_{3}$\end{small}};
\node at (4.8,1.5) {\begin{small}$g_{4}$\end{small}};
\node at (3.72,3.2) {\begin{small}$g_{1}$\end{small}};
\node at (2.9,3.2) {\begin{small}$g_{2}$\end{small}};
\node at (2.9,2.5) {\begin{small}$g_{3}$\end{small}};
\node at (3.72,2.5) {\begin{small}$g_{4}$\end{small}};
\node at (3.54,4) {$\Pi$};
\node at (10,2.2) {$\mathcal{H}$};
\end{tikzpicture}
\caption{For the Hecke surface $\Sigma$ with an order $4$ orbifold point, the cyclic cover $\widetilde{\Sigma}$ is a sphere with one puncture and four order two orbifold points. Left: The preferred fundamental domain $\Pi$ and the action of the associated Bowen-Series map $A^{\textrm{BS}}_{\widetilde{\Sigma}}$ for $\widetilde{\Sigma}$.
The Bowen-Series map $A^{\textrm{BS}}_{\widetilde{\Sigma}}$ commutes with rotation by $\pi/2$.
 The vertical and horizontal radial lines in $\D$ and their pre-images under $g_{1}$ are displayed in green. 
 Right: The factor Bowen-Series map $A_{\Sigma}^{\mathrm{f}}:\overline{\D}\setminus\Int{\mathcal{H}}\to\overline{\D}$, where $\mathcal{H}:=\Pi/\langle M_i\rangle$ is an ideal monogon with its vertex at $1$, restricts to a degree three covering of $\mathbb{S}^1$. It has a unique critical point of multiplicity three at the valence four vertex of the green graph.}
\label{hecke_factor_bs_fig}
\end{figure}

The Bowen-Series map $A^{\textrm{BS}}_{\widetilde{\Sigma}}\equiv A^{\textrm{BS}}_\Gamma:\overline{\D}\setminus\Int{\Pi}\to\overline{\D}$ of $\Gamma$ equipped with the fundamental domain $\Pi$ has jump discontinuities at the $n-$th roots of unity. However, the $2\pi/n$ rotation symmetry of $\Pi$ and of the associated side-pairing transformations imply that $A^{\textrm{BS}}_{\widetilde{\Sigma}}$ commutes with rotation by $2\pi/n$. This allows one to pass to a factor of the above Bowen-Series map on the quotient cone $\D/\langle M_\omega\rangle$ (where $M_\omega(z)=\omega z$, and $\omega=e^{2\pi i/n}$). The resulting map is denoted by 
$$
\widehat{A}^{\textrm{BS}}_{\widetilde{\Sigma}}: \left(\overline{\D}\setminus\Int{\Pi}\right)/\langle M_\omega\rangle\to\overline{\D}/\langle M_\omega\rangle.
$$ 
We uniformize the cone $\overline{\D}/\langle M_\omega\rangle$ by the closed disk $\overline{\D}$ and call this uniformizing map $\xi:\overline{\D}/\langle M_\omega\rangle\to\overline{\D}$ (note that $\xi$ is induced by $z\mapsto z^n$). Finally, we define the \emph{factor Bowen-Series map} associated with $\Sigma$ as
$$
A^{\textrm{f}}_{\Sigma}:=\xi\circ \widehat{A}^{\textrm{BS}}_{\widetilde{\Sigma}} \circ\xi^{-1}: \overline{\D}\setminus\Int{\mathcal{H}}\to\overline{\D},
$$ 
where $\mathcal{H}:=\xi(\Pi/\langle M_\omega\rangle)$. The set $\mathcal{H}$ has $p:=m/n$ ideal boundary points on~$\mathbb{S}^1$, and the map $A_\Sigma^{\mathrm{f}}$ restricts to a self-homeomorphism of order two on $\partial\mathcal{H}$.

\begin{remark}\label{canonical_dom_rem}
1) In the anti-holomorphic world, a Nielsen map or an anti-Farey map has a unique domain of definition which is bounded by the unit circle and curves comprising fixed points of the map. However, this is not the case for factor Bowen-Series maps since there is no unique collection of curves in the disk on which a factor Bowen-series map acts as an involution. As a convention, we will fix $\overline{\D}\setminus\Int{\mathcal{H}}$ as the \emph{canonical domain of definition} of the factor Bowen-Series map $A_\Sigma^{\mathrm{f}}$.

2) {The superscript `f' in $A^{\textrm{f}}_{\Sigma}$ stands for `factor'.}
\end{remark}

The factor Bowen-Series map $A^{\textrm{f}}_{\Sigma}$ turns out to be an expansive covering map of $\mathbb{S}^1$ of degree  
\begin{align*}
d\equiv d(\Sigma)= m-1 = \hspace{2mm}
\begin{cases}
2n(\delta_1-1)-1 \quad \mathrm{when}\quad  \delta_2=0,\\
n(2\delta_1-1)-1 \quad \mathrm{when}\quad \delta_2=1.
\end{cases}
\end{align*} 
Moreover, when $n\geq 3$, the map $A_\Sigma^{\mathrm{f}}$ has $p=m/n$ critical points, each of multiplicity $n-1$.
All these critical points are mapped to the same critical value.

Two special classes of examples of factor Bowen-Series maps are given by the cases 
$$
n=1,\ \delta_1\geq 2,\ \delta_2\in\{0,1\},
$$ and 
$$
n\geq 3,\ \delta_1=1,\ \delta_2=1.
$$ 
In the former case, $\Sigma$ is a punctured sphere with possibly an order two orbifold point, and $\widetilde{\Sigma}=\Sigma$. In particular, the factor Bowen-Series map of $\Sigma$ is simply the Bowen-Series map of $\Sigma$ (see Figure~\ref{punc_sphere_orbifold_bs_fig}). In the latter case, $\Sigma$ is a Hecke surface, and $\widetilde{\Sigma}$ is a sphere with one puncture and $\nu\geq 3$ order two orbifold points (see Figure~\ref{hecke_factor_bs_fig}).

\subsection{Conformal mating of factor Bowen-Series maps with polynomials}\label{fbs_poly_mating_subsec}

We continue to use the notation used above.
Let $f$ be a monic, centered complex polynomial of degree $d$ with a connected and locally connected Julia set. We denote the \emph{B{\"o}ttcher coordinate} of $f$ by
$$
\psi_f:\widehat{\C}\setminus\overline{\D}\to\mathcal{B}_\infty(f):=\widehat{\C}\setminus\mathcal{K}(f).
$$ 
This map, which conjugates $z^{d}$ to $f$, is normalized to be tangent to the identity map near infinity.
As $\mathcal{J}(f)$ is locally connected, $\psi_f$ extends continuously to $\mathbb{S}^1$ to yield a semi-conjugacy between $z^d\vert_{\mathbb{S}^1}$ and $f\vert_{\mathcal{J}(f)}$. 
By \cite[Proposition~2.5]{MM2}, there exists a homeomorphism $\mathfrak{h}:\mathbb{S}^1\to\mathbb{S}^1$ that conjugates $z^d$ to $A_{\Sigma}^{\mathrm{f}}$. We normalize $\mathfrak{h}$ so that it sends the fixed point $1$ of $z^d$ to the fixed point $1$ of $A_{\Sigma}^{\mathrm{f}}$.

We define an equivalence relation $\sim$ on $\mathcal{K}(f)\bigsqcup \overline{\D}$ generated by 
\begin{equation}
\psi_{f}(\zeta)\sim \mathfrak{h}(\overline{\zeta}),\ \textrm{for all}\ z\in\mathbb{S}^1.
\label{conf_mat_equiv_rel}
\end{equation}

\begin{defn}\label{conf_mat_def}
The maps $f$ and $A_{\Sigma}^{\mathrm{f}}$ are said to be \emph{conformally mateable} if there exist a continuous map $S: \mathrm{Dom}(S)\subsetneq\widehat{\C}\to\widehat{\C}$ (called a \emph{conformal mating} of $A_{\Sigma}^{\mathrm{f}}$ and $f$) that is complex-analytic in the interior of $\mathrm{Dom}(S)$ and continuous maps 
	$$
	\mathfrak{X}_f:\mathcal{K}(f)\to\widehat{\C}\ \textrm{and}\ \mathfrak{X}_\Sigma: \overline{\D}\to\widehat{\C},
	$$
	conformal on $\Int{\mathcal{K}(f)}$ and $\D$ (respectively), satisfying
	\begin{enumerate}
		\item\label{topo_cond} $\mathfrak{X}_f\left(\mathcal{K}(f)\right)\cup \mathfrak{X}_\Sigma\left(\overline{\D}\right) = \widehat{\C}$,
		
		\item\label{dom_cond} $\mathrm{Dom}(S)= \mathfrak{X}_f(\mathcal{K}(f))\cup\mathfrak{X}_\Sigma(\overline{\D}\setminus\Int{\mathcal{H}})$,
		
		\item $\mathfrak{X}_f\circ f(z) = S\circ \mathfrak{X}_f(z),\quad \mathrm{for}\ z\in\mathcal{K}(f)$,
		
		\item $\mathfrak{X}_\Sigma\circ A_{\Sigma}^{\mathrm{f}}(w) = F\circ \mathfrak{X}_\Sigma(w),\quad \mathrm{for}\ w\in
		\overline{\D}\setminus\Int{\mathcal{H}}$,\quad and

		\item\label{identifications} $\mathfrak{X}_f(z)=\mathfrak{X}_\Sigma(w)$ if and only if $z\sim w$ where $\sim$ is the equivalence relation on $\mathcal{K}(f)\sqcup \overline{\D}$ defined by Relation~\eqref{conf_mat_equiv_rel}.
	\end{enumerate}
\end{defn}

\begin{theorem}\cite[Theorem~3.2]{MM2}\label{fbs_poly_conf_mat_thm}
Let $f$ be a hyperbolic complex polynomial of degree $d$ with connected Julia set.
Then, there exists a conformal mating $S$ of $f:\mathcal{K}(f)\to\mathcal{K}(f)$ and $A_{\Sigma}^{\mathrm{f}}:\overline{\D}\setminus\Int{\mathcal{H}}\to \overline{\D}$. Moreover, $S$ is a unique up to M{\"o}bius conjugacy.
\end{theorem}

\section{B-involutions}\label{b_inv_sec}

We now introduce a holomorphic analogue of Schwarz reflections.

\begin{defn}\label{b_inv_def}
Let $\{\Omega_1,\cdots,\Omega_k\}$ be a disjoint collection of proper simply connected sub-domains of $\widehat{\C}$ such that $\Int{\overline{\Omega_j}}=\Omega_j$, $j\in\{1,\cdots, k\}$, and let $\cD:=\displaystyle\bigsqcup_{j=1}^k\Omega_j$. Further, let $\mathfrak{S}\subset\partial\cD$ be a (possibly empty) finite set such that $\partial^0\cD:=\partial\cD\setminus\mathfrak{S}$ is a finite union of disjoint non-singular real-analytic curves.

\noindent The set $\cD$ is called an \emph{inversive multi-domain} if it admits a continuous map $S:\overline{\cD}\to\widehat{\C}$ satisfying the properties:
\begin{enumerate}
\item\label{mero_cond} $S$ is meromorphic on $\cD$,
\item\label{permute_cond} $S(\partial\Omega_j)=\partial\Omega_{j'}$, for some $j'\in\{1,\cdots,k\}$, and
\item\label{inv_cond} $S:\partial\cD\to\partial\cD$ is an orientation-reversing involution preserving $\mathfrak{S}$; i.e., $S(\mathfrak{S})=\mathfrak{S}$.
\end{enumerate}
The map $S$ is called a \emph{B-involution} of the inversive multi-domain $\cD$.

\noindent When $k=1$, the domain $\cD$ is called an \emph{inversive domain}.
\end{defn}

For a B-involution $S:\overline{\cD}\to\widehat{\C}$, we set $T:=\widehat{\C}\setminus\cD$, and $T^0:= T\setminus\mathfrak{S}$.

\begin{remark}\label{jordan_rem}
1) If each $\Omega_j$ is a Jordan domain, then Condition~\ref{permute_cond} of Definiion~\ref{b_inv_def} follows from Condition~\ref{inv_cond}.

2) The nomenclature `B-involution' is chosen to emphasize the property that such a map induces an involution on the boundary of an inversive multi-domain $\cD$.
\end{remark}

The main result of this section is an explicit algebraic description of B-involutions, which can be thought of as a holomorphic counterpart of Proposition~\ref{simp_conn_quad_prop}. We first prove this characterization for an inversive domain.
Recall that $\eta^+(z)=1/z$.

\begin{lem}\label{b_inv_lem_1}
Let $\cD$ be an inversive domain with associated B-involution $S:\overline{\cD}\to\widehat{\C}$. 
Then, there exists a rational map $R$ of degree $d+1$ and a Jordan disk $\mathfrak{D}$ such that the following hold.
\begin{enumerate}\upshape
\item $1, -1\in\partial\mathfrak{D}$.
\item $\eta^+:\mathfrak{D}\to\widehat{\C}\setminus\overline{\mathfrak{D}}$ is a homeomorphism.
\item $\partial\mathfrak{D}$ is a piecewise non-singular real-analytic curve.
\item $R:\mathfrak{D}\to\cD$ is a conformal isomorphism.
\item $S\vert_{\cD}:= R\circ\eta^+\circ (R\vert_{\mathfrak{D}})^{-1}$.
\item $S:S^{-1}(\cD)\to\cD$ is a proper map of degree $d$.
\item $S:S^{-1}(\Int{T})\to \Int{T}$ is a proper map of degree $d+1$.
\end{enumerate}
\end{lem}

\begin{proof}
\textbf{Step 1: Extending $S$ to a conformal involution around $\partial^0\cD$.}
By assumption, $S:\partial\cD\to\partial\cD$ is an order two homeomorphism preserving $\mathfrak{S}$. Hence, it carries each component of $\partial^0\cD$ (which is a non-singular real-analytic curve) onto some component of $\partial^0\cD$. Depending on the parity of the number of components of $\partial^0\cD$, no component or exactly one component of $\partial^0\cD$ is mapped onto itself by $S$. Since $S$ is meromorphic on $\cD$ and continuous up to $\partial\cD$, it now follows by the Schwarz Reflection Principle that $S$ extends to a holomorphic map (which we call $S$) in a neighborhood of $\partial^0\cD$. Since $S$ carries $\partial^0\cD$ onto itself, the map $S^{\circ 2}$ is also well-defined in a neighborhood of $\partial^0\cD$. Moreover, as $S^{\circ 2}=\mathrm{id}$ on $\partial^0\cD$, the Identity Principle for holomorphic maps implies that $S$ is a conformal involution in a neighborhood of~$\partial^0\cD$.
\smallskip

\begin{figure}[ht]
\begin{tikzpicture}
\node[anchor=south west,inner sep=0] at (7.2,0) {\includegraphics[width=0.45\textwidth]{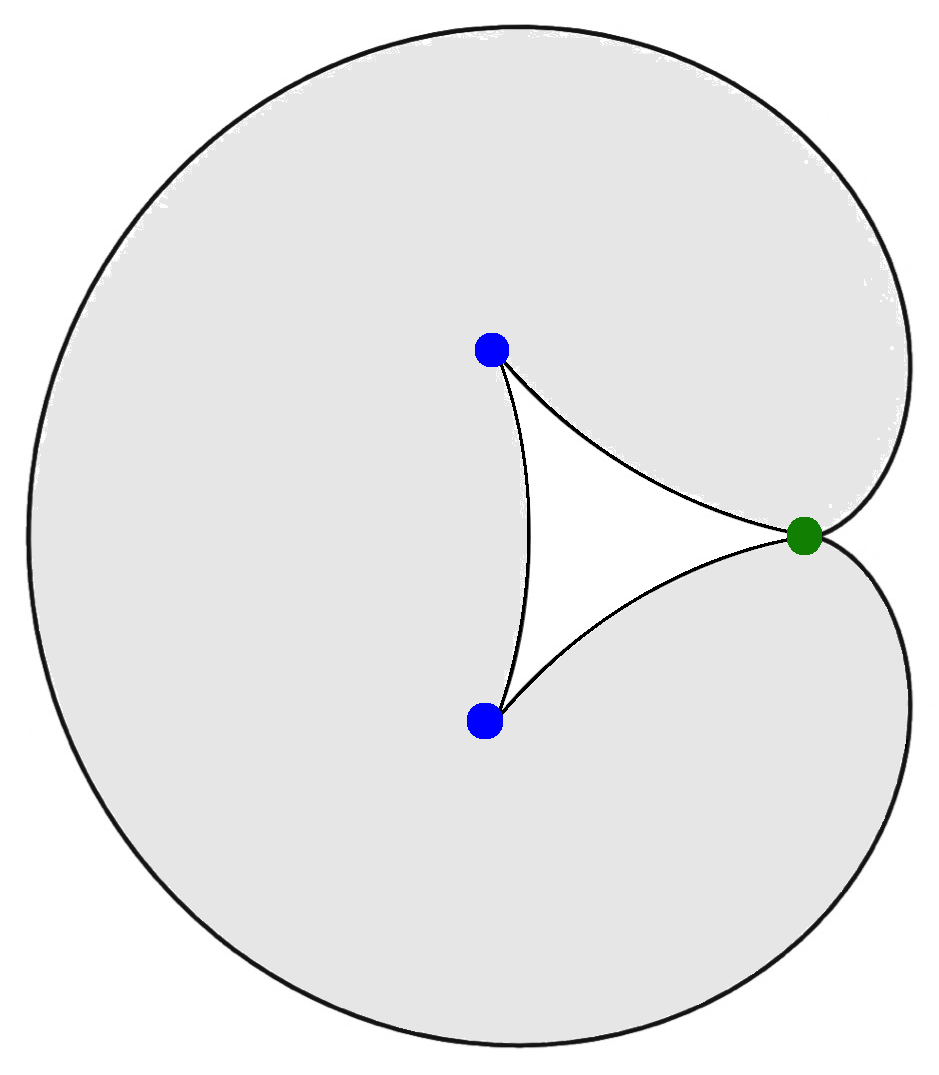}}; 
\node[anchor=south west,inner sep=0] at (0,0) {\includegraphics[width=0.54\textwidth]{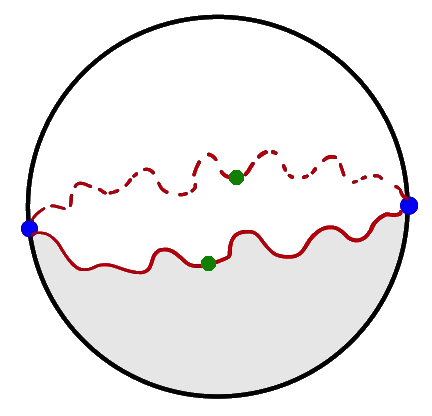}}; 
\node at (9.6,3.2) {\begin{Large}$\cD$\end{Large}};
\draw [->, line width=1pt] (8.8,2.4) to [out=225,in=-5] (3.6,1);
\draw [->, line width=1pt] (8.8,3.6) to [out=135,in=5] (3.6,5.6);
\draw [->, line width=1pt] (-0.2,3) to [out=90,in=180] (2,5);
\draw [->, line width=1pt] (-0.2,3) to [out=270,in=180] (2,1.2);
\node at (6.2,5.8) {\begin{Large}$\psi_2$\end{Large}};
\node at (6.2,0.66) {\begin{Large}$\psi_1$\end{Large}};
\node at (1.2,0.6) {\begin{Large}$\Sigma$\end{Large}};
\node at (-0.5,3) {\begin{Large}$M$\end{Large}};
\node at (4.8,2.2) {\begin{Large}$\mathfrak{J}$\end{Large}};
\node at (2.8,1.2) {\begin{Large}$\mathfrak{D}$\end{Large}};
\end{tikzpicture}
\caption{Illustrated is the proof of Lemma~\ref{b_inv_lem_1}.}
\label{inversive_domain_1_fig}
\end{figure}

\noindent\textbf{Step 2: Constructing a welded Riemann surface $\Sigma$ with $S$ as the welding map.}
We now weld the bordered Riemann surface $\cD\cup\partial^0\cD$ with itself using the conformal involution $S:\partial^0\cD\to\partial^0\cD$ as the boundary identification map. The resulting (non-compact) Riemann surface is denoted by $\Sigma$. Note that $\Sigma$ is a genus zero surface with finitely many boundary components.

By definition of $\Sigma$, there exist homeomorphisms $\psi_1,\psi_2:\cD\cup\partial^0\cD\to\Sigma$ such that
\begin{enumerate}
\item $\psi_i$, $i\in\{1,2\}$, is conformal on $\cD$,
\item $\psi_1(\cD)\cap\psi_2(\cD)=\emptyset$,
\item $\psi_1(\partial^0\cD)=\psi_2(\partial^0\cD)$ is a union of non-singular real-analytic curves (which we denote by $\mathfrak{J}$),
\item $\psi_1(w)=\psi_2(S(w))$, for $w\in\partial^0\cD$, and
\item $\Sigma=\psi_1(\cD\cup\partial^0\cD)\cup\psi_2(\cD)$.
\end{enumerate}
\smallskip

\noindent\textbf{Step 3: Defining a conformal involution $M$ on $\Sigma$.}
Let us define a map 
\begin{align*}
M:\Sigma\to\Sigma, \hspace{1cm} 
M=
\begin{cases}
\psi_2\circ\psi_1^{-1}\quad \mathrm{on}\quad \psi_1(\cD\cup\partial^0\cD),\\
\psi_1\circ\psi_2^{-1}\quad \mathrm{on}\quad \psi_2(\cD).
\end{cases}
\end{align*} 
The relations $\psi_1=\psi_2\circ S$ (on $\partial^0\cD$) and $S^{\circ 2}=\mathrm{id}$ (on $\partial^0\cD$) imply that the piecewise definitions of $M$ match continuously along $\mathfrak{J}$. Moreover, $M$ is easily seen to be holomorphic, and hence a conformal involution on $\Sigma$. We also note that $\psi_1$ conjugates the involution $S\vert_{\partial^0\cD}$ to $M\vert_{\mathfrak{J}}$. 
\smallskip

\noindent\textbf{Step 4: A finite degree branched cover $R$ and the type of $\Sigma$.} We define a holomorphic map $R:\Sigma\setminus\psi_2(S^{-1}(\mathfrak{S}))\longrightarrow\widehat{\C}\setminus\mathfrak{S}$ as follows:
\begin{align*}
R=
\begin{cases}
\psi_1^{-1} \quad \mathrm{on}\quad \psi_1(\cD\cup\partial^0\cD),\\
S\circ \psi_1^{-1}\circ M = S\circ \psi_2^{-1} \quad \mathrm{on}\quad \psi_2\left(\cD \setminus S^{-1}(\mathfrak{S})\right).
\end{cases}
\end{align*} 
The assumption that $S$ extends continuously to an involution of $\partial\cD$ fixing $\mathfrak{S}$ implies that $R$ is a proper holomorphic map. Hence, $R$ is a finite degree branched covering map. Suppose that the degree of $R$ is $d+1$, for some non-negative integer $d$.

As the punctured sphere $\widehat{\C}\setminus\mathfrak{S}$ is a finite type Riemann surface, it follows that $\Sigma\setminus\psi_2(S^{-1}(\mathfrak{S}))$ must also be a finite type Riemann surface. Hence, $\Sigma\setminus\psi_2(S^{-1}(\mathfrak{S}))$ is also a punctured sphere. This means that $S^{-1}(\mathfrak{S})$ is a finite set and the boundary components of $\Sigma$ arising from the welding procedure are points (i.e., not holes). 
\smallskip

\noindent\textbf{Step 5: Extending $R$ to a rational map and $M$ to a M{\"o}bius involution.} Since $R$ is of finite degree, it extends holomorphically to the punctures of $\Sigma\setminus\psi_2(S^{-1}(\mathfrak{S}))$. This produces a holomorphic map $R$ of the Riemann sphere, and hence $R$ is a rational map of degree $d+1$.

The fact that $\Sigma$ is a punctured sphere also implies that the conformal involution $M$ extends to the entire Riemann sphere, and yields a M{\"o}bius involution. Note that $M$ preserves the set of punctures of $\Sigma$.
After possibly a change of coordinates, we can assume that $M=\eta^+$.
\smallskip

\noindent\textbf{Step 6: Putting everything together.}
Since $\Sigma$ is a punctured sphere, the curve $\mathfrak{J}$ can also extended to a Jordan curve $\overline{\mathfrak{J}}\subset\widehat{\C}$ by throwing in the punctures. Moreover, the assumption that $S:\partial\cD\to\partial\cD$ is orientation-reversing implies that $\eta^+$ is an orientation-reversing self-homeomorphism of $\overline{\mathfrak{J}}$. Hence, $\eta^+$ must have exactly two fixed points (namely, $\pm 1$) on $\overline{\mathfrak{J}}$.

We set $\mathfrak{D}:=\psi_1(\cD)$ and note that $\mathfrak{D}$ is one of components of $\widehat{\C}\setminus\overline{\mathfrak{J}}$. Finally, the definition of $R$ and the above discussion imply that 
\begin{equation}
R\vert_{\widehat{\C}\setminus\overline{\mathfrak{D}}}=S\vert_{\cD}\circ R\vert_{\mathfrak{D}}\circ\eta^+\vert_{\widehat{\C}\setminus\overline{\mathfrak{D}}}\ \implies\ S\vert_{\cD}= R\circ\eta^+\circ (R\vert_{\mathfrak{D}})^{-1}.
\label{b_inv_formula}
\end{equation}
Finally, the degrees of the branched coverings $S:S^{-1}(\cD)\to\cD$ and $S:S^{-1}(\Int{T})\to \Int{T}$ follow directly from Expression~\eqref{b_inv_formula}.
\end{proof}

The next result gives a similar description for inversive multi-domains with two components such that the boundary components are swapped by the B-involution.

\begin{lem}\label{b_inv_lem_2}
Let $\cD=\Omega_0\sqcup\Omega_1$ be an inversive multi-domain with associated B-involution $S:\overline{\cD}\to\widehat{\C}$. 
Assume further that $S(\partial\Omega_0)=\partial\Omega_1$.
Then, there exist rational maps $R_0, R_1$ and Jordan disks $\mathfrak{D}_0, \mathfrak{D}_1$ such that the following hold for $j\in\{0,1\}$.
\begin{enumerate}\upshape
\item $\eta^+:\mathfrak{D}_j\to\widehat{\C}\setminus\overline{\mathfrak{D}_{1-j}}$ is a homeomorphism.
\item $\partial\mathfrak{D}_j$ is a piecewise non-singular real-analytic curve.
\item $R_j:\mathfrak{D}_j\to\Omega_j$ is a conformal isomorphism.
\item $S\vert_{\Omega_j}:= R_{1-j}\circ\eta^+\circ (R_j\vert_{\mathfrak{D}_j})^{-1}$.
\item $S:S^{-1}(\Omega_{1-j})\cap \Omega_j\longrightarrow \Omega_{1-j}$ is a branched covering of degree $\deg(R_{1-j})-1$.
\item $S:S^{-1}(\Int{\Omega_{1-j}^c})\cap\Omega_j\longrightarrow \Int{\Omega_{1-j}^c}$ is a branched covering of degree $\deg(R_{1-j})$.
\end{enumerate}
\end{lem}
\begin{proof}
The proof essentially follows the ideas used in the proof of Lemma~\ref{b_inv_lem_1}.

As in the case of inversive domains, the map $S$ extends to a conformal involution in a neighborhood of~$\partial^0\cD$. 
We set $\mathfrak{S}_j:=\mathfrak{S}\cap\partial\Omega_j$ and $\partial^0\Omega_j:=\partial\Omega_j\setminus\mathfrak{S}_j=\partial^0\cD\cap\partial\Omega_j$, $j\in\{0,1\}$. The map $S:\partial\Omega_j\to\partial\Omega_{1-j}$ carries $\mathfrak{S}_j$ (respectively, $\partial^0\Omega_j$) onto $\mathfrak{S}_{1-j}$ (respectively, onto $\partial\Omega^0_{1-j}$). We now weld the bordered Riemann surfaces $\Omega_0\cup\partial^0\Omega_0$ and $\Omega_1\cup\partial^0\Omega_1$ using the real-analytic homeomorphism $S:\partial^0\Omega_0\to\partial^0\Omega_1$, and call the resulting Riemann surface~$\Sigma$. 

By construction of $\Sigma$, there exist homeomorphisms $\psi_0:\Omega_0\cup\partial^0\Omega_0\to\Sigma$ and $\psi_1:\Omega_1\cup\partial^0\Omega_1\to\Sigma$ such that
\begin{enumerate}
\item $\psi_j$ is conformal on $\Omega_j$, $j\in\{0,1\}$, 
\item $\psi_0(\Omega_0)\cap\psi_1(\Omega_1)=\emptyset$,
\item $\mathfrak{J}:=\psi_0(\partial^0\Omega_0)=\psi_1(\partial^0\Omega_1)$ is a union of non-singular real-analytic curves,
\item\label{welding_cond} $\psi_0=\psi_1\circ S$ on $\partial^0\Omega_0$, and
\item $\Sigma=\psi_0(\Omega_0)\sqcup\mathfrak{J}\sqcup\psi_1(\Omega_1)$.
\end{enumerate}
We now define a holomorphic map 
$$
R_0:\Sigma\setminus\psi_1\left(\left(S\vert_{\Omega_1}\right)^{-1}(\mathfrak{S}_0)\right)\longrightarrow\widehat{\C}\setminus\mathfrak{S}_0
$$
\begin{align*}
R_0=
\begin{cases}
\psi_0^{-1} \quad \mathrm{on}\quad \psi_0(\Omega_0\cup\partial^0\Omega_0),\\
S\vert_{\Omega_1}\circ\psi_1^{-1} \quad \mathrm{on}\quad \psi_1\left(\Omega_1 \setminus \left(S\vert_{\Omega_1}\right)^{-1}(\mathfrak{S}_0)\right).
\end{cases}
\end{align*} 
Condition~\eqref{welding_cond} and the relation $S\vert_{\partial\Omega_1}=\left(S\vert_{\partial\Omega_0}\right)^{-1}$ together guarantee that the piecewise definitions of $R_0$ match continuously along the piecewise real-analytic curve $\mathfrak{J}$.
Moreover, the assumption that the continuous extension of $S$ to $\partial\cD$ fixes $\mathfrak{S}$ implies that $R_0$ is a proper holomorphic map, and hence a finite degree branched covering map. As in the proof of Lemma~\ref{b_inv_lem_1}, this implies that $\Sigma\setminus\psi_1(\left(S\vert_{\Omega_1}\right)^{-1}(\mathfrak{S}_0))$ is biholomorphic to a punctured sphere. This allows us to extend $R_0$ to a rational map of the Riemann sphere. Moreover, the closure of $\mathfrak{J}$ in the Riemann sphere (obtained by adding finitely many points to $\mathfrak{J}$) is a Jordan curve.
\begin{figure}[ht]
\begin{tikzpicture}
\node[anchor=south west,inner sep=0] at (0,0) {\includegraphics[width=0.96\textwidth]{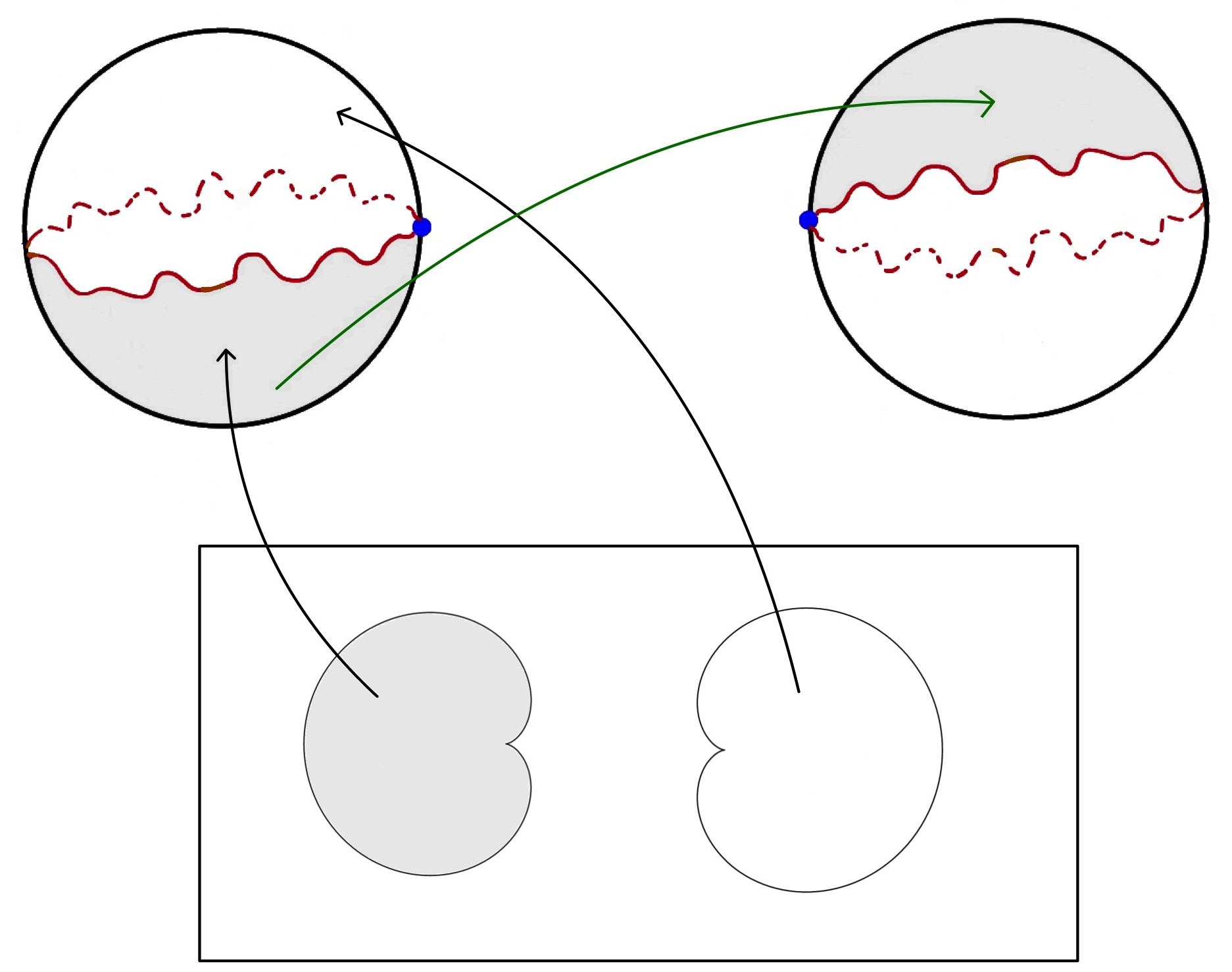}}; 
\node at (2.08,4.8) {\begin{Large}$\psi_0$\end{Large}};
\node at (7.5,5) {\begin{Large}$\psi_1$\end{Large}};
\node at (1.8,6) {\begin{Large}$\mathfrak{D}_0$\end{Large}};
\node at (10,6) {\begin{Large}$\mathfrak{D}_1$\end{Large}};
\node at (4,2) {\begin{Large}$\Omega_0$\end{Large}};
\node at (8,2) {\begin{Large}$\Omega_1$\end{Large}};
\node at (1.8,8.4) {\begin{Large}$\overline{\mathfrak{D}_0}^c$\end{Large}};
\node at (10.66,8.6) {\begin{Large}$\overline{\mathfrak{D}_1}^c$\end{Large}};
\node at (6.6,7.75) {\begin{Large}$\eta^+$\end{Large}};
\end{tikzpicture}
\caption{Illustrated is the proof of Lemma~\ref{b_inv_lem_2}.}
\label{inversive_domain_2_fig}
\end{figure}
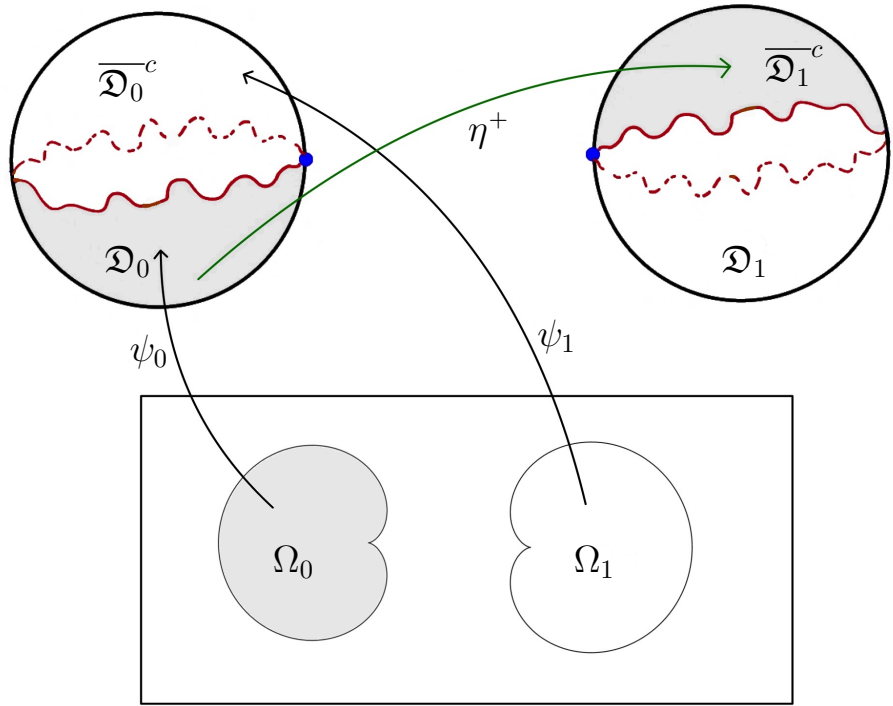

Let us now set $\mathfrak{D}_0:=\psi_0(\Omega_0)$ and $\mathfrak{D}_1:=\eta^+\left(\widehat{\C}\setminus\overline{\mathfrak{D}_0}\right)$. We further define a rational map $R_1$ as
\begin{align*}
R_1=
\begin{cases}
\psi_1^{-1}\circ\eta^+ \quad \mathrm{on}\quad \overline{\mathfrak{D}_1},\\
S\vert_{\Omega_0}\circ\psi_0^{-1}\circ\eta^+ \quad \mathrm{on}\quad \widehat{\C}\setminus \overline{\mathfrak{D}_1}.
\end{cases}
\end{align*} 
Once again, Condition~\eqref{welding_cond} shows that the piecewise definitions of $R_1$ match continuously along the Jordan curve $\overline{\mathfrak{J}}$ (where the closure is taken in $\widehat{\C}$).

By construction, each $R_j:\mathfrak{D}_j\to\Omega_j$ is a conformal isomorphism, $j~\in~\{0,1\}$. To justify that $R_0, R_1$ are the desired rational maps, it only remains to check that $S\vert_{\Omega_j}:= R_{1-j}\circ\eta^+\circ (R_j\vert_{\mathfrak{D}_j})^{-1}$, $j\in\{0,1\}$. To this end, first note that
\begin{equation*}
\begin{split}
& \hspace{1.1cm} S\vert_{\Omega_1}\circ \psi_1^{-1} = R_0 \quad \textrm{ on } \quad \widehat{\C}\setminus\overline{\mathfrak{D}_0}\\
& \implies S\vert_{\Omega_1} = R_0\circ \psi_1 = R_0\circ \eta^+\circ \eta^+\circ \psi_1\\
& \implies S\vert_{\Omega_1} = R_0\circ\eta^+\circ \left(\psi_1^{-1}\circ\eta^+\vert_{\mathfrak{D}_1}\right)^{-1} = R_0\circ\eta^+\circ \left(R_1\vert_{\mathfrak{D}_1}\right)^{-1}.
\end{split}
\end{equation*}
On the other hand,
\begin{equation*}
\begin{split}
& \hspace{1.1cm} S\vert_{\Omega_0}\circ \psi_0^{-1}\circ\eta^+ = R_1 \quad \textrm{ on } \quad  \widehat{\C}\setminus \overline{\mathfrak{D}_1}\\
& \implies S\vert_{\Omega_0} = R_1\circ \eta^+\circ\psi_0\\
& \implies S\vert_{\Omega_0} = R_1\circ\eta^+\circ \left(\psi_0^{-1}\vert_{\mathfrak{D}_0}\right)^{-1} = R_1\circ\eta^+\circ \left(R_0\vert_{\mathfrak{D}_0}\right)^{-1}.
\end{split}
\end{equation*}
Finally, the statements about the degrees of $S$ follow from the relations $S\vert_{\Omega_j}\equiv R_{1-j}\circ\eta^+\circ(R_j\vert_{\mathfrak{D}_j})^{-1}$, $j\in\{0,1\}$.
The proof is now complete.
\end{proof}

We are now ready to establish the promised algebraic description of a general B-involution.

\begin{theorem}\label{b_inv_thm}
Let $S:\overline{\cD}\to\widehat{\C}$ be a B-involution of the inversive multi-domain $\cD=\displaystyle\bigsqcup_{j=1}^k\Omega_j$.
Then, there exist Jordan domains $\mathfrak{D}_j$ and rational maps $R_j$, $j\in\{1,\cdots,k\}$, such that the following hold.
\begin{enumerate}
\item $\eta^+:\mathfrak{D}_j\to\widehat{\C}\setminus\overline{\mathfrak{D}_{\kappa(j)}}$ is a homeomorphism.
\item $\partial\mathfrak{D}_j$ is a piecewise non-singular real-analytic curve.
\item  $R_j:\mathfrak{D}_j\to\Omega_j$ is a conformal isomorphism.
\item $S\vert_{\Omega_j}\equiv R_{\kappa(j)}\circ\eta^+\circ(R_j\vert_{\mathfrak{D}_j})^{-1}$.
\end{enumerate}
Here, $\kappa$ is a permutation of the set $\{1,\cdots,k\}$ of order two.
\end{theorem}
\begin{proof}
The assumption that $S:\partial\cD\to\partial\cD$ is an involution which maps each $\partial\Omega_j$ homeomorphically onto some $\partial\Omega_{j'}$ implies that $S$ induces an order two permutation on the set $\{1,\cdots,k\}$. We denote this permutation by $\kappa$.
Since $\kappa$ is of order two, it can be factored as a product of disjoint $1$-cycles and $2$-cycles; i.e., $S$ either preserves $\partial\Omega_j$, or swaps $\partial\Omega_j$ and $\partial\Omega_{\kappa(j)}$.

If $\kappa(j)=j$, then $\Omega_j$ is an inversive domain with associated B-involution $S\vert_{\Omega_j}$. If $\kappa(j)\neq j$, then $\Omega_j\sqcup\Omega_{\kappa(j)}$ is an inversive multi-domain with associated B-involution given by the restriction of $S$. The existence of the desired Jordan domains $\mathfrak{D}_j$ and rational maps $R_j$ now follow from Lemmas~\ref{b_inv_lem_1} and~\ref{b_inv_lem_2}.
\end{proof}

\section{B-involutions with factor Bowen-Series external maps}\label{holo_para_space_reln_sec}

Throughout this section, we will work with a fixed $\Sigma\in\mathfrak{F}$. Let $m,n\in\N$ be as in Section~\ref{fbs_sec}.

\begin{defn}\label{degen_poly_like_space_def}
We define the space $\mathscr{B}_\Sigma$ to be the collection of all degenerate polynomial-like maps with connected non-escaping set that admit the factor Bowen-Series map $A_\Sigma^{\mathrm{f}}$ as their external map.
\end{defn}

{
\begin{remark}
We equip every degenerate polynomial-like map (with connected filled Julia set) with an invariant external access to its filled Julia set (or equivalently, by a fixed point on the ideal boundary of its escaping set). Two marked degenerate polynomial-like maps are identified if they are conjugate via a M{\"o}bius map that preserves the marking.      
\end{remark}
}

By definition, for a map $(S,P_1,P_2)\in\mathscr{B}_\Sigma$, there exists a conformal isomorphism 
$$
\psi:\D\to T^\infty(S):=\widehat{\C}\setminus K(S)
$$ 
that conjugates $A_\Sigma^{\mathrm{f}}$ to $S$, wherever the maps are defined. We set 
$$
\cD:=\widehat{\C}\setminus\overline{\psi(\mathcal{H})}
$$ 
and remark that $\overline{\cD}$ is a pinched polygon (note that the image of $\partial\mathcal{H}$ under the conformal isomorphism $\psi$ consists of finitely many non-singular real-analytic curves). The map $S$ extends to a continuous map on $\overline{\cD}$ which is meromorphic on $\cD$. We call $\overline{\cD}$ the \emph{canonical domain of definition} of $S$ (see Remark~\ref{canonical_dom_rem}).

We normalize $S\in\mathscr{B}_\Sigma$ so that the conformal map $\psi:\D\to T^\infty(S)$ that conjugates $A_\Sigma^{\mathrm{f}}$ to $S$ sends $0$ to $\infty$, and has asymptotics $z\mapsto 1/z+O(z)$ as $z\to 0$.

We also remark that since $S:\partial S^{-1}(\cD)\to\partial\cD$ has degree $d$, the map $S:\overline{S^{-1}(\cD)}\to\overline{\cD}$ can be extended to a degree $d$ branched cover of the sphere with a fixed critical point of multiplicity $d-1$ at $\infty$. In fact, this can be done in such a way that there is no other critical point of the branched cover outside $K(S)$. The fact that a degree $d$ branched cover of the sphere has $2d-2$ critical points (counted with multiplicities), it follows that $S$ has $d-1$ critical points (counted with multiplicities) in $K(S)$.

\subsection{Coarse lamination}\label{coarse_lami_subsec}
We now proceed to give an explicit description of the pinched polygon $\overline{\cD}$. 
Let us denote the set of $p=m/n$ ideal boundary points of $\mathcal{H}$ on $\mathbb{S}^1$ by $S_{\Sigma}$. These points are periodic of period at most two under $A_\Sigma^{\mathrm{f}}$; in fact, at most two points of $S_\Sigma$ are fixed by $A_\Sigma^{\mathrm{f}}$.
Under the circle homeomorphism $\mathfrak{h}$ that conjugates $z^d$ to $A_{\Sigma}^{\mathrm{f}}$, the set $S_{\Sigma}$ is pulled back to the set 
$$
\displaystyle\mathcal{A}:=\{\frac{i}{p}:i\in\{0,\cdots,p-1\}\}.
$$
(Here we identify $\mathbb{S}^1$ with $\T=\R/\Z$.) If $p$ is even, $A_{\Sigma}^{\mathrm{f}}$ has two fixed points in $S_\Sigma$ and these correspond to $0,1/2\in\mathcal{A}$. For odd $p$, the map $A_{\Sigma}^{\mathrm{f}}$ has a unique fixed point in $S_\Sigma$ and this fixed point corresponds to $0\in\mathcal{A}$. On the other hand, the $2-$cycles of $A_{\Sigma}^{\mathrm{f}}$ in $S_{\Sigma}$ correspond to the $2-$cycles $\pm \frac{r}{p}$ of $m_d:\R/\Z\to\R/\Z,\ \theta\mapsto d\theta$.

Following Section~\ref{partition_sec}, we can associate a \emph{coarse lamination} $\mathcal{L}$ with $S\in\mathscr{B}_\Sigma$ that records which points of $\mathcal{A}$ are identified on the boundary of $K(S)$. Then, $\overline{\cD}$ is homoemorphic to the pinched disk $\overline{\D}/\mathcal{L}$ (cf. \cite[Section~4.1]{MM2}.

As a slight abuse of terminology, we will call the components of $\D\setminus\cL$ of infinite (respectively, finite) hyperbolic area \emph{gaps} (respectively, \emph{polygons}) of $\cL$.
For a gap $\mathcal{G}$, the `boundary at infinity' $\partial \mathcal{G} \cap \T$ has length equal to $\frac{k}{p}$ for some integer $k\geq 1$.
We shall call this number $k$ the \emph{degree} of the gap, denoted by $\deg(\mathcal{G})$. 
We also denote the number of points of $\mathcal{A}$ in $\Int(\partial\mathcal{G} \cap \T)$ by $\cusp(\mathcal{G})$, and call it the {\em cusp number} of the gap.

The next result is immediate from the definition of coarse laminations.

\begin{lem}\label{inv_domain_touching_lem}
Let $S\in\mathscr{B}_\Sigma$ with canonical domain of definition $\overline{\cD}$.
The components of $\cD$ are in bijective correspondence with the gaps of $\mathcal{L}$. Two such components share at most one boundary point; moreover, they share a boundary point if and only if the corresponding gaps are cobordant on a leaf or a polygon of $\cL$.
\end{lem}

We enumerate the gaps of as $\mathcal{G}_1,\cdots,\mathcal{G}_k$ such that $\mathcal{G}_j$ corresponds to $\Omega_j$.

Note that no leaf of $\mathcal{L}$ connects the fixed points $0$ or $1/2$ to the rest of ($2$-peridoic) points of $\mathcal{A}$ (cf. \cite[\S 18]{Mil06}).
If $r/p, s/p\in\mathcal{A}$ lie in the same $\cL-$equivalence class, then so do $-r/p=m_d(r/p)$ and $-s/p=m_d(s/p)$. Hence, the gaps of the lamination $\mathcal{L}$ either intersect the real line, or come in complex conjugate pairs. It is easily seen that
\smallskip

\noindent$\bullet$ if $\mathcal{G}_j$ intersects the real line, then $S(\partial\Omega_j)=\partial\Omega_j$; otherwise,\\

\noindent$\bullet$ $\mathcal{G}_j$ and $\mathcal{G}_{j'}$ (for some $j'\in\{1,\cdots,k\}$) form a complex conjugate pair of gaps, and $S(\partial\Omega_j)=\partial\Omega_{j'}$.
\smallskip

The next result demonstrates that degenerate polynomial-like maps having factor Bowen-Series maps as their external maps are B-involutions.

\begin{prop}\label{mating_always_b_inv_prop}
Let $S\in\mathscr{B}_\Sigma$ with canonical domain of definition $\overline{\cD}$.
The map $S:\overline{\cD}\to\widehat{\C}$ is a B-involution of an inversive multi-domain. In particular, the conformal mating $S$ described in Theorem~\ref{fbs_poly_conf_mat_thm} is a B-involution of an inversive multi-domain. 
\end{prop}
\begin{proof}
We have already observed that $\overline{\cD}$ is a pinched polygon with. Let us set $\mathfrak{S}=\psi(S_\Sigma)$, where $\psi:\D\to T^\infty(S)$ is the conformal conjugacy between $A_\Sigma^{\mathrm{f}}$ to $S$. The first statement is now an immediate consequence of the facts that a factor Bowen-Series map $A_\Sigma^{\mathrm{f}}:\overline{\D}\setminus\Int{\mathcal{H}}\to\overline{\D}$ restricts to a self-homeomorphism of order two on $\partial\mathcal{H}$ and the discussion above (or from the fact that each component of $\cD$ is a Jordan domain, see Remark~\ref{jordan_rem}). The second statement is now immediate as the conformal matings of Theorem~\ref{fbs_poly_conf_mat_thm} lie in the class $\mathscr{B}_\Sigma$.
\end{proof}

\subsection{Degrees and critical points of uniformizing rational maps}\label{crit_pnt_unif_rat_map_subsec}
Throughout this subsection, we fix an $S\in\mathscr{B}_\Sigma$ with canonical domain of definition $\overline{\cD}$.
By Theorem~\ref{b_inv_thm} and Proposition~\ref{mating_always_b_inv_prop}, there exist Jordan domains $\mathfrak{D}_j$ and rational maps $R_j$, $j\in\{1,\cdots,k\}$, such that \begin{enumerate}
\item $\eta^+:\mathfrak{D}_j\to\widehat{\C}\setminus\overline{\mathfrak{D}_{\kappa(j)}}$ is a homeomorphism,
\item  $R_j:\mathfrak{D}_j\to\Omega_j$ is a conformal isomorphism, and
\item $S\vert_{\Omega_j}\equiv R_{\kappa(j)}\circ\eta^+\circ(R_j\vert_{\mathfrak{D}_j})^{-1}$,
\end{enumerate}
where $\kappa$ is a permutation of the set $\{1,\cdots,k\}$ of order two. We denote the domain of $R_j$ by $\widehat{\C}_j$.

Let $d_j:=\deg{R_j}$. The relation $S\vert_{\Omega_j}\equiv R_{\kappa(j)}\circ\eta^+\circ(R_j\vert_{\mathfrak{D}_j})^{-1}$ implies that $S:S^{-1}(\Omega_{\kappa(j)})\cap \Omega_j\longrightarrow \Omega_{\kappa(j)}$ is a branched covering of degree $d_{\kappa(j)}-1$, and
$S:S^{-1}(\Int{\Omega_{\kappa(j)}^c})\cap\Omega_j\longrightarrow \Int{\Omega_{\kappa(j)}^c}$ is a branched covering of degree $d_{\kappa(j)}$. As $S:\overline{S^{-1}(\cD)}\to\overline{\cD}$ has degree $d$, we conclude that
\begin{equation}
\sum_{j=1}^k d_j = d+1.
\label{degree_formula}
\end{equation}

\begin{lem}\label{rat_map_deg_lem}
For each $j\in\{1,\cdots,k\}$,
\noindent\begin{enumerate}\upshape
\item $\deg(\mathcal{G}_j)=\deg(\mathcal{G}_{\kappa(j)})$, and
\item $d_j=n\cdot \deg(\mathcal{G}_j)$, and hence $d_j=d_{\kappa(j)}$.
\end{enumerate}
\end{lem}
\begin{proof}
1) This follows from the structure of the gaps of $\cL$ described in Section~\ref{coarse_lami_subsec}.

2) Recall that $S:S^{-1}(\Omega_{\kappa(j)})\cap\Omega_j\longrightarrow \Omega_{\kappa(j)}$ is a branched covering of degree $d_j-1$.

Note that under the map $m_{d}$, each arc of $\T\setminus\mathcal{A}$ is wrapped onto the whole circle $(n-1)$ times and onto the complement of the closure of its complex conjugate arc one more time. Hence, $m_{d}(\partial\mathcal{G}_j\cap\T)$ covers $\partial\mathcal{G}_{\kappa(j)}\cap\T$ exactly 
$$
(n-1)\cdot\deg(\mathcal{G}_j)+(\deg(\mathcal{G}_j)-1)=n\cdot \deg(\mathcal{G}_j)-1
$$ 
times. It follows that $\partial K(S)\cap\Omega_j$ covers $\partial K(S)\cap\Omega_{\kappa(j)}$ exactly $n\cdot \deg(\mathcal{G}_j)-1$ times under the map $S$. Since $\partial K(S)$ is completely invariant under $S$, we conclude that $S:S^{-1}(\Omega_{\kappa(j)})\cap\Omega_j\longrightarrow \Omega_{\kappa(j)}$ is a degree $n\cdot \deg(\mathcal{G}_j)-1$ branched covering. Therefore, $d_j=n\cdot \deg(\mathcal{G}_j)$.
\end{proof}

To describe the structure of critical points of the rational maps $R_j$ concisely, we employ the following notation. 
Consider the disjoint union 
$$
\mathfrak{U}\ :=\ \bigsqcup_{j=1}^k \widehat{\C}_{j}\ \cong\ \widehat{\C}\times\{1,\cdots,k\},
$$
and define the maps
$$
\pmb{R}:\ \mathfrak{U}\longrightarrow \widehat{\C},\quad (z,j)\mapsto R_{j}(z),
$$
and
$$
\pmb{\eta^+}\ : \mathfrak{U}\longrightarrow \mathfrak{U},\quad (z,j)\mapsto (\eta^+(z),\kappa(j)).
$$
Note that $\pmb{R}$ is a branched covering of degree $d+1$, and $\pmb{\eta^+}$ is a homeomorphism. We also set 
$$
\mathfrak{D}:= \bigsqcup_{j=1}^k \mathfrak{D}_j\subset\mathfrak{U}.
$$
Recall the notation $\mathfrak{S}=\psi(S_\Sigma)$. We denote the set of points in $\mathfrak{S}$ that do not disconnect $\partial\cD$ by $\mathfrak{S}^{\mathrm{cusp}}$. Note that $\vert \mathfrak{S}^{\mathrm{\cusp}}\cap\partial\Omega_j\vert = \cusp(\mathcal{G}_j)$.

The following result was proved in \cite[\S 4]{MM2} in certain special cases. We include the proof for completeness (and since the current setting is more general).

\begin{prop}\label{crit_pnt_r_sharp_prop}
\noindent\begin{enumerate}
\item $\pmb{R}$ has no critical points in $\mathfrak{D}$.

\item $\pmb{R}$ has $d-1=np-2$ critical points in $\pmb{R}^{-1}(\cK)\setminus\overline{\mathfrak{D}}$. On the other hand, for $n\geq 3$, the map $\pmb{R}$ has $p$ distinct critical points, each of multiplicity $n-1$, in $\pmb{R}^{-1}(T^\infty(S))\setminus\overline{\mathfrak{D}}$, and all these critical points are mapped by $\pmb{R}$ to the same point in $T^\infty(S)$. 

\item $\mathrm{crit}(\pmb{R})\cap \partial\mathfrak{D}=(\pmb{R}\vert_{\partial\mathfrak{D}})^{-1}(\mathfrak{S}^{\mathrm{cusp}})$.
\end{enumerate}
\end{prop}
\begin{proof}
1) This follows from injectivity of $R_j$ on $\mathfrak{D}_j$.

2) Recall that $S$ has $d-1=np-2$ critical points in $K(S)$, and $p$ critical points, each of multiplicity $n-1$ (when $n\geq 3$), in $T^\infty(S)$ (coming from the $p$ critical points of $A_{\Sigma}^{\mathrm{f}}$). Moreover, $S$ maps all the $p(n-1)$ critical points in $T^\infty(S)$ to the same critical value since $A_{\Sigma}^{\mathrm{f}}$ sends all of its $p(n-1)$ critical points to the origin.

The result now follows from the relation $S\vert_{\Omega_{\kappa(j)}}\equiv R_{j}\circ\eta^+\circ(R_{\kappa(j)}\vert_{\mathfrak{D}_{\kappa(j)}})^{-1}$.

3) Since $A_{\Sigma}^{\mathrm{f}}$ does not admit an analytic continuation in a neighborhood of any point of $S_{\Sigma}$, the map $S$ does not admit an analytic continuation in a neighborhood of any point of $\mathfrak{S}$. Suppose that $x\in \mathfrak{S}^{\mathrm{cusp}}$ lies on the boundary of $\Omega_j$.
Then, the relation $S\vert_{\overline{\Omega_j}}=R_{\kappa(j)}\circ\eta^+\circ (R_j\vert_{\overline{\mathfrak{D}_j}})^{-1}$ implies that $(R_j\vert_{\overline{\mathfrak{D}_j}})^{-1}$ does not extend complex-analytically to a neighborhood of $x$. This forces $x$ to be a critical value of $R_j$ with a corresponding critical point on $\partial\mathfrak{D}_j$. Hence, $(\pmb{R}\vert_{\partial\mathfrak{D}})^{-1}(\mathfrak{S}^{\mathrm{cusp}})\subset \mathrm{crit}(\pmb{R})\cap\partial\mathfrak{D}$.

It remains to show that $\pmb{R}$ has no other critical point on $\partial\mathfrak{D}$. This will be derived from the following counting argument.
\begin{equation*}
\begin{split}
&\vert \mathrm{crit}(\pmb{R})\vert =\sum_{j=1}^k (2d_j-2)\geq (np-2)+p(n-1)+\sum_{j=1}^k \cusp(\mathcal{G}_j)\\
& \implies 2np-2k \geq 2np-p-2+ \sum_{j=1}^k \cusp(\mathcal{G}_j)\\
&\quad ( \textrm{recall that } d_j=n\cdot \deg(\mathcal{G}_j)\ \textrm{and}\  \sum_{j=1}^k\deg(\mathcal{G}_j)=p\ ) \\
&\implies p+2\geq 2k+ \sum_{j=1}^k \cusp(\mathcal{G}_j).\hspace{4cm} (\bigstar)
\end{split}
\end{equation*}

\noindent\textbf{Claim: The lamination $\cL$ contains no polygon.}
\begin{proof}[Proof of claim] If the lamination $\cL$ is empty; i.e., if $k=1$, then $\cusp(\mathcal{G}_1)$ consists of $p$ points, and the two sides of Inequality~($\bigstar$) coincide. Every time an $r-$gon is introduced in $\cL$, the number $k$ (i.e., the number of gaps) increases by $r-1$ and the number $\sum_{j=1}^k \cusp(\mathcal{G}_j)$ drops by $r$. Thus, if $r>2$, the introduction of an $r-$gon in $\cL$ would increase the right side of Inequality~($\bigstar$) by $2(r-1)-r=r-2>0$. But this would violate the inequality, proving that $\cL$ contains no polygon.
\end{proof}

Thanks to the above claim, the leaves of $\cL$ have disjoint closures in $\overline{\D}$. It is now easily seen that Inequality~($\bigstar$) is actually an equality, and hence $\vert \mathrm{crit}(\pmb{R})\cap\partial\mathfrak{D}\vert=\vert\mathfrak{S}^{\mathrm{cusp}}\vert$. We conclude that $(\pmb{R}\vert_{\partial\mathfrak{D}})^{-1}(\mathfrak{S}^{\mathrm{cusp}})= \mathrm{crit}(\pmb{R})~\cap~\partial\mathfrak{D}$.
\end{proof}

\begin{cor}\label{no_polygon_cor}
The coarse lamination $\cL$ associated with $S\in\mathscr{B}_\Sigma$ has no~polygon.
\end{cor}

\subsection{Coarse partition and topology of $\mathscr{B}_\Sigma$}\label{b_sigma_top_subsec}
As in Section~\ref{partition_sec}, the above coarse laminations lead to a coarse partition of $\mathscr{B}_{\Sigma}$:
$$
\mathscr{B}_\Sigma = \bigcup_{\mathcal{L}} \mathscr{B}_{\Sigma,\mathcal{L}},
$$
where $\mathscr{B}_{\Sigma,\mathcal{L}}$ consists of all maps in $\mathscr{B}_\Sigma$ with coarse lamination $\mathcal{L}$, and the union is taken over all possible coarse laminations. This coarse partition can be used to topologize $\mathscr{B}_\Sigma$ as in Definition~\ref{topo_def_1}. Further the proof of Theorem~\ref{thm:compact} also applies mutatis mutandis to show compactness of $\mathscr{B}_\Sigma$.

\subsection{Parameter space homeomorphisms}\label{b_sigma_poly_conn_locus_subsec}

The local dynamics of the B-involution $S$ near the singular points of $\partial\cD$ is similar to that of a Schwarz reflection near the singular points of the boundary of its domain of definition, thanks to the algebraic description of B-involutions given in Theorem~\ref{b_inv_thm}. This fact allows one to give a classification of Fatou components of $S$ (i.e., connected components of $K(S)$) and to establish the classical relations between Fatou components/non-repelling cycles and critical points. Moreover, the tessellation of $\D$ given by the Fuchsian group $\Gamma$ (uniformizing the cyclic cover $\widetilde{\Sigma}$ of $\Sigma$) and the preferred fundamental domain $\Pi$ can be used to construct external dynamical rays for $S$ (following Definition~\ref{dyn_ray_schwarz}). The proof of Proposiion~\ref{per_rays_land} can also be applied to current setting to prove landing of pre-periodic dynamical rays, and this gives rise to the notion of the \emph{rational lamination} of a B-involution $S$ (see Definition~\ref{def_preper_lami}).

One can now define the spaces $\mathscr{B}_{\Sigma, r},\ \widehat{\mathscr{B}_{\Sigma, r}},\ \mathscr{B}_{\Sigma, fr},$ and $\mathscr{B}_{\Sigma, gf}$ as in the case of polygonal (or anti-Farey) Schwarz reflections.

Let $d\geq 2$ be the degree of $A_\Sigma^{\mathrm{f}}:\mathbb{S}^1\to\mathbb{S}^1$.
\begin{theorem}
\noindent\begin{enumerate}
\item There is a dynamically natural homeomorphism 
$$
\Phi: \widehat{\mathcal{C}^+_{d, r}} \longrightarrow \widehat{\mathscr{B}_{\Sigma, r}},
$$
where `dynamically natural' means that the circle homeomorphism conjugating $A_\Sigma^{\mathrm{f}}$ to $z^d$ carries the rational lamination of a polynomial to the rational lamination of a B-involution.

\item There is a dynamically natural homeomorphism 
$$
\Phi: \mathcal{C}^+_{d, fr}  \longrightarrow \mathscr{B}_{\Sigma, fr}.
$$
Here, dynamically natural means that for any $f\in \mathcal{C}^+_{d, fr}$, the B-involution $\Phi(f)$ is the unique conformal mating of $f$ with the factor Bowen-Series map $A_\Sigma^{\mathrm{f}}$.

\item There is a dynamically natural bijection 
$$
\Phi: \mathcal{C}^+_{d, gf} \longrightarrow \mathscr{B}_{\Sigma, gf}.
$$ 
Here, dynamically natural has the same meaning as in item (2) of this theorem.
\end{enumerate}
\label{b_inv_fbs_thm}
\end{theorem}

\begin{remark}
As in the antiholomorphic setting, the bijection $\Phi: \mathcal{C}^+_{d, gf} \longrightarrow \mathscr{B}_{\Sigma, gf}$ should be discontinuous at infinitely many parameters; however, we refrain from justifying this statement here (see \cite{Ino09}).
\end{remark}

The David surgery technique used in Theorem~\ref{thm:C} works mutatis mutandis in the current setting, and yields a dynamically natural bijection between geometrically finite parameters as asserted in (3).
To construct puzzles for $B$-involutions, we need a slight modification at level $0$; namely, we consider all external rays landing at fixed and period $2$ cycles.
This is to accommodate the coarse lamination coming from the degenerate polynomial-like structure of the $B$-involution (see \S \ref{coarse_lami_subsec}).
With this modification, the proofs of (1) and (2) are the same as the proofs for Theorem \ref{thm:CL} and Theorem \ref{thm:A}.

\subsection{Correspondences associated with conformal matings}\label{corr_b_sigma_subsec}

As in Section~\ref{corr_construct_subsec}, the map $S\in\mathscr{B}_\Sigma$ can be lifted by the degree $d+1$ branched cover $\pmb{R}:\mathfrak{U}\to\widehat{\C}$ to obtain a bi-degree $d$:$d$ correspondence $\mathfrak{C}^{\circledast}$ on $\mathfrak{U}$. Specifically,  $\mathfrak{C}^{\circledast}$ is defined the algebraic equation:
\begin{equation}
\{(\mathfrak{u}_1,\mathfrak{u}_2)\in\mathfrak{U}\times\mathfrak{U}: \frac{\pmb{R}(\mathfrak{u}_2)-\pmb{R}(\pmb{\eta^+}(\mathfrak{u}_1))}{\mathfrak{u}_2-\pmb{\eta^+}(\mathfrak{u}_1)}=0\}.
\label{holo_corr_eqn}
\end{equation}
We then pass to the quotient 
$$
\mathfrak{W}\ :=\ \faktor{\mathfrak{U}}{\sim_{\mathrm{w}}},
$$
where $\sim_{\mathrm{w}}$ is the finite equivalence relation defined as
\begin{center}
For $z\in\partial\mathfrak{D}_i\subset\widehat{\C}_i$ and $w\in\partial\mathfrak{D}_j\subset\widehat{\C}_j$,
\smallskip

$(z,i)\sim_{\mathrm{w}} (w,j)\iff R_i(z)=R_j(w)$.
\end{center}
The space $\mathfrak{W}$ can be viewed as a compact, simply connected, (possibly) noded Riemann surface.

As in Section~\ref{corr_construct_subsec}, the maps $\pmb{R}, \pmb{\eta^+}$ descend to $\mathfrak{W}$, defining a bi-degree $d$:$d$ correspondence $\mathfrak{C}$ on $\mathfrak{W}$.

\begin{proof}[Proof of Theorem~\ref{all_corr_thm} (Fuchsian case)]
Let $f$ be a degree $d=d(\Sigma)$
\begin{itemize}
\item geometrically finite; or 
\item finitely renormalizable and periodically repelling
\end{itemize}
polynomial with connected Julia set. By Theorem~\ref{conf_mating_b_inv_thm}, there exists $S\in\mathscr{S}_\Sigma$ that is a conformal mating of $f$ with the factor Bowen-Series map $A_\Sigma^{\mathrm{f}}$. The above discussion allows us to associate with $S$ a bi-degree $d$:$d$ correspondence $\mathfrak{C}$ on a compact, simply connected, (possibly) noded Riemann surface $\mathfrak{W}$. The proof of \cite[Theorem~5.16]{MM2} applies verbatim to the current setting to show that $\mathfrak{C}$ is a mating of $f$ with the Fuchsian group $G$ uniformizing~$\Sigma$.
\end{proof}

\appendix

\section{Teichm{\"u}ller spaces and connectedness loci in the space of degenerate polynomial-like maps}\label{appendix}

\subsection{Hybrid classes and vertical fibers}\label{subsec:HC}

Recall that the triple $(g, P_1, P_2)$ is a \emph{degenerate (anti-)polynomial-like map} of degree $d$ if $g: P_1 \longrightarrow P_2$ is a degree $d$ holomorphic/anti-holomorphic map between two pinched polygons $P_1\subset P_2$ satisfying conditions in Definition \ref{degenerate_poly_def}.

\begin{defn}\label{hybrid_conj_def}
Let $(g_1,P_1, P_2)$ and $(g_2,Q_1,Q_2)$ be two degenerate (anti-)polynomial-like maps. We say that $g_1$ and $g_2$ are \textit{hybrid conjugate} if there exists a quasiconformal map $\Phi\colon \widehat{\C}\to \widehat{\C}$ such that:

\begin{enumerate}
	\item $\Phi$ sends the pinched points (respectively, corners) of $P_2$ to the pinched points (respectively, corners) of $Q_2$,
	\item $\Phi$ conjugates $g_1$ to $g_2$, restricted to the closure of a neighborhood of $K(g_i)$ pinched at the singular points of $\partial P_1, \partial Q_1$,
	\item $\overline \partial \Phi\equiv 0$ almost everywhere on $K(g_1)$.
\end{enumerate}
\end{defn}

Hybrid conjugacy defines an an equivalence relation on the space of degenerate (anti-)polynomial-like maps, and we call the corresponding equivalence classes \emph{hybrid classes} of degenerate (anti-)polynomial-like maps.

Recall that for a degenerate (anti-)polynomial-like map, the \emph{external map} can be defined in a similar way as external maps of polynomial-like maps, which is a piecewise real-analytic, degree $d$, covering map of $\mathbb{S}^1$.
As in the classical setting of polynomial-like maps, equivalence classes of piecewise real-analytic degree $d$ coverings of the circle under the equivalence relation induced by piecewise real-analytic circle conjugacy are called \emph{external classes} (cf. \cite[Chapter I.1]{DH2}). Two (piecewise) conformally conjugate degenerate (anti-)polynomial-like maps have the same external class. For our purposes, however, it is more convenient to follow the approach of \cite[\S 3]{Lyu99} and consider the space of degenerate (anti-)polynomial-like maps up to M{\"o}bius equivalence. Analogously, we will identify two external maps when they are conjugate by a M{\"o}bius automorphism of the disk.
Finally, we define a \emph{vertical fiber} in the space of degree $d$ degenerate (anti-)polynomial-like maps with connected Julia set as the collection of all such maps with a fixed external map.

\subsection{Total space as union of vertical fibers}

The space $\mathscr{S}_{\pmb{G}_d}\equiv\mathscr{S}_{\pmb{\cN}_d}$ introduced in Section~\ref{poly_schwarz_sec} consists of degree $d$ degenerate anti-polynomial-like maps with connected Julia set having the Nielsen map $\pmb{\cN}_d$ of the \emph{regular} ideal $(d+1)-$gon reflection group $\pmb{G}_d$ as their external map. The Teichm{\"u}ller space $\mathrm{Teich}(\pmb{G}_d)$ of the group $\pmb{G}_d$ is real $(d-2)-$dimensional. Specifically,
the space $\mathrm{Teich}(\pmb{G}_d)$ consists of M{\"o}bius conjugacy classes of discrete, faithful, strongly type-preserving representations of $\pmb{G}_d$ into $\mathrm{Aut}^{\pm}(\D)$. Any such representation $\rho$ with image $G$ is induced by a quasiconformal homeomorphism $\psi_\rho$ of $\widehat{\C}$ that preserves the disk $\D$ and carries the preferred fundamental domain $\Pi$ of $\pmb{G}_d$ to a preferred fundamental domain $\Pi(G)$ of $G$:
$$
\rho(g)=\psi_\rho\circ g\circ\psi_\rho^{-1},\ g\in\pmb{G}_d
$$
(cf. \cite[\S 3.3]{LLM1}).
Thus, the Nielsen map $\cN_G$ of $G$ associated with the fundamental domain $\Pi(G)$ is simply the conjugate of $\pmb{\cN}_d$ via $\psi_\rho$. We denote by $\mathscr{S}_G\equiv \mathscr{S}_{\cN_G}$ the space of degree $d$ degenerate anti-polynomial-like maps with connected Julia set having the Nielsen map $\cN_G$ as their external map. Note that each $\mathscr{S}_{G}$ is a vertical fiber in the space of degree $d$ degenerate anti-polynomial-like maps with connected Julia set (see Subsection~\ref{subsec:HC}).
Putting all these vertical fibers together, we get the \emph{total space}
$$
\mathscr{T}:=\bigsqcup_{G\in\mathrm{Teich}(\pmb{G}_d)}\ \mathscr{S}_{G}.
$$

On the holomorphic side, for a marked surface $\Sigma_0\in\mathfrak{F}$ with associated Fuchsian model $G_0$ (see Section~\ref{fbs_sec}), the slice $\mathscr{B}_{G_0}\equiv\mathscr{B}_{\Sigma_0}$ consists of all degree $d\equiv d(\Sigma_0)$ degenerate polynomial-like maps with connected Julia set that have the factor Bowen-Series map $A^{\mathrm{f}}_{\Sigma_0}$ as their external map. It follows from the construction of factor Bowen-Series maps that for $\Sigma\in\mathrm{Teich}(\Sigma_0)$ (or equivalently, for $G\in\mathrm{Teich}(G_0)$), the map $A_{\Sigma}^{\mathrm{f}}$ is quasiconformally conjugate to $A_{\Sigma_0}^{\mathrm{f}}$. Specifically, the quasiconformal conjugacy between $A_{\Sigma}^{\mathrm{f}}, A_{\Sigma_0}^{\mathrm{f}}$ is induced by the quasiconformal conjugacy between the Fuchsian models $G, G_0$ of $\Sigma, \Sigma_0$ (cf. \cite[\S 2.2]{MM2}). We will use the notation $\mathscr{B}_\Sigma$ and $\mathscr{B}_G$ interchangeably.
As in the antiholomorphic case, we consider the \emph{total space}
$$
\mathscr{T}:=\bigsqcup_{G\in\mathrm{Teich}(G_0)}\ \mathscr{B}_{G},
$$
which is the disjoint union of all the vertical fibers $\mathscr{B}_{G}$ where $G$ runs over the Teichm{\"u}ller space of $G_0$.

\begin{remark}
    To make things interesting, we will assume that $\mathrm{Teich}(\pmb{G}_d)$ (respectively, $\mathrm{Teich}(G_0)$) is not a singleton set. This is equivalent to saying that the ideal polygon defining $\pmb{G}_d$ has at least four sides (in the antiholomorphic case) or the surface $\Sigma_0$ has at least four marked points (in the holomorphic case). In both cases, this implies that $d>2$.
\end{remark}

\subsection{Cartesian product structure of the total space}

In what follows, the notation $\E_G$ will stand for the vertical fiber $\mathscr{S}_G$ or $\mathscr{B}_G$.
It is easy to see that maps in $\E_G$ are completely determined by their hybrid class. This can be proved using the arguments of \cite[\S 1.5]{DH2} (cf. \cite[\S 4.2]{LMM23}).

\begin{theorem}
Let $G_0$ stand for the regular ideal polygon reflection group $\pmb{G}_d$ or the Fuchsian model of a surface $\Sigma_0\in\mathfrak{F}$.
\begin{enumerate}
\item For each $S\in\E_{G_0}$, there exists an injective continuous map 
$$
\Psi_S:\left(\mathrm{Teich}(G_0),G_0\right)\hookrightarrow \left(\mathscr{T},S\right)
$$
such that $\Psi_S(G)$ is hybrid conjugate to $S$ and has $A_G^{\mathrm{f}}$ (respectively, $\cN_G$) as its external map.
\item For $G\in\mathrm{Teich}(G_0)$, the map
\begin{align*}
\Psi_G:\E_{G_0}\to\E_G \\ 
S\mapsto \Psi_S(G)
\end{align*}
is a bijection.
\item The space $\mathscr{T}$ has a Cartesian product structure; i.e., 
the map
\begin{align*}
\Psi:\mathrm{Teich}(G_0)\times\E_{G_0}\longrightarrow\mathscr{T}\\
(G,S)\mapsto \Psi_S(G)=\Psi_G(S)
\end{align*}
is a bijection.
\end{enumerate}
\label{cart_prod_thm}
\end{theorem}
\begin{proof}
1) The factor Bowen-Series maps $A_G^{\mathrm{f}}$ (respectively, the Nielsen maps $\cN_G$), $G\in\mathrm{Teich}(G_0)$, are quasiconformally conjugate to $A_{G_0}^{\mathrm{f}}$ (respectively, to $\cN_{G_0}$), and these conjugacies can be chosen to depend continuously on $\mathrm{Teich}(G_0)$. With these quasiconformal conjugacies at hand, one can employ a standard quasiconformal surgery argument to replace the external map of $S$ with $A_G^{\mathrm{f}}$ (respectively, $\cN_G$), for any given $G\in\mathrm{Teich}(G_0)$. By construction, this quasiconformal deformation of $S$ is supported on its escaping set (i.e., it is conformal on $K(S)$), and hence, any such deformed map is hybrid conjugate to $S$ and lies in the vertical fiber $\E_G$. It follows that the deformed map is unique up to affine conjugacy (preserving marking).
This produces a well-defined map $\Psi_S:\mathrm{Teich}(G_0)\hookrightarrow \mathscr{T}$.
Finally, injectivity of $\Psi_S$ is a consequence of the fact that for $G\neq G'\in\mathrm{Teich}(G_0)$, the maps $\Psi_S(G), \Psi_S(G')$ lie in distinct vertical fibers.

2) Injectivity of $\Psi_G$ follows from the definition of $\Psi_G$ and the fact that maps in the vertical fiber $\E_{G_0}$ are completely determined by their hybrid class. That the map $\Psi_G$ is surjective can be seen from the observation that the quasiconformal deformation of the previous part can be reversed. More precisely, the
external map of any map $S_1\in\E_G$ can be replaced with the external map $A_{G_0}^{\mathrm{f}}$ (respectively, $\cN_{G_0}$) producing a degenerate (anti-)polynomial-like map $S\in\E_{G_0}$ with $\Psi_G(S)=S_1$.

3) By definition of $\Psi$ and part (2) of this theorem, the map $\Psi$ induces a bijection between $\{G\}\times\E_{G_0}$ and the vertical fiber $\E_G$. Since $\mathscr{T}$ is the disjoint union of all vertical fibers $\E_{G}$ as $G$ runs over $\mathrm{Teich}(G_0)$, the result~follows.
\end{proof}

\begin{remark}
The upshot of Theorem~\ref{cart_prod_thm} can be summarized as follows. The image $\Psi_S(\mathrm{Teich}(G_0))\subset\mathscr{T}$ passes through any map $S\in\E_{G_0}$, and it is a single hybrid class $[S]$. We call this, which is defined by quasiconformal deformation of the external map, the \emph{horizontal leaf} through $S$ in $\mathscr{T}$. Such a horizontal leaf intersects each vertical fiber $\E_G$ at a unique point. This defines a bijective \emph{holonomy map} $\Psi_G$ between the vertical fibers $\E_{G_0}$ and $\E_G$. These holonomy maps piece together to yield a product structure $\mathscr{T}\cong \mathrm{Teich}(G_0)\times\E_{G_0}$.
\end{remark}

It is desirable to replace $\E_{G_0}$ with $\mathcal{C}_d^\pm$ in the product structure of $\mathscr{T}$ given in Theorem~\ref{cart_prod_thm}. However, our techniques do not allow us to deduce this result due to analytic limitations of David surgery and combinatorial rigidity issues. However, thanks to Theorems~\ref{thm:A},~\ref{thm:C}, and~\ref{b_inv_fbs_thm}, we have the following weaker substitute. Here, $\mathscr{T}_{gf}$ (respectively, $\mathscr{T}_{fr}$) denotes the collection of geometrically finite (respectively, periodically repelling and finitely renormalizable) maps in $\mathscr{T}$.

\begin{cor}\label{cart_prod_poly_cor}
There exists a natural bijection between the restricted total space $\mathscr{T}_{gf}\cup\mathscr{T}_{fr}$ and the product space $\mathrm{Teich}(G_0)\times\left(\mathcal{C}_{d,gf}^\pm\cup\mathcal{C}_{d,fr}^\pm\right)$.
\end{cor}

\subsection{A topological product structure for a restricted total space}
The Cartesian product structure of Theorem~\ref{cart_prod_thm} may not be upgradable to a topological product structure due to quasiconformally non-rigid parameters. 
On the other hand, if we restrict to the smaller space of periodically repelling, finitely renormalizable maps, we have transverse continuity and hence the desired product structure.
\begin{theorem}\label{top_prod_poly_thm}
There exists a dynamically natural homeomorphism
$$
\mathscr{T}_{fr}\cong \mathrm{Teich}(G_0) \times \mathcal{C}_{d,fr}^\pm.
$$
\end{theorem}
\begin{proof}[Sketch of proof]
For simplicity of the presentation, we will present the sketch of the proof for the holomorphic case.
Let $\Phi:\mathcal{C}_{d,fr}^+\to\mathscr{B}_{G_0,fr}$ be the homeomorphism of Theorem~\ref{b_inv_fbs_thm}. We will show that the map
\begin{align*}
\widehat{\Psi}: \mathrm{Teich}(G_0) \times \mathcal{C}_{d,fr}^+\longrightarrow \mathscr{T}_{fr}\\
(G,f)\mapsto \Psi(G,\Phi(f))
\end{align*}
is the desired homeomorphism, where $\Psi:\mathrm{Teich}(G_0)\times\mathscr{B}_{G_0}\longrightarrow\mathscr{T}$ is the bijection of Theorem~\ref{cart_prod_thm}.
For this, we only need to justify that $\Psi$ is a homeomorphism.

We proceed to demonstrate continuity of the map $\Psi$. Let $(G_n,S_n)\to (G_\infty,S_\infty)$ in $\mathrm{Teich}(G_0) \times \mathscr{B}_{G_0}$. We need to verify that the sequence of maps $\Psi(G_n,S_n)=\Psi_{S_n}(G_n)$ converges to the map $\Psi(G_\infty,S_\infty)=\Psi_{S_\infty}(G_\infty)$.
Since $\{G_n\}$ is contained in a compact set of $\mathrm{Teich}(G_0)$, it is easy to see that $\Psi(G_n,S_n)$ is pre-compact (cf. Theorem~\ref{thm:compact}.
Let $\mathcal{S}$ be a subsequential limit of $\Psi(G_n,S_n)$.
By Theorem~\ref{cart_prod_thm}, $\mathcal{S} = \Psi(G_\infty',S_\infty')$ for some $G_\infty' \in \mathrm{Teich}(G_0)$ and $S_\infty' \in \mathscr{B}_{G_0}$.
Consider the uniformization map $\phi_n$ of the tiling set $U_n$ of $\Psi(G_n,S_n)$.
Then $\phi_n$ conjugates the dynamics of the factor Bowen-Series maps $A_{G_n}^{\mathrm{f}}$ to $\Psi(G_n,S_n)|_{U_n}$.
It is easy to see that after appropriate normalization, $\phi_n$ is pre-compact.
Thus we can take a subsequential limit, and conclude that $G_\infty' = G_\infty$.
Since $\Phi(G_n,S_n)$ is hybrid conjugate to $S_n = \Phi(G_0,S_n)$, and the hybrid conjugacy is uniformly quasiconformal, after passing to a subsequence, we conclude that $\mathcal{S}$ is quasiconformally conjugate to $S_\infty = \Phi(G_0,S_\infty)$.
As $S_\infty$ is periodically repelling and finitely renormalizable,the non-escaping set $K(S_\infty)$ does not support any non-trivial quasiconformal deformation. Hence, the quasiconformal deformation space of the map $S_\infty$ agrees with the horizontal leaf (or hybrid class) $\Psi_{S_\infty}(\mathrm{Teich}(G_0))$.
It follows that $\mathcal{S} = \Psi(G_\infty,S_\infty)$.

 Continuity of the inverse $\Psi^{-1}$ can be proved using similar arguments. This completes the proof of the result.
\end{proof}

\subsection{A holomorphic motion and source of discontinuity}
We will end our discussion with a brief interpretation of the picture in terms of holomorphic motions. The account given below has roots in the analogous description of the space of quadratic-like maps given in \cite[\S 4]{Lyu99}.

\subsubsection{The holomorphic case}

We note that the domain of definition of a map in $\mathscr{T}$ is a finite tree of closed Jordan disks (a tree-like inversive multi-domain), and these Jordan domains can be uniformized by rational maps. 
One can regard  each strata
$$
\mathscr{T}_{\mathcal{L}}:=\bigsqcup_{G\in\mathrm{Teich}(G_0)} \mathscr{B}_{G,\mathcal{L}}
$$
of the total space as a subset of a certain quotient of the product space $\prod_{j=1}^k \mathrm{Rat}_{d_j}(\C)$, where $d_j$ are the degrees of the uniformizing rational maps and the quotients come from normalizations of the rational maps (cf. \cite[\S 6]{MM2}). Here, $\mathscr{B}_{G,\mathcal{L}}$ consists of maps in $\mathscr{B}_G$ with coarse lamination $\mathcal{L}$, see Subsection~\ref{coarse_lami_subsec}.

The map $\Psi$ of Theorem~\ref{cart_prod_thm} defines a `holomorphic motion' of $\mathscr{B}_{G_0}$ over the Teichm{\"u}ller space $\mathrm{Teich}(G_0)$ in the above ambient space; indeed, the maps $\Psi_S$ can be shown to be holomorphic (cf. \cite[Proposition~7.2]{MM2}). However, the dimension of $\mathscr{B}_{G_0}$ is larger than one as $d> 2$. Hence, there is no $\lambda$-lemma in this setting that guarantees transverse continuity of the holomorphic motion. Note that transverse continuity of the motion is equivalent to continuity of the holonomy maps $\Psi_G:\mathscr{B}_{G_0}\to\mathscr{B}_G$, $G\in\mathrm{Teich}(G_0)$.

We will now analyze the source of possible discontinuity of the above holomorphic motion in the transverse direction. To this end, let us first fix $G_1\in\mathrm{Teich}(G_0)$. For $G$ close to $G_1$ in $\mathrm{Teich}(G_0)$, the conjugacies between $G$ and $G_0$ can be chosen to be $K-$quasiconformal, for some $K>1$. This implies that the hybrid conjugacies between $S\in\mathscr{B}_{G_0}$ and $\Psi_S(G)$ are given by global $K$-quasiconformal homeomorphisms, for $G$ close to $G_1$ (note that $K$ does not depend on $S\in\mathscr{B}_{G_0}$). Consequently, the family of holomorphic maps 
$$
\left\{\Psi_S:\mathrm{Teich}(G_0)\to\mathscr{T}_{\mathcal{L}}\right\}_{S\in\mathscr{B}_{G_0,\mathcal{L}}}
$$
is locally bounded and hence normal (recall that $\mathscr{B}_{G_0}$ is compact).
Therefore, for a sequence $S_n\to S_\infty$ in $\mathscr{B}_{G_0,\mathcal{L}}$, the holomorphic maps $\Psi_{S_n}$ have (injective) holomorphic subsequential limits that send each $G\in\mathrm{Teich}(G_0)$ into the vertical fiber $\mathscr{B}_G$. However, although each $\Psi_{S_n}$ takes values in the hybrid class $[S_n]$, a priori, the images of the holomorphic subsequential limits are contained in the quasiconformal deformation space of the map $S_\infty$, which is possibly larger than the hybrid class $[S_\infty]$ (cf. \cite[\S II.7, Page 313]{DH2}). In other words, if the filled Julia set $K(S_\infty)$ supports non-trivial quasiconformal deformations, then the map $\Psi_{S_\infty}$ may not be the local uniform limit of $\{\Psi_{S_n}\}$. This is exactly the reason why the holonomy maps $\Psi_G$ may fail to be continuous at quasiconformally non-rigid parameters. (This is also the same phenomenon that is responsible for discontinuity in Theorem~\ref{thm:C} and Theorem~\ref{b_inv_fbs_thm}, and c.f. Theorem \ref{top_prod_poly_thm}.)

\begin{remark}
    The situation discussed above is analogous to the holomorphic motion of the Mandelbrot set over $\D$ in the space of quadratic rational maps (where one varies the multiplier of the fixed point at $\infty$ throughout $\D$, cf. \cite[Theorem~5.4]{BB09}). Since the Mandelbrot set lives in $\C$, this motion enjoys transverse continuity (in fact, quasiconformality).
\end{remark}

\subsubsection{The antiholomorphic case}

Much of the discussion of the previous subsection can be applied to Schwarz reflections after passing to the second iterate. Nevertheless, since the Teichm{\"u}ller space of the reflection group $\pmb{G}_d$ only admits a real-analytic structure (as opposed to complex structures on Teichm{\"u}ller spaces of Fuchsian groups), the motion of vertical fibers is only real-analytic in this context. 
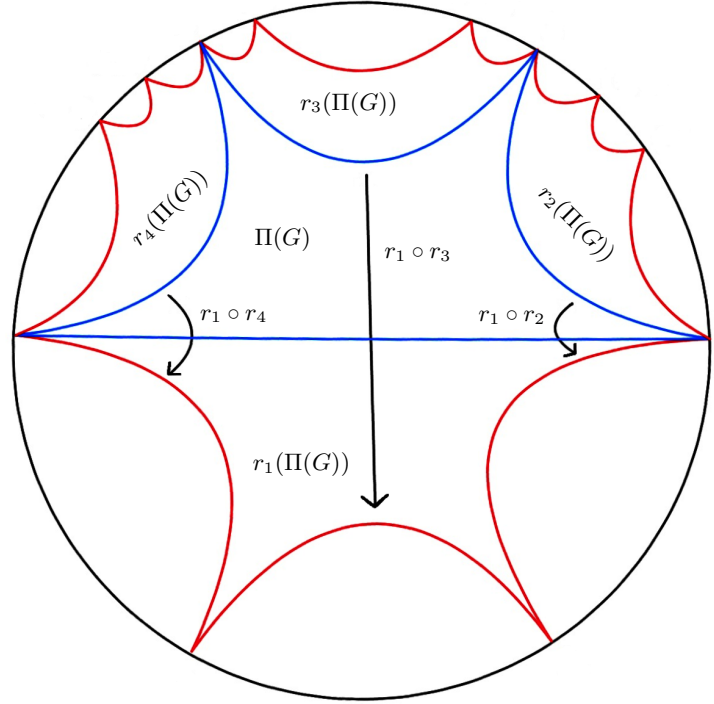
\begin{figure}[ht]
\captionsetup{width=0.98\linewidth}
\begin{tikzpicture}
\node[anchor=south west,inner sep=0] at (0,0) {\includegraphics[width=0.8\textwidth]{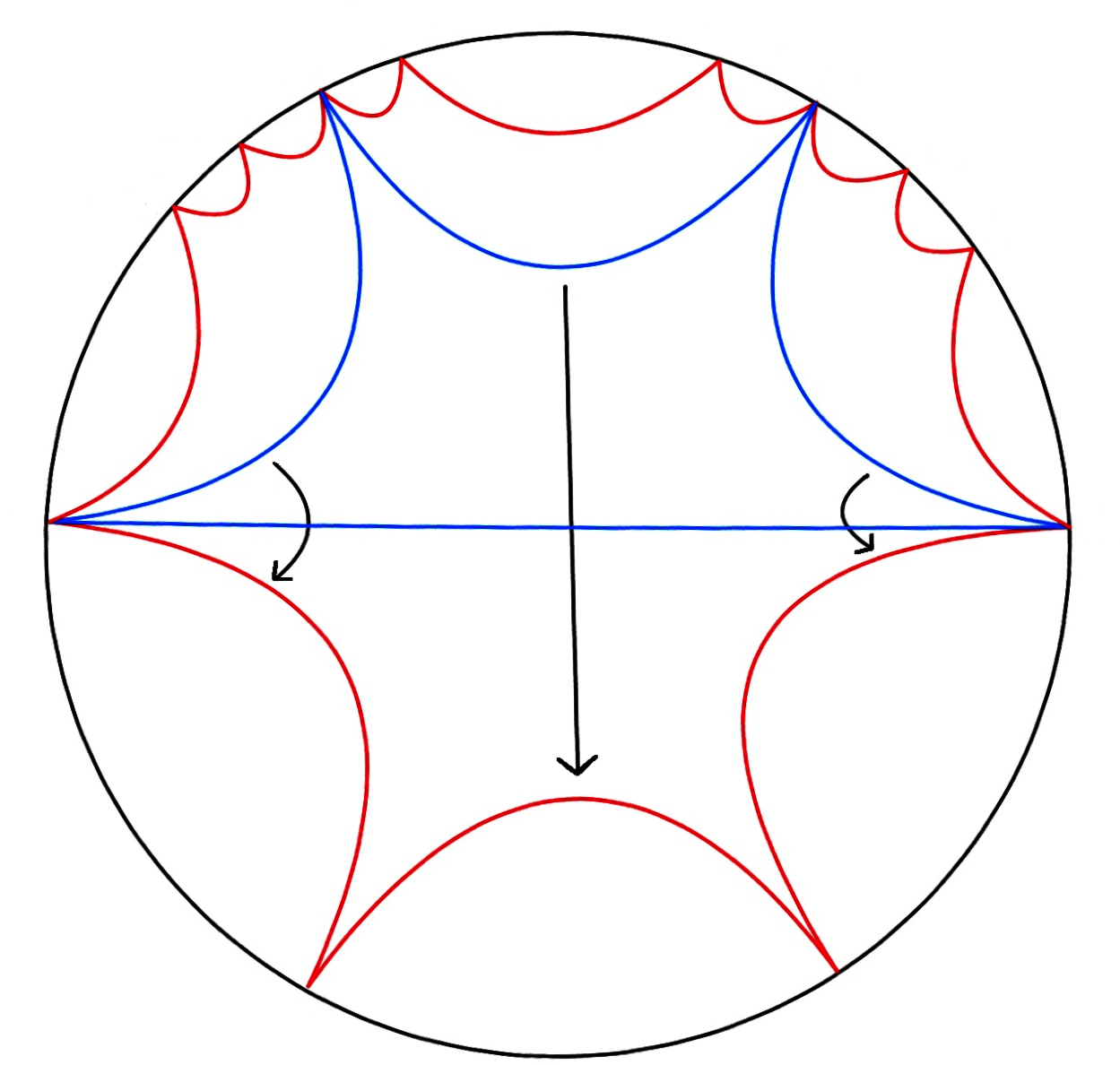}};
\node at (4,6.2) {\begin{footnotesize}$\Pi(G)$\end{footnotesize}};
\node at (4.25,3.2) {\begin{footnotesize}$r_1(\Pi(G))$\end{footnotesize}};
\node at (4.88,8) {\begin{footnotesize}$r_3(\Pi(G))$\end{footnotesize}};
\node [rotate=-45] at (7.9,6.4) {\begin{footnotesize}$r_2(\Pi(G))$\end{footnotesize}};
\node [rotate=45] at (2.5,6.6) {\begin{footnotesize}$r_4(\Pi(G))$\end{footnotesize}};
\node at (7.05,5.16) {\begin{footnotesize}$r_1\circ r_2$\end{footnotesize}};
\node at (5.8,6) {\begin{footnotesize}$r_1\circ r_3$\end{footnotesize}};
\node at (3.35,5.18) {\begin{footnotesize}$r_1\circ r_4$\end{footnotesize}};
\end{tikzpicture}
\caption{Depicted is the higher Bowen-Series map which arises as the second iterate of the Nielsen map of an ideal quadrilateral reflection group.}
\label{hbs_fig}
\end{figure}

To set up the framework, let us consider the second iterates $\cN_G^{\circ 2}$ of Nielsel maps of the reflection groups $G\in\mathrm{Teich}(\pmb{G}_d)$. Such a map is defined on the complement of the interior of an ideal $d(d+1)-$gon $\widehat{\Pi(G)}$ in $\D$. In Figure~\ref{hbs_fig} (where $d=3$), the polygon $\widehat{\Pi(G)}$ is bounded by the red geodesics, while the blue geodesics bound the original polygon $\Pi(G)$ whose boundary reflections generate $G$. If $r_1,\cdots, r_{d+1}$ are the reflections in the sides of $\Pi(G)$, then for each $j\in\{1,\cdots,d+1\}$, $$
\Xi_j:=\{r_j\circ r_k:\ k\in\{1,\cdots,d+1\},\ k\neq j\}
$$ 
is a minimal generating set of the index two Fuchsian subgroup $\widehat{G}$ of $G$. In fact, the elements of $\Xi_j$ are the side-pairing transformations for the fundamental domain $\Pi(G)\cup r_j(\Pi(G))$ of $\widehat{G}$. It is also easily seen that 
$$
\widehat{\Pi(G)}= \Pi(G)\cup\bigcup_{j=1}^{d+1}r_j(\Pi(G)),
$$
and $\D/\widehat{G}$ is a $(d+1)-$times punctured sphere that admits an antiholomorphic involution (induced by $r_j$). It can now be readily verified that the maps $\cN_G^{\circ 2}$, $G\in\mathrm{Teich}(\pmb{G}_d)$, are \emph{higher Bowen-Series maps} associated with $(d+1)-$times punctured spheres in the sense of \cite[\S 4]{MM1} (cf. \cite[\S 4]{MM3}).

According to \cite[\S 4]{MM1}, one can associate a higher Bowen-Series map to each member $\Gamma\in\mathrm{Teich}(\widehat{\pmb{G}_d})$ (where $\widehat{\pmb{G}_d}$ is the index two Fuchsian subgroup of $\pmb{G}_d$). Further, a quasiconformal conjugation between $\widehat{\pmb{G}_d}$ and $\Gamma$ induces a quasiconformal conjugation between their higher Bowen-Series maps. We denote the higher Bowen-Series map by $A_\Gamma^{h}$. We denote the vertical fiber of $A_\Gamma^{h}$ in the space of degree $d^2$ degenerate polynomial-like maps with connected Julia set by $\E_\Gamma^h$. We define the total space
$$
\mathscr{T}^h:=\bigsqcup_{\Gamma\in\mathrm{Teich}(\widehat{\pmb{G}_d})}\ \E_{\Gamma}^h.
$$
The arguments of the previous subsection now implies that the `base' vertical fiber $\E_{\widehat{\pmb{G}_d}}^h$ moves holomorphically over the Teichm{\"u}ller space of $\widehat{\pmb{G}_d}$. Here, we embed the total space $\mathscr{T}^h$ into the product space of holomorphic maps on finitely many domains in $\C$.

The operation of passing to the second iterate embeds the vertical fiber $\mathscr{S}_{\pmb{G}_d}$ (in the space of degree $d$ degenerate anti-polynomial-like maps with connected Julia set) into the vertical fiber $\E_{\widehat{\pmb{G}_d}}^h$ (in the space of degree $d^2$ degenerate polynomial-like maps with connected Julia set). Thus, this copy of $\mathscr{S}_{\pmb{G}_d}$ in $\E_{\widehat{\pmb{G}_d}}^h$ moves holomorphically over $\mathrm{Teich}(\widehat{\pmb{G}_d})$. This holomorphic motion can be used to explain the topological product structure in the antiholomorphic case of Theorem~\ref{top_prod_poly_thm}.

Finally, note that $\mathrm{Teich}(\pmb{G}_d)$ is the fixed-point set of an antiholomorphic involution on $\mathrm{Teich}(\widehat{\pmb{G}_d})$ (elements of $\mathrm{Teich}(\pmb{G}_d)$ correspond to `symmetric' $(d+1)-$times punctured spheres, cf. \cite[\S 3.3]{LLM1}), and hence $\mathrm{Teich}(\pmb{G}_d)$ can be regarded as a real slice of the complex manifold $\mathrm{Teich}(\widehat{\pmb{G}_d})$. It follows that the above copy of $\mathscr{S}_{\pmb{G}_d}$ moves real-analytically over $\mathrm{Teich}(\pmb{G}_d)$.

\end{document}